\documentclass[11pt]{amsart}
\usepackage{amsthm,amsmath,amssymb}

\usepackage{geometry}
\usepackage{graphicx}

\numberwithin{equation}{section}
\numberwithin{figure}{section}
\theoremstyle{plain}
\newtheorem{thm}{Theorem}[section]
\newtheorem{lem}[thm]{Lemma}
\newtheorem{cor}[thm]{Corollary}

\newtheorem{claim}[thm]{Claim}

\theoremstyle{remark}
\newtheorem{rmk}[thm]{Remark}

\newcommand{\M}{\operatorname{M}}
\newcommand{\Hf}{\operatorname{H}}

\newcommand{\Od}{\operatorname{O}}
\newcommand{\od}{\operatorname{\mathbf{o}}}
\newcommand{\e}{\operatorname{\mathbf{e}}}
\newcommand{\E}{\operatorname{E}}
\newcommand{\s}{\operatorname{\mathbf{s}}}

\newcommand{\Pn}{\operatorname{P}}
\newcommand{\Q}{\operatorname{Q}}
\newcommand{\K}{\operatorname{K}}
\begin{document}

\title[Lozenge tilings of a halved hexagon with an array of triangles removed]{Lozenge tilings of a halved hexagon with an array of triangles removed from the boundary}

\author{Tri Lai}
\address{Department of Mathematics, University of Nebraska -- Lincoln, Lincoln, NE 68588}
\email{tlai3@unl.edu}



\date{\today}

\dedicatory{}

\begin{abstract}
Proctor's work on staircase plane partitions yields an enumeration of lozenge tilings of a halved hexagon on the triangular lattice. Rohatgi recently extended this tiling enumeration to a halved hexagon with a triangle removed from the boundary. In this paper we prove a generalization of the results of Proctor and Rohatgi by enumerating lozenge tilings of a halved hexagon in which an array of an arbitrary number of adjacent triangles has been removed from the boundary.
\end{abstract}

\maketitle
\section{Introduction}\label{Intro}
Given $k$ positive integers $\lambda_1\geq \lambda_2\geq \dots \geq \lambda_k$, a \emph{plane partition} of shape $(\lambda_1,\lambda_2,\dots,\lambda_k)$ is an array of non-negative integers
\begin{center}
\begin{tabular}{rccccccccc}
$n_{1,1}$   &$n_{1,2}$                 &$n_{1,3}$               & $\dotsc$               &  $\dotsc$                        & $\dotsc$                            &   $n_{1,\lambda_1}$ \\\noalign{\smallskip\smallskip}
$n_{2,1}$   &  $n_{2,2}$              & $n_{2,3}$             &  $\dotsc$               & $\dotsc$                        &         $n_{2,\lambda_2}$&          \\\noalign{\smallskip\smallskip}
$\vdots$    &       $\vdots$            & $\vdots$                &        $\vdots$         &     \reflectbox{$\ddots$\quad}               &    &              \\\noalign{\smallskip\smallskip}
 $n_{k,1}$  &  $n_{k,2}$               & $n_{k,3}$              &     $\dotsc$             &   $n_{k,\lambda_k}$ &                                          &           \\\noalign{\smallskip\smallskip}
\end{tabular}
\end{center}
so that $n_{i,j}\geq n_{i,j+1}$ and $n_{i,j}\geq n_{i+1,j}$ (i.e. all rows and all columns are weakly decreasing from left to right and from top to bottom, respectively). In 1980s, R. Proctor \cite{Proc} proved a simple product formula for the number of staircase plane partitions.

\begin{thm}[Proctor \cite{Proc}]
For any non-negative integers $a,b,c$ with $a\leq b$ the number of plane partitions of the staircase shape $(b, b-1, ..., b-a+1)$ with parts no larger than $c$ is equal to \begin{equation}
\prod_{i=1}^{a}\left[\prod_{j=1}^{b-a+1}\frac{c+i+j-1}{i+j-1}\prod_{j=b-a+2}^{b-a+i}\frac{2c+i+j-1}{i+j-1}\right],
\end{equation}
where empty products are taken to be 1.
\end{thm}

Note that when $a=b$, Proctor's formula yields the number of \emph{transpose-complementary plane partitions}, one of the ten symmetry classes of plane partitions mentioned in Stanley's classical paper \cite{Stanley}.

The plane partitions in Proctor's theorem can be identified with their $3$-D interpretations --- stacks of unit cubes with certain monotonicity. The latter are in bijection with lozenge tilings of a hexagon of side-lengths $a,b,c,a,b,c$ (in the  counter clockwise order, starting from the northwestern side)  on the triangular lattice with a maximal staircase cut off, denoted by $\mathcal{P}_{a,b,c}$ (see Figure \ref{halfhex6}(a)). Here a \emph{lozenge} (or \emph{unit rhombus}) is the union of any two unit equilateral triangles sharing an edge, and a \emph{lozenge tiling} of a region on a triangular lattice is a covering of the region by lozenges so that there are no gaps or overlaps. This way, Proctor's plane partition enumeration implies the following tiling enumeration.
\begin{cor}\label{Proctiling}For any non-negative integers $a,$ $b$, and $c$ with $a\leq b$, we have
\begin{equation}
\M(\mathcal{P}_{a,b,c})=\prod_{i=1}^{a}\left[\prod_{j=1}^{b-a+1}\frac{c+i+j-1}{i+j-1}\prod_{j=b-a+2}^{b-a+i}\frac{2c+i+j-1}{i+j-1}\right],
\end{equation}
where empty products are taken to be 1. Here we use the notation $\M(\mathcal{R})$ for the number of tilings\footnote{We only consider regions on the triangular lattice in this paper. Therefore,  from now only, we use the words  ``\emph{region(s)}" and ``\emph{tiling(s)}" to mean ``\emph{ region(s) on the triangular lattice}" and ``\emph{lozenge tiling(s)}", respectively.} of the region $\mathcal{R}$.
\end{cor}

When $a=b$, the region $\mathcal{P}_{a,b,c}$ becomes a ``\emph{halved hexagon}"  $\mathcal{P}_{a,a,c}$, and  we can view $\mathcal{P}_{a,b,c}$ as a halved hexagon with a `defect'. Tiling enumeration of halved hexagons with certain defects has been investigated by a number of authors (see e.g. \cite{Ciucu1}, \cite{Cutoff}, \cite{Ranjan}, \cite{Lai}, and the lists of references therein).

\begin{figure}
  \centering
  \includegraphics[width=10cm]{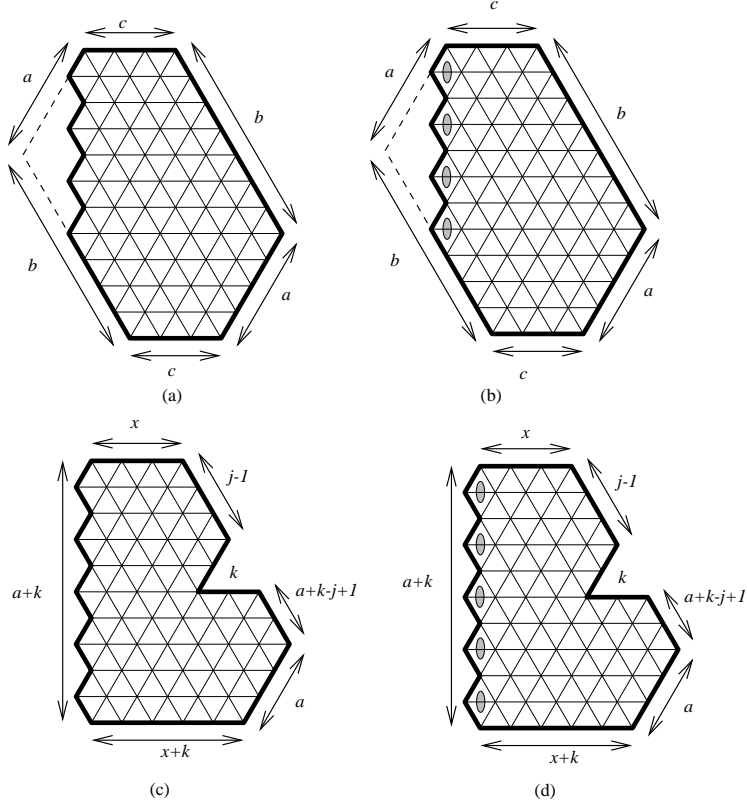}
  \caption{(a) Halved hexagon (with defect) $\mathcal{P}_{4,7,3}$. (b) The weighted halved hexagon (with defect) $\mathcal{P}'_{4,7,3}$. (c)--(d) The regions in Rohatgi's paper \cite{Ranjan}.}\label{halfhex6}
\end{figure}



Lozenges in a region can carry `weights'. In this case, we use the notation $\M(\mathcal{R})$ for the sum of weights of all tilings of $\mathcal{R}$, where the \emph{weight} of a tiling is the weight product of its constituent lozenges. We are also interested in the weighted counterpart $\mathcal{P}'_{a,b,c}$  of $\mathcal{P}_{a,b,c}$  where all the vertical lozenges along the western side are all weighted by $\frac{1}{2}$ (see the lozenges with shaded `core' in Figure \ref{halfhex6}(b)).  M. Ciucu \cite{Ciucu1} proved the following weighted counterpart of Corollary \ref{Proctiling}.

\begin{lem}\label{Ciuculem} For any non-negative integers $a,$ $b$, and $c$ with $a\leq b$
\begin{equation}
\M(\mathcal{P}'_{a,b,c})=2^{-a}\prod_{i=1}\frac{2c+b-a+i}{c+b-a+i}\prod_{i=1}^{a}\left[\prod_{j=1}^{b-a+1}\frac{c+i+j-1}{i+j-1}\prod_{j=b-a+2}^{b-a+i}\frac{2c+i+j-1}{i+j-1}\right].
\end{equation}
\end{lem}

In his generalization of  W. Jockusch and J. Propp's result on \emph{quartered Aztec diamonds} \cite{JP}, the author \cite{Lai} investigated a family of regions, called \emph{quartered hexagons}, that is   obtained from a quarter of a symmetry hexagon by removing several unit triangles from the base (see the regions in Figure \ref{halfhex3b}).  We also notice that the tiling enumeration of the quartered hexagons in \cite{Lai} is equivalent to the lattice-path enumeration of the so-called \emph{stars} by C. Krattenthaler, A. J. Guttmann, and X. G. Viennot \cite{KGV}. This result implies the tiling enumeration of the following halved hexagons with triangles removed from the base.

Assume that $\textbf{a}:=(a_1,a_2,a_3,\dotsc,a_n)$ is a sequence. We define several operations on sequences as follows:
\[\E(\textbf{a})=\sum_{\text{$i$ even}} a_i,\quad \Od(\textbf{a})=\sum_{\text{$i$ odd}} a_i,\quad\text{and } \s_k(\textbf{a})=\sum_{i=1}^{k}a_i.\]

Assume that $a,b$ are nonnegative integers and $\textbf{t}=(t_1,t_2,\dots,t_{2l})$ is a sequence of non-negative integers. Consider a trapezoidal region whose northern, northeastern, and southern sides have lengths $\Od(\textbf{t}),$ $ 2\E(\textbf{t}),$ and $\E(\textbf{t})+\Od(\textbf{t})$, respectively, and whose western side follows a vertical zigzag lattice paths with $\E(\textbf{t})$ steps.  We remove the triangles of sides $t_{2i}$'s from the base of the latter region  so that the distances between two consecutive triangles are $t_{2i-1}$'s. Denote the resulting region by $\mathcal{Q}(\textbf{t})=\mathcal{Q}(t_1,t_2,\dotsc,t_{2l})$ (see the regions in Figure \ref{halfhex3c}(a)  for the case when $t_{1}>0$ and Figure \ref{halfhex3c}(b) for the  case when $t_{1}=0$).   We also consider the weighted counterpart  $\mathcal{Q}'(\textbf{t})$ of the latter region, where the vertical lozenges on the western side are weighted by $\frac{1}{2}$ (see Figure \ref{halfhex3c}(c); the vertical lozenges with shaded cores are weighted by $\frac{1}{2}$).

We are also interested in a variation of the $\mathcal{Q}$-type regions as follows. Consider the  trapezoidal region whose northern, northeastern, and southern sides have lengths $\Od(\textbf{t}), 2\E(\textbf{t})-1, \E(\textbf{t})+\Od(\textbf{t}),$ respectively, and whose western side follows the vertical zigzag lattice path with $\E(\textbf{t})-\frac{1}{2}$ steps (i.e. the western side has $\E(\textbf{t})-1$ and $\frac{1}{2}$ `bumps'). Next, we  also remove the triangles of sides $t_{2i}$'s from the base so that the distances between two consecutive ones are $t_{2i-1}$'s. Denote by $\mathcal{K}(\textbf{t})=\mathcal{K}(t_1,t_2,\dotsc,t_{2l})$ the resulting regions (see the regions in Figure \ref{halfhex3c}(d)  for the case when $t_1>0$ and Figure \ref{halfhex3c}(e) for the  case when $t_1=0$). Similar to the case of $\mathcal{Q}'$-type regions, we also define a weighted version $\mathcal{K}'(\textbf{t})$ of the $\mathcal{K}(\textbf{t})$  by assigning to each vertical lozenge on its western side a weight $\frac{1}{2}$ (see Figure \ref{halfhex3c}(f)).

From now on, we use respectively the notations $\Pn_{a,b,c}$, $\Pn'_{a,b,c}$, $\Q(\textbf{t})$, $\Q'(\textbf{t})$, $\K(\textbf{t})$, and  $\K'(\textbf{t})$ for the numbers of tilings of the regions $\mathcal{P}_{a,b,c}$, $\mathcal{P}'_{a,b,c}$, $\mathcal{Q}(\textbf{t})$, $\mathcal{Q}'(\textbf{t})$, $\mathcal{K}(\textbf{t})$, and  $\mathcal{K}'(\textbf{t})$.

\begin{figure}\centering
\setlength{\unitlength}{3947sp}%
\begingroup\makeatletter\ifx\SetFigFont\undefined%
\gdef\SetFigFont#1#2#3#4#5{%
  \reset@font\fontsize{#1}{#2pt}%
  \fontfamily{#3}\fontseries{#4}\fontshape{#5}%
  \selectfont}%
\fi\endgroup%
\resizebox{13cm}{!}{
\begin{picture}(0,0)%
\includegraphics{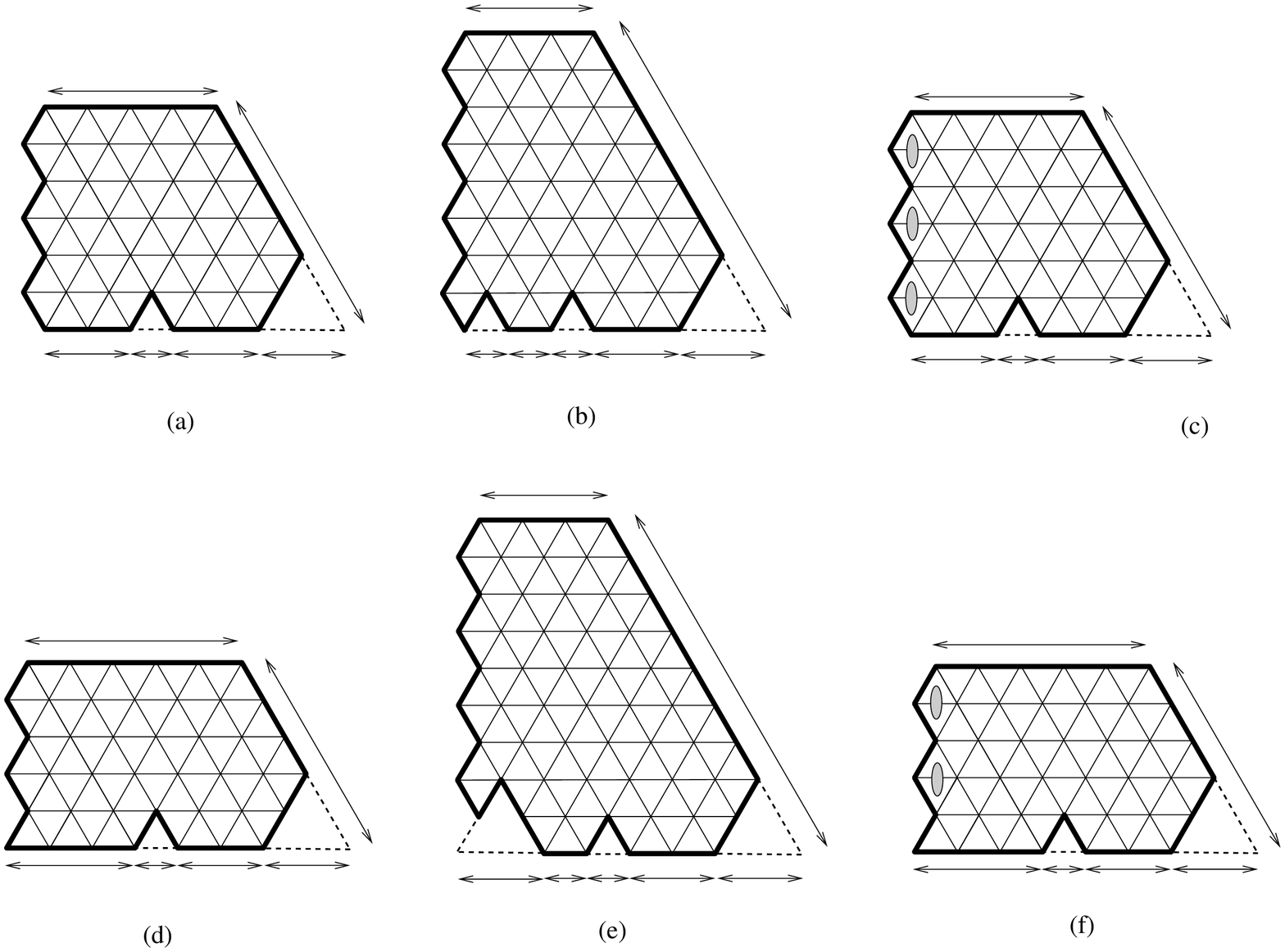}%
\end{picture}
\begin{picture}(11857,8988)(3679,-9536)
\put(4501,-1553){\makebox(0,0)[lb]{\smash{{\SetFigFont{16}{16.8}{\rmdefault}{\mddefault}{\updefault}{$t_1+t_3$}%
}}}}
\put(4148,-4338){\makebox(0,0)[lb]{\smash{{\SetFigFont{16}{16.8}{\rmdefault}{\mddefault}{\updefault}{$t_1$}%
}}}}
\put(4863,-4358){\makebox(0,0)[lb]{\smash{{\SetFigFont{16}{16.8}{\rmdefault}{\mddefault}{\updefault}{$t_2$}%
}}}}
\put(5513,-4368){\makebox(0,0)[lb]{\smash{{\SetFigFont{16}{16.8}{\rmdefault}{\mddefault}{\updefault}{$t_3$}%
}}}}
\put(8020,-4331){\makebox(0,0)[lb]{\smash{{\SetFigFont{16}{16.8}{\rmdefault}{\mddefault}{\updefault}{$t_2$}%
}}}}
\put(8387,-4331){\makebox(0,0)[lb]{\smash{{\SetFigFont{16}{16.8}{\rmdefault}{\mddefault}{\updefault}{$t_3$}%
}}}}
\put(4058,-9001){\makebox(0,0)[lb]{\smash{{\SetFigFont{16}{16.8}{\rmdefault}{\mddefault}{\updefault}{$t_1$}%
}}}}
\put(4938,-9011){\makebox(0,0)[lb]{\smash{{\SetFigFont{16}{16.8}{\rmdefault}{\mddefault}{\updefault}{$t_2$}%
}}}}
\put(5548,-9026){\makebox(0,0)[lb]{\smash{{\SetFigFont{16}{16.8}{\rmdefault}{\mddefault}{\updefault}{$t_3$}%
}}}}
\put(4583,-6566){\makebox(0,0)[lb]{\smash{{\SetFigFont{16}{16.8}{\rmdefault}{\mddefault}{\updefault}{$t_1+t_3$}%
}}}}
\put(6287,-9022){\makebox(0,0)[lb]{\smash{{\SetFigFont{16}{16.8}{\rmdefault}{\mddefault}{\updefault}{$t_4$}%
}}}}
\put(8785,-4331){\makebox(0,0)[lb]{\smash{{\SetFigFont{16}{16.8}{\rmdefault}{\mddefault}{\updefault}{$t_4$}%
}}}}
\put(9295,-4331){\makebox(0,0)[lb]{\smash{{\SetFigFont{16}{16.8}{\rmdefault}{\mddefault}{\updefault}{$t_5$}%
}}}}
\put(10045,-4331){\makebox(0,0)[lb]{\smash{{\SetFigFont{16}{16.8}{\rmdefault}{\mddefault}{\updefault}{$t_6$}%
}}}}
\put(8089,-9145){\makebox(0,0)[lb]{\smash{{\SetFigFont{16}{16.8}{\rmdefault}{\mddefault}{\updefault}{$t_2$}%
}}}}
\put(8681,-9130){\makebox(0,0)[lb]{\smash{{\SetFigFont{16}{16.8}{\rmdefault}{\mddefault}{\updefault}{$t_3$}%
}}}}
\put(9094,-9137){\makebox(0,0)[lb]{\smash{{\SetFigFont{16}{16.8}{\rmdefault}{\mddefault}{\updefault}{$t_4$}%
}}}}
\put(9611,-9130){\makebox(0,0)[lb]{\smash{{\SetFigFont{16}{16.8}{\rmdefault}{\mddefault}{\updefault}{$t_5$}%
}}}}
\put(10406,-9115){\makebox(0,0)[lb]{\smash{{\SetFigFont{16}{16.8}{\rmdefault}{\mddefault}{\updefault}{$t_6$}%
}}}}
\put(6256,-4355){\makebox(0,0)[lb]{\smash{{\SetFigFont{16}{16.8}{\rmdefault}{\mddefault}{\updefault}{$t_4$}%
}}}}
\put(6204,-2123){\rotatebox{300.0}{\makebox(0,0)[lb]{\smash{{\SetFigFont{16}{16.8}{\rmdefault}{\mddefault}{\updefault}{$2t_2+2t_4$}%
}}}}}
\put(6347,-7064){\rotatebox{300.0}{\makebox(0,0)[lb]{\smash{{\SetFigFont{16}{16.8}{\rmdefault}{\mddefault}{\updefault}{$2t_2+2t_4-1$}%
}}}}}
\put(8200,-798){\makebox(0,0)[lb]{\smash{{\SetFigFont{16}{16.8}{\rmdefault}{\mddefault}{\updefault}{$t_3+t_5$}%
}}}}
\put(8228,-5242){\makebox(0,0)[lb]{\smash{{\SetFigFont{16}{16.8}{\rmdefault}{\mddefault}{\updefault}{$t_3+t_5$}%
}}}}
\put(9685,-1481){\rotatebox{300.0}{\makebox(0,0)[lb]{\smash{{\SetFigFont{16}{16.8}{\rmdefault}{\mddefault}{\updefault}{$2t_2+2t_4+2t_6$}%
}}}}}
\put(9881,-6092){\rotatebox{300.0}{\makebox(0,0)[lb]{\smash{{\SetFigFont{16}{16.8}{\rmdefault}{\mddefault}{\updefault}{$2t_2+2t_4+2t_6-1$}%
}}}}}
\put(12392,-1604){\makebox(0,0)[lb]{\smash{{\SetFigFont{16}{16.8}{\rmdefault}{\mddefault}{\updefault}{$t_1+t_3$}%
}}}}
\put(12039,-4389){\makebox(0,0)[lb]{\smash{{\SetFigFont{16}{16.8}{\rmdefault}{\mddefault}{\updefault}{$t_1$}%
}}}}
\put(12754,-4409){\makebox(0,0)[lb]{\smash{{\SetFigFont{16}{16.8}{\rmdefault}{\mddefault}{\updefault}{$t_2$}%
}}}}
\put(13404,-4419){\makebox(0,0)[lb]{\smash{{\SetFigFont{16}{16.8}{\rmdefault}{\mddefault}{\updefault}{$t_3$}%
}}}}
\put(14147,-4406){\makebox(0,0)[lb]{\smash{{\SetFigFont{16}{16.8}{\rmdefault}{\mddefault}{\updefault}{$t_4$}%
}}}}
\put(14095,-2174){\rotatebox{300.0}{\makebox(0,0)[lb]{\smash{{\SetFigFont{16}{16.8}{\rmdefault}{\mddefault}{\updefault}{$2t_2+2t_4$}%
}}}}}
\put(12326,-9036){\makebox(0,0)[lb]{\smash{{\SetFigFont{16}{16.8}{\rmdefault}{\mddefault}{\updefault}{$t_1$}%
}}}}
\put(13206,-9046){\makebox(0,0)[lb]{\smash{{\SetFigFont{16}{16.8}{\rmdefault}{\mddefault}{\updefault}{$t_2$}%
}}}}
\put(13816,-9061){\makebox(0,0)[lb]{\smash{{\SetFigFont{16}{16.8}{\rmdefault}{\mddefault}{\updefault}{$t_3$}%
}}}}
\put(12851,-6601){\makebox(0,0)[lb]{\smash{{\SetFigFont{16}{16.8}{\rmdefault}{\mddefault}{\updefault}{$t_1+t_3$}%
}}}}
\put(14555,-9057){\makebox(0,0)[lb]{\smash{{\SetFigFont{16}{16.8}{\rmdefault}{\mddefault}{\updefault}{$t_4$}%
}}}}
\put(14615,-7099){\rotatebox{300.0}{\makebox(0,0)[lb]{\smash{{\SetFigFont{16}{16.8}{\rmdefault}{\mddefault}{\updefault}{$2t_2+2t_4-1$}%
}}}}}
\end{picture}}
\caption{ (a) The region $\mathcal{Q}(2,1,2,2)$. (b) The region $\mathcal{Q}(0,1,1,1,2,2)$.  (c) The region $\mathcal{Q}'(2,1,2,2)$. (d) The region $\mathcal{K}(3,1,2,2)$. (e) The region $\mathcal{K}(0,2,1,1,2,2)$. (f) The region $\mathcal{K}'(3,1,2,2)$. The lozenges with shaded cores are weighted by $\frac{1}{2}$.}\label{halfhex3c}
\end{figure}

\medskip

We define the  \emph{hyperfactorial} $\Hf(n)$ by
\begin{equation*}\Hf(n):=0!\cdot1!\cdot2!\dotsc(n-1)!,
\end{equation*}
 and the \emph{`skipping' hyperfactorial } $\Hf_2(n)$ by
\begin{equation*}
\Hf_2(n)=
\begin{cases}
 0!\cdot2!\cdot4!\dots(n-2)! &\text{if $n$ is even;}\\
1!\cdot2!\cdot 3!\dots(n-2)! &\text{if $n$ is odd.}
 \end{cases}
 \end{equation*}

\begin{lem}\label{QAR}
For any sequence of non-negative integers $\textbf{t}=(t_1,t_2,\dotsc,t_{2l})$
\begin{align}\label{QARa}
\Q(\textbf{t})&=\dfrac{\prod_{i=1}^{l}\frac{(\s_{2i}(\textbf{t}))!}{(\s_{2i-1}(\textbf{t}))!}}{\Hf_2(2\E(\textbf{t})+1)} \prod_{i=1}^{l}\frac{\Hf_2(2\s_{2i}(\textbf{t})+1)\Hf(2\s_{2i-1}(\textbf{t})+2)}{\Hf_2(2\s_{2i-1}(\textbf{t})+3)}\notag\\
                  &\times \displaystyle {\prod_{\substack{1\leq i< j\leq 2l\\
                  \text{$j-i$ odd}}}}\dfrac{\Hf(\s_j(\textbf{t})-\s_{i}(\textbf{t}))}{\Hf(\s_j(\textbf{t})+\s_{i}(\textbf{t})+1)}\displaystyle {\prod_{\substack{1\leq i<j\leq 2l\\
                  \text{$j-i$ even }}}}\dfrac{\Hf(\s_j(\textbf{t})+\s_{i}(\textbf{t})+1)}{\Hf(\s_j(\textbf{t})-\s_{i}(\textbf{t}))},
\end{align}
\begin{align}\label{QARb}
\Q'(\textbf{t})&=\dfrac{2^{-\E(\textbf{t})}}{\Hf_2(2\E(\textbf{t})+1)}  \prod_{i=1}^{l}\frac{\Hf_2(2\s_{2i}(\textbf{t})+1)\Hf(2\s_{2i-1}(\textbf{t}))}{\Hf_2(2\s_{2i-1}(\textbf{t})+1)}\notag\\
                  & \times \displaystyle {\prod_{\substack{1\leq i< j\leq 2l\\
                  \text{$j-i$ odd}}}}\dfrac{\Hf(\s_j(\textbf{t})-\s_{i}(\textbf{t}))}{\Hf(\s_j(\textbf{t})+\s_{i}(\textbf{t}))}\displaystyle {\prod_{\substack{1\leq i<j\leq 2l\\
                  \text{$j-i$ even }}}}\dfrac{\Hf(\s_j(\textbf{t})+\s_{i}(\textbf{t}))}{\Hf(\s_j(\textbf{t})-\s_{i}(\textbf{t}))},
\end{align}
\begin{align}\label{QARc}
\K(\textbf{t})&=\dfrac{1}{\Hf_2(2\E(\textbf{t}))}   \prod_{i=1}^{l}\frac{\Hf_2(2\s_{2i}(\textbf{t}))\Hf(2\s_{2i-1}(\textbf{t})+1)}{\Hf_2(2\s_{2i-1}(\textbf{t})+2)}\notag\\
                  &\times   \displaystyle {\prod_{\substack{1\leq i< j\leq 2l\\
                  \text{$j-i$ odd}}}}\dfrac{\Hf(\s_j(\textbf{t})-\s_{i}(\textbf{t}))}{\Hf(\s_j(\textbf{t})+\s_{i}(\textbf{t}))}\displaystyle {\prod_{\substack{1\leq i<j\leq 2l\\
                  \text{$j-i$ even }}}}\dfrac{\Hf(\s_j(\textbf{t})+\s_{i}(\textbf{t}))}{\Hf(\s_j(\textbf{t})-\s_{i}(\textbf{t}))},
\end{align}
and
\begin{align}\label{QARd}
\K'(\textbf{t})&=\dfrac{1}{\Hf_2(2\E(\textbf{t}))} \prod_{i=1}^{l}\frac{\Hf_2(2\s_{2i}(\textbf{t})-1)\Hf(2\s_{2i-1}(\textbf{t}))}{\Hf_2(2\s_{2i-1}(\textbf{t})+1)}\notag\\
                  &\times   \displaystyle {\prod_{\substack{1\leq i< j\leq 2l\\
                  \text{$j-i$ odd}}}}\dfrac{\Hf(\s_j(\textbf{t})-\s_{i}(\textbf{t}))}{\Hf(\s_j(\textbf{t})+\s_{i}(\textbf{t})-1)}\displaystyle {\prod_{\substack{1\leq i<j\leq 2l\\
                  \text{$j-i$ even }}}}\dfrac{\Hf(\s_j(\textbf{t})+\s_{i}(\textbf{t})-1)}{\Hf(\s_j(\textbf{t})-\s_{i}(\textbf{t}))}.
\end{align}

\end{lem}
The proof of Lemma \ref{QAR} will be given later in Section \ref{Quarter}. The tiling formulas in our main theorem will be written in terms of the above enumerations $\Q(\textbf{t})$, $\Q'(\textbf{t})$, $\K(\textbf{t})$, and  $\K'(\textbf{t})$.

\medskip

Recently, R. Rohatgi \cite{Ranjan} extended Proctor's enumeration in Corollary \ref{Proctiling} and Ciucu's result in Lemma \ref{Ciuculem} by enumerating  tilings of a halved hexagon with a triangle missing on the northeastern side (see Figures \ref{halfhex6}(c) and (d)).  In this paper, we generalize  much further Rohagi's work  by investigating a halved hexagon in which  an array of an \emph{arbitrary} number of alternating up-pointing and down-pointing triangles\footnote{The latter array of triangles was called a ``\emph{fern}" in \cite{Ciucu2}.} has been removed from the northeastern side (see Figure \ref{halfhex13}). One of the nice aspects of our main result is that the number of removed triangles is \emph{arbitrary}, not just a particular number as in previous results. We will show that the number of tilings of this new region is always given by a simple product formula in the theorem below.

Let $x,y,z$ be non-negative integers, and let $\textbf{a}=(a_1,a_2,\dotsc,a_n)$ be  a sequence of positive integers ($\textbf{a}$ may be empty). We consider a halved hexagon whose the northern, northeastern, southeastern, and southern sides have  respectively lengths $x+\E(\textbf{a}), y+z+2\Od(\textbf{a})-1, y+z+\E(\textbf{a})-1, x+\Od(\textbf{a})$, and whose western side follows a vertical lattice path with  $y+z+\Od(\textbf{a})-1$ steps.  We remove at the level $z$ above the rightmost vertex  of the halved hexagon an array of $n$ alternating triangles of side-lengths $a_1,a_2,\dotsc,a_n$ (from right to left). Denote by $\mathcal{R}_{x,y,z}(\textbf{a})=\mathcal{R}_{x,y,z}(a_1,a_2,\dotsc,a_n)$ the resulting region. When the sequence $\textbf{a}$ is empty, our $\mathcal{R}$ type region becomes the halved hexagon $\mathcal{P}_{y+z-1,y+z-1,x}$; and when $n=1$, we get Rohatgi's region in \cite{Ranjan}.

We define the \emph{Pochhammer symbol} $(x)_n$ by
\begin{equation}
(x)_n=
\begin{cases}
x(x+1)(x+2)\dotsc(x+n-1) & \text{if $n>0$;}\\
\quad\quad\quad\quad\quad\;1 &\text{if $n=0$;}\\
\dfrac{1}{(x-1)(x-2)\dotsc(x+n)} &\text{if $n<0$.}
\end{cases}
\end{equation}

\begin{figure}\centering
\setlength{\unitlength}{3947sp}%
\begingroup\makeatletter\ifx\SetFigFont\undefined%
\gdef\SetFigFont#1#2#3#4#5{%
  \reset@font\fontsize{#1}{#2pt}%
  \fontfamily{#3}\fontseries{#4}\fontshape{#5}%
  \selectfont}%
\fi\endgroup%
\resizebox{12cm}{!}{
\begin{picture}(0,0)%
\includegraphics{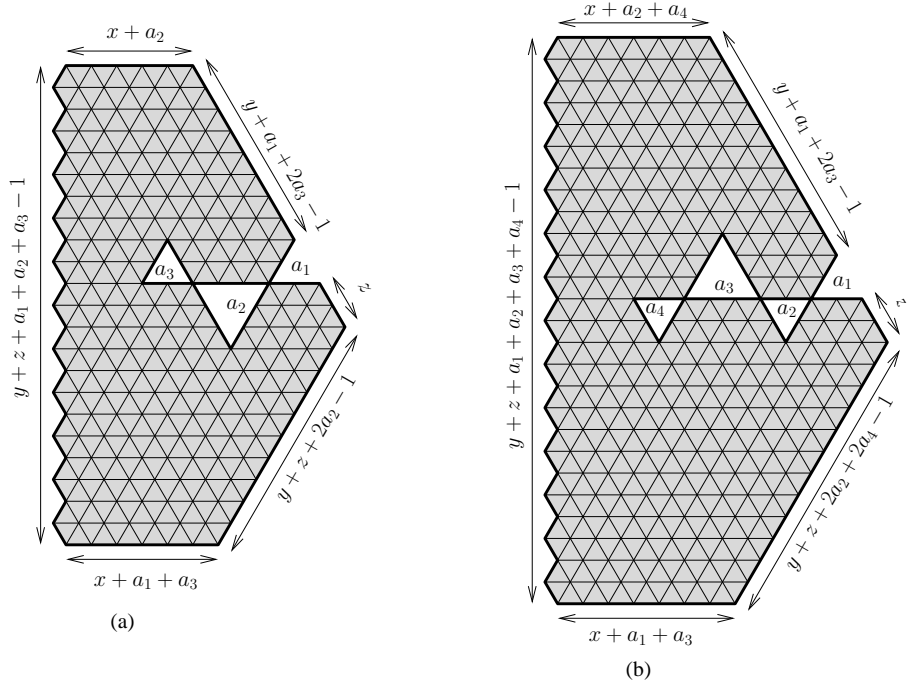}%
\end{picture}%

\begin{picture}(13959,10649)(1761,-9926)
\put(3256, 74){\makebox(0,0)[lb]{\smash{{\SetFigFont{20}{16.8}{\rmdefault}{\mddefault}{\itdefault}{$x+a_2$}%
}}}}
\put(5371,-931){\rotatebox{300.0}{\makebox(0,0)[lb]{\smash{{\SetFigFont{20}{16.8}{\rmdefault}{\mddefault}{\itdefault}{$y+a_1+2a_3-1$}%
}}}}}
\put(6091,-3541){\makebox(0,0)[lb]{\smash{{\SetFigFont{20}{16.8}{\rmdefault}{\mddefault}{\itdefault}{$a_1$}%
}}}}
\put(5100,-4126){\makebox(0,0)[lb]{\smash{{\SetFigFont{20}{16.8}{\rmdefault}{\mddefault}{\itdefault}{$a_2$}%
}}}}
\put(4000,-3601){\makebox(0,0)[lb]{\smash{{\SetFigFont{20}{16.8}{\rmdefault}{\mddefault}{\itdefault}{$a_3$}%
}}}}
\put(7111,-3751){\rotatebox{300.0}{\makebox(0,0)[lb]{\smash{{\SetFigFont{20}{16.8}{\rmdefault}{\mddefault}{\itdefault}{$z$}%
}}}}}
\put(6001,-7111){\rotatebox{60.0}{\makebox(0,0)[lb]{\smash{{\SetFigFont{20}{16.8}{\rmdefault}{\mddefault}{\itdefault}{$y+z+2a_2-1$}%
}}}}}
\put(3061,-8416){\makebox(0,0)[lb]{\smash{{\SetFigFont{20}{16.8}{\rmdefault}{\mddefault}{\itdefault}{$x+a_1+a_3$}%
}}}}
\put(1996,-5506){\rotatebox{90.0}{\makebox(0,0)[lb]{\smash{{\SetFigFont{20}{16.8}{\rmdefault}{\mddefault}{\itdefault}{$y+z+a_1+a_2+a_3-1$}%
}}}}}
\put(9601,-6286){\rotatebox{90.0}{\makebox(0,0)[lb]{\smash{{\SetFigFont{20}{16.8}{\rmdefault}{\mddefault}{\itdefault}{$y+z+a_1+a_2+a_3+a_4-1$}%
}}}}}
\put(10666,-9271){\makebox(0,0)[lb]{\smash{{\SetFigFont{20}{16.8}{\rmdefault}{\mddefault}{\itdefault}{$x+a_1+a_3$}%
}}}}
\put(13831,-8176){\rotatebox{60.0}{\makebox(0,0)[lb]{\smash{{\SetFigFont{20}{16.8}{\rmdefault}{\mddefault}{\itdefault}{$y+z+2a_2+2a_4-1$}%
}}}}}
\put(13533,-661){\rotatebox{300.0}{\makebox(0,0)[lb]{\smash{{\SetFigFont{20}{16.8}{\rmdefault}{\mddefault}{\itdefault}{$y+a_1+2a_3-1$}%
}}}}}
\put(14416,-3766){\makebox(0,0)[lb]{\smash{{\SetFigFont{20}{16.8}{\rmdefault}{\mddefault}{\itdefault}{$a_1$}%
}}}}
\put(13561,-4201){\makebox(0,0)[lb]{\smash{{\SetFigFont{20}{16.8}{\rmdefault}{\mddefault}{\itdefault}{$a_2$}%
}}}}
\put(12600,-3826){\makebox(0,0)[lb]{\smash{{\SetFigFont{20}{16.8}{\rmdefault}{\mddefault}{\itdefault}{$a_3$}%
}}}}
\put(11536,-4186){\makebox(0,0)[lb]{\smash{{\SetFigFont{20}{16.8}{\rmdefault}{\mddefault}{\itdefault}{$a_4$}%
}}}}
\put(10591,434){\makebox(0,0)[lb]{\smash{{\SetFigFont{20}{16.8}{\rmdefault}{\mddefault}{\itdefault}{$x+a_2+a_4$}%
}}}}
\end{picture}}
\caption{Halved hexagons with an array of triangles removed from the boundary: (a) $\mathcal{R}_{2,3,2}(2,3,2)$  and (b)$\mathcal{R}_{2,3,2}(2,2,3,2)$.} \label{halfhex13}
\end{figure}

\begin{thm}\label{main1} Let  $x,y,z$ be non-negative integers, and $\textbf{a}=(a_1,a_2,\dotsc,a_{n})$ a sequence of positive integers, such that $y+2\Od(\textbf{a})\geq a_1+1$.
 Then for odd $y$
\begin{align}\label{ee}
\M&(\mathcal{R}_{x,y,z}(a_1,a_2,\dotsc,a_{n}))=\prod_{i=1}^{\frac{y-1}{2}}(2x+2i)_{2\s_n(\textbf{a})+2y+2z-4i+1}\notag\\
&\frac{1}{2^{y-1}}\frac{\Hf(\s_n(\textbf{a})+y+z-1)\Hf_2(y)\Hf_2(2\E(\textbf{a})+2z+1)\Hf_2(2\Od(\textbf{a})+1)\Hf_2(2\s_n(\textbf{a})+y+2z)}
{\Hf(\s_n(\textbf{a})+z)\Hf_2(2\E(\textbf{a})+y+2z)\Hf_2(2\Od(\textbf{a})+y)\Hf_2(2\s_n(\textbf{a})+2y+2z-1)}\\
&\Q\left(x+\frac{y-1}{2}+a_{2\lfloor\frac{n+1}{2}\rfloor},a_{2\lfloor\frac{n+1}{2}\rfloor-1},\dotsc,a_1\right)\Q\left(x+\frac{y-1}{2}+a_{2\lfloor\frac{n}{2}\rfloor+1},\dotsc,a_1,z\right),\notag
\end{align}
and for even $y$
\begin{align}\label{oe}
\M&(\mathcal{R}_{x,y,z}(a_1,a_2,\dotsc,a_{n}))=\prod_{i=1}^{\frac{y}{2}}(2x+2i)_{2\s_n(\textbf{a})+2y+2z-4i+1}\notag\\
&\times\frac{\Hf(\s_n(\textbf{a})+y+z)\Hf_2(y)\Hf_2(2\E(\textbf{a})+2z)\Hf_2(2\Od(\textbf{a}))\Hf_2(2\s_n(\textbf{a})+y+2z)}{\Hf(\s_n(\textbf{a})+z)\Hf_2(2\E(\textbf{a})+y+2z)\Hf_2(2\Od(\textbf{a})+y)\Hf_2(2\s_n(\textbf{a})+2y+2z)}\\
&\K\left(x+\frac{y}{2}+a_{2\lfloor\frac{n+1}{2}\rfloor},a_{2\lfloor\frac{n+1}{2}\rfloor-1},\dotsc,a_1\right)\K\left(x+\frac{y}{2}+a_{2\lfloor\frac{n}{2}\rfloor+1},a_{2\lfloor\frac{n}{2}\rfloor},\dotsc,a_1,z\right),
\notag
\end{align}
where $a_{k}:=0$ when $k>n$, and where empty products are taken to $1$ by convention.
\end{thm}

\begin{figure}\centering
\setlength{\unitlength}{3947sp}%
\begingroup\makeatletter\ifx\SetFigFont\undefined%
\gdef\SetFigFont#1#2#3#4#5{%
  \reset@font\fontsize{#1}{#2pt}%
  \fontfamily{#3}\fontseries{#4}\fontshape{#5}%
  \selectfont}%
\fi\endgroup%
\resizebox{12cm}{!}{
\begin{picture}(0,0)%
\includegraphics{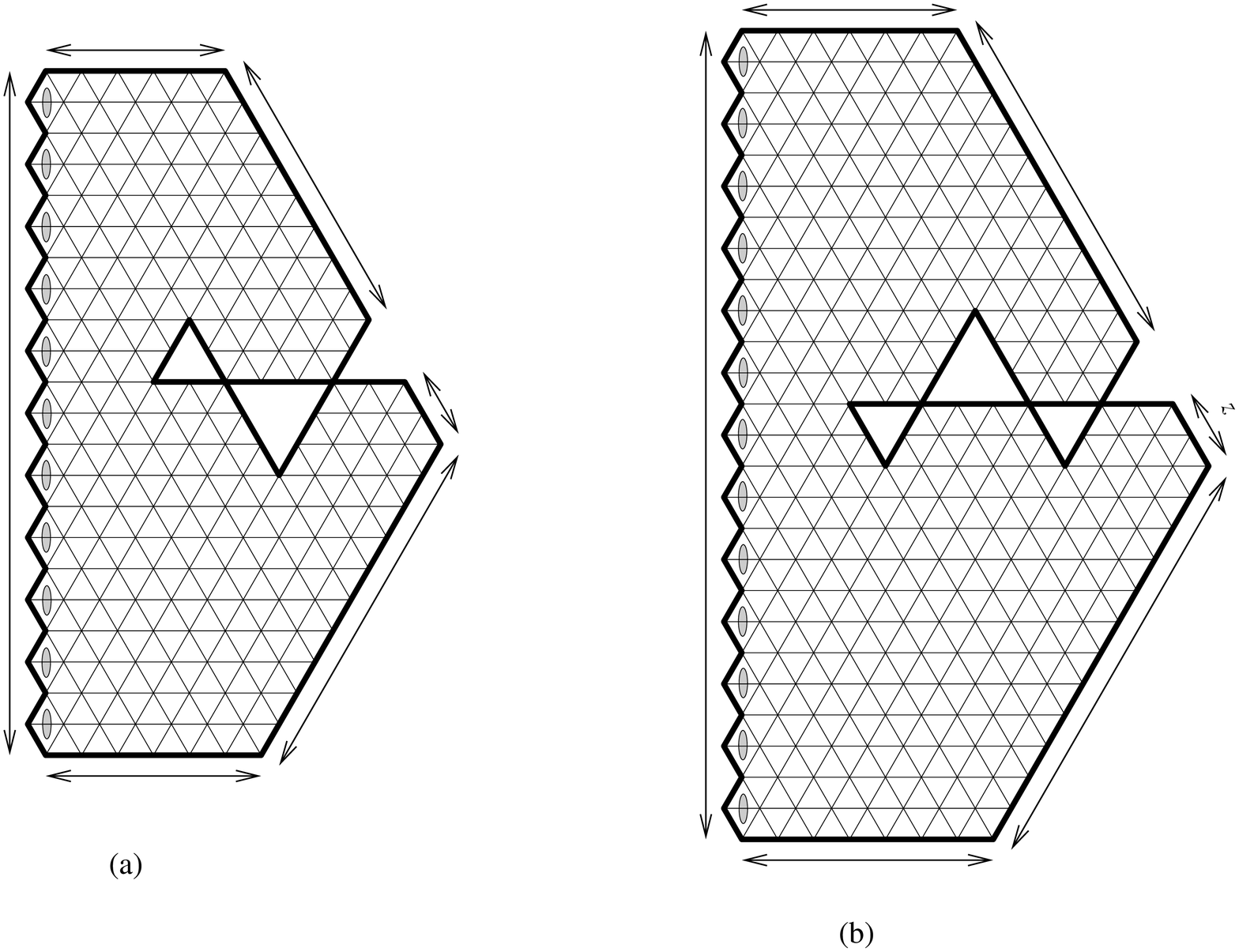}%
\end{picture}%
\begin{picture}(13963,10653)(1759,-9927)
\put(3256, 74){\makebox(0,0)[lb]{\smash{{\SetFigFont{20}{16.8}{\rmdefault}{\mddefault}{\itdefault}{$x+a_2$}%
}}}}
\put(5371,-931){\rotatebox{300.0}{\makebox(0,0)[lb]{\smash{{\SetFigFont{20}{16.8}{\rmdefault}{\mddefault}{\itdefault}{$y+a_1+2a_3-1$}%
}}}}}
\put(6091,-3541){\makebox(0,0)[lb]{\smash{{\SetFigFont{20}{16.8}{\rmdefault}{\mddefault}{\itdefault}{$a_1$}%
}}}}
\put(4921,-4126){\makebox(0,0)[lb]{\smash{{\SetFigFont{20}{16.8}{\rmdefault}{\mddefault}{\itdefault}{$a_2$}%
}}}}
\put(3991,-3601){\makebox(0,0)[lb]{\smash{{\SetFigFont{20}{16.8}{\rmdefault}{\mddefault}{\itdefault}{$a_3$}%
}}}}
\put(7111,-3751){\rotatebox{300.0}{\makebox(0,0)[lb]{\smash{{\SetFigFont{20}{16.8}{\rmdefault}{\mddefault}{\itdefault}{$z$}%
}}}}}
\put(6001,-7111){\rotatebox{60.0}{\makebox(0,0)[lb]{\smash{{\SetFigFont{20}{16.8}{\rmdefault}{\mddefault}{\itdefault}{$y+z+2a_2-1$}%
}}}}}
\put(3061,-8416){\makebox(0,0)[lb]{\smash{{\SetFigFont{20}{16.8}{\rmdefault}{\mddefault}{\itdefault}{$x+a_1+a_3$}%
}}}}
\put(1996,-5506){\rotatebox{90.0}{\makebox(0,0)[lb]{\smash{{\SetFigFont{20}{16.8}{\rmdefault}{\mddefault}{\itdefault}{$y+z+a_1+a_2+a_3-1$}%
}}}}}
\put(9601,-6286){\rotatebox{90.0}{\makebox(0,0)[lb]{\smash{{\SetFigFont{20}{16.8}{\rmdefault}{\mddefault}{\itdefault}{$y+z+a_1+a_2+a_3+a_4-1$}%
}}}}}
\put(10666,-9271){\makebox(0,0)[lb]{\smash{{\SetFigFont{20}{16.8}{\rmdefault}{\mddefault}{\itdefault}{$x+a_1+a_3$}%
}}}}
\put(13831,-8176){\rotatebox{60.0}{\makebox(0,0)[lb]{\smash{{\SetFigFont{20}{16.8}{\rmdefault}{\mddefault}{\itdefault}{$y+z+2a_2+2a_4-1$}%
}}}}}
\put(13533,-661){\rotatebox{300.0}{\makebox(0,0)[lb]{\smash{{\SetFigFont{20}{16.8}{\rmdefault}{\mddefault}{\itdefault}{$y+a_1+2a_3-1$}%
}}}}}
\put(14416,-3766){\makebox(0,0)[lb]{\smash{{\SetFigFont{20}{16.8}{\rmdefault}{\mddefault}{\itdefault}{$a_1$}%
}}}}
\put(13561,-4201){\makebox(0,0)[lb]{\smash{{\SetFigFont{20}{16.8}{\rmdefault}{\mddefault}{\itdefault}{$a_2$}%
}}}}
\put(12436,-3826){\makebox(0,0)[lb]{\smash{{\SetFigFont{20}{16.8}{\rmdefault}{\mddefault}{\itdefault}{$a_3$}%
}}}}
\put(11536,-4186){\makebox(0,0)[lb]{\smash{{\SetFigFont{20}{16.8}{\rmdefault}{\mddefault}{\itdefault}{$a_4$}%
}}}}
\put(10591,434){\makebox(0,0)[lb]{\smash{{\SetFigFont{20}{16.8}{\rmdefault}{\mddefault}{\itdefault}{$x+a_2+a_4$}%
}}}}
\end{picture}}
\caption{Weighted halved hexagons with an array of triangles removed from the boundary: (a) $\mathcal{R}'_{2,3,2}(2,3,2)$  and (b) $\mathcal{R}'_{2,3,2}(2,2,3,2)$. The lozenges with shaded cores are weighted by $\frac{1}{2}$.} \label{halfhex13b}
\end{figure}

Similar to the case of the weighted halved hexagon $\mathcal{P}'_{a,b,c}$, we  are also interested in the weighted version $\mathcal{R}'_{x,y,z}(\textbf{a})=\mathcal{R}'_{x,y,z}(a_1,a_2,\dotsc,a_n)$ of the above $\mathcal{R}$-type region where each vertical lozenge along the western side is  weighted by $\frac{1}{2}$ (see Figure \ref{halfhex13b} for examples). We also have a simple product formula for the (weighted) tiling number of $\mathcal{R}'_{x,y,z}(\textbf{a})$.
\begin{thm}\label{main2}  Let  $x,y,z$ be non-negative integers, and $\textbf{a}=(a_1,a_2,\dotsc,a_{n})$ a sequence of positive integers, such that $y+2\Od(\textbf{a})\geq a_1+1$.
Then  for odd $y$
\begin{align}\label{eeb}
\M&(\mathcal{R}'_{x,y,z}(a_1,a_2,\dotsc,a_{n}))=\prod_{i=1}^{\frac{y-1}{2}}(2x+2i-1)_{2\s_n(\textbf{a})+2y+2z-4i+1}\notag\\
&\frac{1}{2^{y-1}}\frac{\Hf(\s_n(\textbf{a})+y+z-1)\Hf_2(y)\Hf_2(2\E(\textbf{a})+2z+1)\Hf_2(2\Od(\textbf{a})+1)\Hf_2(2\s_n(\textbf{a})+y+2z)}
{\Hf(\s_n(\textbf{a})+z)\Hf_2(2\E(\textbf{a})+y+2z)\Hf_2(2\Od(\textbf{a})+y)\Hf_2(2\s_n(\textbf{a})+2y+2z-1)}\\
&\Q'\left(x+\frac{y-1}{2}+a_{2\lfloor\frac{n+1}{2}\rfloor},a_{2\lfloor\frac{n+1}{2}\rfloor-1},\dotsc,a_1\right)\Q'\left(x+\frac{y-1}{2}+a_{2\lfloor\frac{n}{2}\rfloor+1},a_{2\lfloor\frac{n}{2}\rfloor},\dotsc,a_1,z\right)\notag;
\end{align}
and for even $y$
\begin{align}\label{oob}
\M&(\mathcal{R}'_{x,y,z}(a_1,a_2,\dotsc,a_{n}))=\prod_{i=1}^{\frac{y}{2}}(2x+2i-1)_{2\s_n(\textbf{a})+2y+2z-4i+1}\notag\\
&\times\frac{\Hf(\s_n(\textbf{a})+y+z)\Hf_2(y)\Hf_2(2\E(\textbf{a})+2z)\Hf_2(2\Od(\textbf{a}))\Hf_2(2\s_n(\textbf{a})+y+2z)}
{\Hf(\s_n(\textbf{a})+z)\Hf_2(2\E(\textbf{a})+y+2z)\Hf_2(2\Od(\textbf{a})+y)\Hf_2(2\s_n(\textbf{a})+2y+2z)}\\
&\K'\left(x+\frac{y}{2}+a_{2\lfloor\frac{n+1}{2}\rfloor},a_{2\lfloor\frac{n+1}{2}\rfloor-1},\dotsc,a_1\right)\K'\left(x+\frac{y}{2}+a_{2\lfloor\frac{n}{2}\rfloor+1},a_{2\lfloor\frac{n}{2}\rfloor},\dotsc,a_1,z\right)
\notag,
\end{align}
where $a_{k}:=0$ when $k>n$, and where empty products are taken to $1$ by convention.
\end{thm}

It is easy to see that when $n=0$, our region becomes the halved hexagon $\mathcal{P}_{y+z-1,y+z-1,x}$, and when $n=1$ it becomes Rohatgi's region \cite{Ranjan}. It means that our main theorems can be considered as a common generalization of the enumeration of transpose-complementary plane partitions and Rohatgi's result.

In addition to the above `natural' reduction, the extreme case of $x= 0$ is also interesting, since it gives a reduction to a product of the tiling numbers of two quartered hexagons in \cite{Lai} (see Figure  Figure \ref{halfhex4}), i.e.,
for odd $y$
\begin{equation}
\M(\mathcal{R}_{0,y,z}(a_1,a_2,\dotsc,a_{2l}))=\M\left(\mathcal{Q}\left( 0,\frac{y-1}{2},a_{2l},\dotsc,a_1\right)\right)\M\left(\mathcal{Q}\left(0,\frac{y-1}{2}+a_{2l},a_{2l-1},\dotsc,a_1,z\right)\right),
\end{equation}
\begin{equation}
\M(\mathcal{R}_{0,y,z}(a_1,a_2,\dotsc,a_{2l+1}))=\M\left(\mathcal{Q}\left( 0,\frac{y-1}{2}+a_{2l+1},\dotsc,a_1\right)\right)\M\left(\mathcal{Q}\left(0,\frac{y-1}{2},a_{2l+1},\dotsc,a_1,z\right)\right),
\end{equation}
and for even $y$
\begin{equation}
\M(\mathcal{R}_{0,y,z}(a_1,a_2,\dotsc,a_{2l}))=\M\left(\mathcal{K}\left( 0,\frac{y}{2},a_{2l},\dotsc,a_1\right)\right)\M\left(\mathcal{K}\left(0,\frac{y}{2}+a_{2l},a_{2l-1},\dotsc,a_1,z\right)\right),
\end{equation}
\begin{equation}
\M(\mathcal{R}_{0,y,z}(a_1,a_2,\dotsc,a_{2l+1}))=\M\left(\mathcal{K}\left( 0,\frac{y}{2}+a_{2l+1},\dotsc,a_1\right)\right)\M\left(\mathcal{K}\left(0,\frac{y}{2},a_{2l+1},\dotsc,a_1,z\right)\right).
\end{equation}
The proofs of these identities will be shown in the proofs Section \ref{Proofmain}. We have similar identity for the case of $\mathcal{R}'$-type region.

\medskip

By applying a factorization theorem by Ciucu \cite{Ciucu3}, our main results also imply the number of tilings of a symmetric hexagon with two arrays of triangles removed from two sides as follows.

\begin{figure}
  \centering
  \setlength{\unitlength}{3947sp}%
\begingroup\makeatletter\ifx\SetFigFont\undefined%
\gdef\SetFigFont#1#2#3#4#5{%
  \reset@font\fontsize{#1}{#2pt}%
  \fontfamily{#3}\fontseries{#4}\fontshape{#5}%
  \selectfont}%
\fi\endgroup%
\resizebox{7cm}{!}{
\begin{picture}(0,0)%
\includegraphics{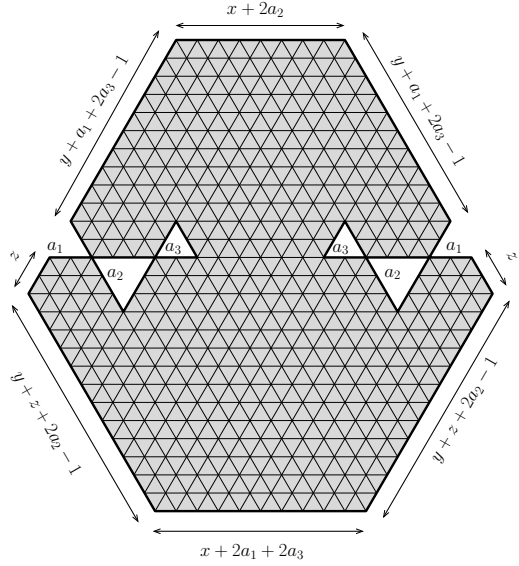}%
\end{picture}%
\begin{picture}(9771,10460)(5335,-9871)
\put(5641,-4381){\rotatebox{60.0}{\makebox(0,0)[lb]{\smash{{\SetFigFont{20}{24.0}{\rmdefault}{\mddefault}{\itdefault}{$z$}%
}}}}}
\put(6571,-2596){\rotatebox{60.0}{\makebox(0,0)[lb]{\smash{{\SetFigFont{20}{24.0}{\rmdefault}{\mddefault}{\itdefault}{$y+a_1+2a_3-1$}%
}}}}}
\put(9541,239){\makebox(0,0)[lb]{\smash{{\SetFigFont{20}{24.0}{\rmdefault}{\mddefault}{\itdefault}{$x+2a_2$}%
}}}}
\put(13456,-8116){\rotatebox{60.0}{\makebox(0,0)[lb]{\smash{{\SetFigFont{20}{24.0}{\rmdefault}{\mddefault}{\itdefault}{$y+z+2a_2-1$}%
}}}}}
\put(14701,-4231){\rotatebox{300.0}{\makebox(0,0)[lb]{\smash{{\SetFigFont{20}{24.0}{\rmdefault}{\mddefault}{\itdefault}{$z$}%
}}}}}
\put(12611,-601){\rotatebox{300.0}{\makebox(0,0)[lb]{\smash{{\SetFigFont{20}{24.0}{\rmdefault}{\mddefault}{\itdefault}{$y+a_1+2a_3-1$}%
}}}}}
\put(9046,-9856){\makebox(0,0)[lb]{\smash{{\SetFigFont{20}{24.0}{\rmdefault}{\mddefault}{\itdefault}{$x+2a_1+2a_3$}%
}}}}
\put(6219,-4156){\makebox(0,0)[lb]{\smash{{\SetFigFont{20}{24.0}{\rmdefault}{\mddefault}{\itdefault}{$a_1$}%
}}}}
\put(7329,-4666){\makebox(0,0)[lb]{\smash{{\SetFigFont{20}{24.0}{\rmdefault}{\mddefault}{\itdefault}{$a_2$}%
}}}}
\put(8401,-4171){\makebox(0,0)[lb]{\smash{{\SetFigFont{20}{24.0}{\rmdefault}{\mddefault}{\itdefault}{$a_3$}%
}}}}
\put(11491,-4186){\makebox(0,0)[lb]{\smash{{\SetFigFont{20}{24.0}{\rmdefault}{\mddefault}{\itdefault}{$a_3$}%
}}}}
\put(12451,-4651){\makebox(0,0)[lb]{\smash{{\SetFigFont{20}{24.0}{\rmdefault}{\mddefault}{\itdefault}{$a_2$}%
}}}}
\put(13591,-4156){\makebox(0,0)[lb]{\smash{{\SetFigFont{20}{24.0}{\rmdefault}{\mddefault}{\itdefault}{$a_1$}%
}}}}
\put(5591,-6481){\rotatebox{300.0}{\makebox(0,0)[lb]{\smash{{\SetFigFont{20}{24.0}{\rmdefault}{\mddefault}{\itdefault}{$y+z+2a_2-1$}%
}}}}}
\end{picture}}
  \caption{A symmetric hexagon with two arrays of triangles removed.}\label{halfhex14}
\end{figure}

Assume that $x,y,z$ are nonnegative integers, and $\textbf{a}=(a_1,a_2,\dotsc,a_n)$ is a sequence of positive integers. Consider a symmetric hexagon of side-lengths $y+z+2\Od(\textbf{a})-1,x+2\E(\textbf{a}), y+z+\Od(\textbf{a})-1, y+z+2\E(\textbf{a})-1,  x+2\E(\textbf{a}),  y+z+2\Od(\textbf{a})-1$ (in the clockwise order, starting from the northwestern side). We remove two arrays consisting of adjacent triangles of sides $a_1,a_2,\dotsc,a_n$ at the level $z$ above the western vertex of the hexagon; one array goes from left to right and another one goes from right to left. Denote by $\mathcal{F}=\mathcal{F}_{x,y,z}(a_1,a_2,\dotsc,a_n)$ the resulting region (see Figure \ref{halfhex14}).  Applying Ciucu's Factorization Theorem (Theorem 1.2 in \cite{Ciucu3}) to the  region $\mathcal{F}$, we split $\mathcal{F}$ into two subregions $\mathcal{F}^+$ and $\mathcal{F}^-$  along its vertical symmetry axis (see Figure \ref{halfhex15}(a)  for the case when $x$ is even, and Figure \ref{halfhex15}(b) for the case when $x$ is odd; the vertical lozenges with shaded cores are weighted by $\frac{1}{2}$). After removing forced lozenges, the subregions $\mathcal{F}^+$ and $\mathcal{F}^-$ become a $\mathcal{R}$- and a $\mathcal{R}'$-type regions, respectively, and we have:

\begin{figure}
  \centering
  \setlength{\unitlength}{3947sp}%
\begingroup\makeatletter\ifx\SetFigFont\undefined%
\gdef\SetFigFont#1#2#3#4#5{%
  \reset@font\fontsize{#1}{#2pt}%
  \fontfamily{#3}\fontseries{#4}\fontshape{#5}%
  \selectfont}%
\fi\endgroup%
\resizebox{12cm}{!}{
\begin{picture}(0,0)%
\includegraphics{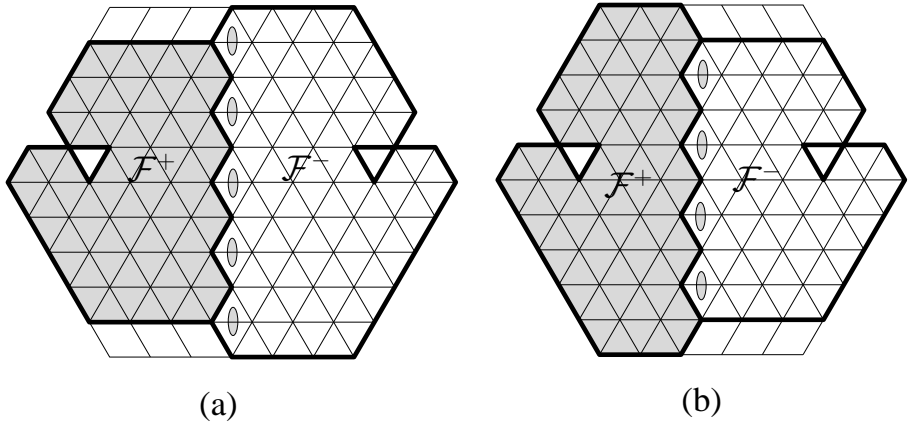}%
\end{picture}%
%
%

\begin{picture}(8674,4113)(1308,-7055)
\put(7059,-4838){\makebox(0,0)[lb]{\smash{{\SetFigFont{20}{24.0}{\rmdefault}{\mddefault}{\itdefault}{$\mathcal{F}^+$}%
}}}}
\put(8285,-4786){\makebox(0,0)[lb]{\smash{{\SetFigFont{20}{24.0}{\rmdefault}{\mddefault}{\itdefault}{$\mathcal{F}^-$}%
}}}}
\put(3971,-4674){\makebox(0,0)[lb]{\smash{{\SetFigFont{20}{24.0}{\rmdefault}{\mddefault}{\itdefault}{$\mathcal{F}^-$}%
}}}}
\put(2493,-4689){\makebox(0,0)[lb]{\smash{{\SetFigFont{20}{24.0}{\rmdefault}{\mddefault}{\itdefault}{$\mathcal{F}^+$}%
}}}}
\end{picture}}
  \caption{Applying Ciucu's Factorization Theorem to a symmetric hexagon with two arrays of triangles removed.}\label{halfhex15}
\end{figure}

\begin{cor}\label{coro1} For non-negative integers $x,y,z$ and a sequence of positive integers $\textbf{a}=(a_1,a_2,\dotsc,a_n)$
\begin{align}
\M(\mathcal{F}_{x,y,z}&(a_1,a_2,\dotsc,a_n))=2^{y+z+a_1+a_2+\dotsc+a_n-1}\notag\\
&\times\begin{cases} \M(\mathcal{R}_{\frac{x}{2},y-1,z}(a_1,a_2,\dotsc,a_n))\M(\mathcal{R}'_{\frac{x}{2},y,z}(a_1,a_2,\dotsc,a_n)) & \text{if $x$ is even};\\
 \M(\mathcal{R}_{\frac{x-1}{2},y,z}(a_1,a_2,\dotsc,a_n))\M(\mathcal{R}'_{\frac{x+1}{2},y-1,z}(a_1,a_2,\dotsc,a_n)) & \text{if $x$ is odd}.\end{cases}\end{align}
\end{cor}

\medskip

In \cite{Ranjan}, Rohatgi revealed the following beautiful  identities for the tiling numbers of $\mathcal{R}$- and $\mathcal{R}'$-type regions with a single triangle removed (i.e. when $n=1$):
\begin{equation}\label{Ranjan1}
\M(\mathcal{R}_{x,y,z}(a_1))=\Pn_{y+a_1-1,y+a_1-1,x} \Pn_{z,y+z-1,a_1}\frac{\Pn_{y+z-1,y+z-1,x+a_1}}{\Pn_{y+z-1,y+z-1,a_1}}\frac{\Pn_{y-1,y-1,a_1}}{\Pn_{y-1,y-1,x+a_1}},
\end{equation}
and
\begin{equation}\label{Ranjan2}
\M(\mathcal{R}'_{x,y,z}(a_1))=\Pn'_{y+a_1-1,y+a_1-1,x} \Pn'_{z,y+z-1,a_1}\frac{\Pn'_{y+z-1,y+z-1,x+a_1}}{\Pn'_{y+z-1,y+z-1,a_1}}\frac{\Pn'_{y-1,y-1,a_1}}{\Pn'_{y-1,y-1,x+a_1}}
\end{equation}
(see Theorems 4.1 and 4.2 and Lemma 4.4 in \cite{Ranjan}).

We also get the following similar identities for any $\mathcal{R}$-type regions from Theorem \ref{main1} and Lemma \ref{QAR}.
\begin{cor}\label{coro2} Assume that $x,y,z$ are non-negative  integers , and that $\textbf{a}=(a_1,a_2,\dotsc,a_n)$ is a sequence of positive integers. Then
\begin{align}
\M(\mathcal{R}_{x,y,z}&(a_1,a_2,\dotsc,a_{n}))=\M(\mathcal{R}_{x,y,a_{1}}(a_2,\dotsc,a_{n}))
\frac{\Q(0,\frac{y-1}{2}+a_{2\lfloor\frac{n+1}{2}\rfloor},a_{2\lfloor\frac{n+1}{2}\rfloor-1},\dotsc,a_1,z)}
{\Q(0,\frac{y-1}{2}+a_{2\lfloor\frac{n+1}{2}\rfloor},a_{2\lfloor\frac{n+1}{2}\rfloor-1},\dotsc,a_2)}\notag\\
&\times\frac{\Q(x+a_{2\lfloor\frac{n}{2}\rfloor+1},a_{2\lfloor\frac{n}{2}\rfloor},\dotsc,a_{1},y+z-1)}
{\Q(a_{2\lfloor\frac{n}{2}\rfloor+1},\dotsc,a_{1},y+z-1)}
\frac{\Q(a_{2\lfloor\frac{n}{2}\rfloor+1},\dotsc,a_{1},y-1)}{\Q(x+a_{2\lfloor\frac{n}{2}\rfloor+1},\dotsc,a_{1},y-1)}
\end{align}
if $y$ is odd,  and for even $y$
\begin{align}
\M(\mathcal{R}_{x,y,z}&(a_1,a_2,\dotsc,a_{n}))=\M(\mathcal{R}_{x,y,a_{1}}(a_2,\dotsc,a_{n}))
\frac{\K(0,\frac{y}{2}+a_{2\lfloor\frac{n+1}{2}\rfloor},a_{2\lfloor\frac{n+1}{2}\rfloor-1},\dotsc,a_1,z)}
{\K(0,\frac{y}{2}+a_{2\lfloor\frac{n+1}{2}\rfloor},a_{2\lfloor\frac{n+1}{2}\rfloor-1},\dotsc,a_3,a_2)}\notag\\
&\times\frac{\K(x+a_{2\lfloor\frac{n}{2}\rfloor+1},a_{2\lfloor\frac{n}{2}\rfloor},\dotsc,a_{1},y+z)}
{\K(a_{2\lfloor\frac{n}{2}\rfloor+1},\dotsc,a_{1},y+z)}\frac{\K(a_{2\lfloor\frac{n}{2}\rfloor+1},\dotsc,a_{1},y)}
{\K(x+a_{2\lfloor\frac{n}{2}\rfloor+1},a_{2\lfloor\frac{n}{2}\rfloor},\dotsc,a_{1},y)},
\end{align}
where $a_k:=0$ when $k>n$ by convention.
\end{cor}
Similarly, we have the following recurrences on the number of tilings of $\mathcal{R}'$-type regions from Theorem \ref{main2} and Lemma \ref{QAR}.
\begin{cor}\label{coro3} Assume that $x,y,z$ are non-negative integers and that $\textbf{a}=(a_1,a_2,\dotsc,a_n)$ is a sequence of positive integers. Then
\begin{align}
\M(\mathcal{R}'_{x,y,z}&(a_1,a_2,\dotsc,a_{n}))=\M(\mathcal{R}'_{x,y,a_{1}}(a_2,\dotsc,a_{n}))
\frac{\Q'(0,\frac{y-1}{2}+a_{2\lfloor\frac{n+1}{2}\rfloor},a_{2\lfloor\frac{n+1}{2}\rfloor-1},\dotsc,a_1,z)}
{\Q'(0,\frac{y-1}{2}+a_{2\lfloor\frac{n+1}{2}\rfloor},a_{2\lfloor\frac{n+1}{2}\rfloor-1},\dotsc,a_2)}\notag\\
&\times\frac{\Q'(x+a_{2\lfloor\frac{n}{2}\rfloor+1},a_{2\lfloor\frac{n}{2}\rfloor},\dotsc,a_{1},y+z-1)}
{\Q'(a_{2\lfloor\frac{n}{2}\rfloor+1},\dotsc,a_{1},y+z-1)}
\frac{\Q'(a_{2\lfloor\frac{n}{2}\rfloor+1},\dotsc,a_{1},y-1)}{\Q'(x+a_{2\lfloor\frac{n}{2}\rfloor+1},\dotsc,a_{1},y-1)}
\end{align}
if $y$ is odd,  and for even $y$
\begin{align}
\M(\mathcal{R}'_{x,y,z}&(a_1,a_2,\dotsc,a_{n}))=\M(\mathcal{R}'_{x,y,a_{1}}(a_2,\dotsc,a_{n}))\frac{\K(0,\frac{y}{2}+a_{2\lfloor\frac{n+1}{2}\rfloor},a_{2\lfloor\frac{n+1}{2}\rfloor-1},\dotsc,a_1,z)}
{\K'(0,\frac{y}{2}+a_{2\lfloor\frac{n+1}{2}\rfloor},a_{2\lfloor\frac{n+1}{2}\rfloor-1},\dotsc,a_3,a_2)}\notag\\
&\times\frac{\K(x+a_{2\lfloor\frac{n}{2}\rfloor+1},a_{2\lfloor\frac{n}{2}\rfloor},\dotsc,a_{1},y+z)}
{\K'(a_{2\lfloor\frac{n}{2}\rfloor+1},\dotsc,a_{1},y+z)}\frac{\K'(a_{2\lfloor\frac{n}{2}\rfloor+1},\dotsc,a_{1},y)}
{\K'(x+a_{2\lfloor\frac{n}{2}\rfloor+1},a_{2\lfloor\frac{n}{2}\rfloor},\dotsc,a_{1},y)},
\end{align}
where $a_k:=0$ when $k>n$  by convention.
\end{cor}

\begin{rmk}
We notice that even though our main theorems do not allow degenerated triangular holes (i.e., $a_1,a_2,\dotsc,a_n$ are all positive), the degenerated cases can be implied from our theorems as follows. The case when all $a_i\ 's$ equal $0$ is exactly the same as the case when the sequence $\textbf{a}$ is empty. If the last term $a_n=0$, then the region $\mathcal{R}_{x,y,z}(a_1,a_2,\dotsc,a_n)$ is the same as $\mathcal{R}_{x,y,z}(a_1,a_2,\dotsc,a_{n-1})$. If $a_1=0$ and $n>1$, then there are several lozenges that are forced to be in any tilings of the regions. By removing these forced lozenges, we get the region $\mathcal{R}_{x,y,z+a_2}(a_3,\dotsc,a_n)$ that has the same number of tilings as the original region (see Figure \ref{forcehalfhex}(a)). If $a_i=0$ for some $1<i<n$, there are also some forced lozenges between the $(i-1)$-th and the $(i+1)$-th holes. The removal of these forced lozenges does not change the number of tilings of the original region and gives the region $\mathcal{R}_{x,y,z}(a_1,\dotsc,a_{i-2},a_{i-1}+a_{i+1},\dotsc,a_n)$ (see Figure \ref{forcehalfhex}(b)). This way we can eliminate all degenerated holes in our region and obtain a new $\mathcal{R}$-region with the same number of tilings whose holes are not degenerated. The same arguments work for the case of $\mathcal{R}'$-type regions.
\end{rmk}

\begin{figure}\centering
\setlength{\unitlength}{3947sp}%
\begingroup\makeatletter\ifx\SetFigFont\undefined%
\gdef\SetFigFont#1#2#3#4#5{%
  \reset@font\fontsize{#1}{#2pt}%
  \fontfamily{#3}\fontseries{#4}\fontshape{#5}%
  \selectfont}%
\fi\endgroup%
\resizebox{10cm}{!}{
\begin{picture}(0,0)%
\includegraphics{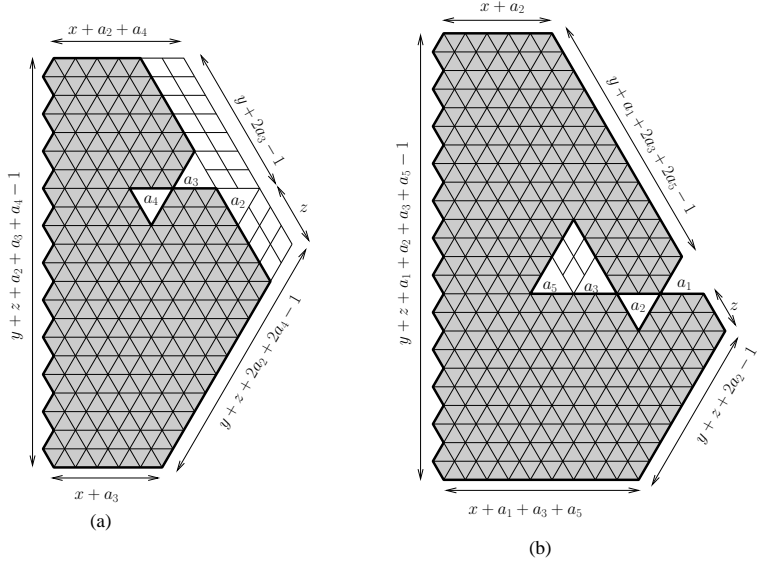}%
\end{picture}%
\begin{picture}(18178,13764)(2380,-13368)
\put(7831,-4786){\makebox(0,0)[lb]{\smash{{\SetFigFont{25}{30.0}{\rmdefault}{\mddefault}{\updefault}{$a_2$}%
}}}}
\put(6706,-4261){\makebox(0,0)[lb]{\smash{{\SetFigFont{25}{30.0}{\rmdefault}{\mddefault}{\updefault}{$a_3$}%
}}}}
\put(5746,-4726){\makebox(0,0)[lb]{\smash{{\SetFigFont{25}{30.0}{\rmdefault}{\mddefault}{\updefault}{$a_4$}%
}}}}
\put(7966,-2191){\rotatebox{300.0}{\makebox(0,0)[lb]{\smash{{\SetFigFont{25}{30.0}{\rmdefault}{\mddefault}{\updefault}{$y+2a_3-1$}%
}}}}}
\put(3931,-631){\makebox(0,0)[lb]{\smash{{\SetFigFont{25}{30.0}{\rmdefault}{\mddefault}{\updefault}{$x+a_2+a_4$}%
}}}}
\put(4081,-11896){\makebox(0,0)[lb]{\smash{{\SetFigFont{25}{30.0}{\rmdefault}{\mddefault}{\updefault}{$x+a_3$}%
}}}}
\put(7711,-10186){\rotatebox{60.0}{\makebox(0,0)[lb]{\smash{{\SetFigFont{25}{30.0}{\rmdefault}{\mddefault}{\updefault}{$y+z+2a_2+2a_4-1$}%
}}}}}
\put(9481,-4951){\makebox(0,0)[lb]{\smash{{\SetFigFont{25}{30.0}{\rmdefault}{\mddefault}{\updefault}{$z$}%
}}}}
\put(2761,-8041){\rotatebox{90.0}{\makebox(0,0)[lb]{\smash{{\SetFigFont{25}{30.0}{\rmdefault}{\mddefault}{\updefault}{$y+z+a_2+a_3+a_4-1$}%
}}}}}
\put(13816,-76){\makebox(0,0)[lb]{\smash{{\SetFigFont{25}{30.0}{\rmdefault}{\mddefault}{\updefault}{$x+a_2$}%
}}}}
\put(16891,-1726){\rotatebox{300.0}{\makebox(0,0)[lb]{\smash{{\SetFigFont{25}{30.0}{\rmdefault}{\mddefault}{\updefault}{$y+a_1+2a_3+2a_5-1$}%
}}}}}
\put(19801,-7306){\makebox(0,0)[lb]{\smash{{\SetFigFont{25}{30.0}{\rmdefault}{\mddefault}{\updefault}{$z$}%
}}}}
\put(19156,-10561){\rotatebox{60.0}{\makebox(0,0)[lb]{\smash{{\SetFigFont{25}{30.0}{\rmdefault}{\mddefault}{\updefault}{$y+z+2a_2-1$}%
}}}}}
\put(13531,-12256){\makebox(0,0)[lb]{\smash{{\SetFigFont{25}{30.0}{\rmdefault}{\mddefault}{\updefault}{$x+a_1+a_3+a_5$}%
}}}}
\put(12076,-8161){\rotatebox{90.0}{\makebox(0,0)[lb]{\smash{{\SetFigFont{25}{30.0}{\rmdefault}{\mddefault}{\updefault}{$y+z+a_1+a_2+a_3+a_5-1$}%
}}}}}
\put(18541,-6766){\makebox(0,0)[lb]{\smash{{\SetFigFont{25}{30.0}{\rmdefault}{\mddefault}{\updefault}{$a_1$}%
}}}}
\put(17431,-7306){\makebox(0,0)[lb]{\smash{{\SetFigFont{25}{30.0}{\rmdefault}{\mddefault}{\updefault}{$a_2$}%
}}}}
\put(16306,-6856){\makebox(0,0)[lb]{\smash{{\SetFigFont{25}{30.0}{\rmdefault}{\mddefault}{\updefault}{$a_3$}%
}}}}
\put(15316,-6811){\makebox(0,0)[lb]{\smash{{\SetFigFont{25}{30.0}{\rmdefault}{\mddefault}{\updefault}{$a_5$}%
}}}}
\end{picture}}
\caption{Eliminating degenerated holes from a $\mathcal{R}$-type region.} \label{forcehalfhex}
\end{figure}

The rest of the paper is organized as follows. In Section \ref{Background}, we introduce the particular version of Kuo condensation that will be employed in our proofs. Next, we prove Lemma \ref{QAR} by using a bijection between lozenge tilings and families of non-intersecting lattice paths in Section \ref{Quarter}.  In Section \ref{Proofmain}, we present the proofs of our main results, and all algebraic simplifications will be carried out in Section \ref{Recurrence}. Finally, we conclude the paper by several remarks Section \ref{Conclusion}.

\section{Preliminaries}\label{Background}

Let $G=(E,V)$ be a finite simple graph without loops. A \emph{perfect matching} (or simply \emph{matching}) $\mu$ of $G$ is a subset of the edge set $E$  that covers each vertex in $V$ exactly once. 
Tilings of a region $\mathcal{R}$ can be identified with  matchings of its \emph{(planar) dual graph}, the graph whose vertices are unit triangles in $\mathcal{R}$ and whose edges connect precisely two unit triangles sharing an edge. In the weighted case, each edge of the dual graph $G$ carries the same weight as it corresponding lozenge in $\mathcal{R}$. We use the notation $\M(G)$ for the sum of weights of all matchings in $G$, where the \emph{weight} of a matching is the product of all weights of  its edges. If the graph $G$ is unweighted, then $\M(G)$ is reduced to  the number of matchings of $G$.

A \emph{forced lozenge} of a region is a lozenge that is contained in every tiling of the region. Assume that we remove several forced lozenges $l_1,l_2,\dotsc,l_k$ from a region $\mathcal{R}$ and obtain a new region $\mathcal{R}'$, then
\begin{equation}
\M(\mathcal{R})=\M(\mathcal{R}') \cdot \prod_{i=1}^{k} wt(l_i),
\end{equation}
where $wt(l_i)$ is the weight of the forced lozenge $l_i$.

The unit triangles in the triangular lattice have two orientations: up-pointing and down-pointing. It is easy to see that if a region $\mathcal{R}$ admits a tiling, then $\mathcal{R}$ has the same number of up-pointing and down-pointing unit triangles. If a region $\mathcal{R}$ satisfies the latter balancing condition, we say that  $\mathcal{R}$ is \emph{balanced}. The following useful lemma allows one to `break down' a region into two smaller subregions when enumerating its tilings.

\begin{lem}[Region-Splitting Lemma]\label{RS}
Let $\mathcal{R}$ be a balanced region on the triangular lattice, and $\mathcal{S}$ a subregion of $\mathcal{R}$ satisfying the following conditions:
\begin{enumerate}
\item[(i)] All unit triangles running along each side of the boundary between $\mathcal{S}$ and its complement $\mathcal{R}-\mathcal{S}$ have the same orientation  (all up-pointing or all    down-pointing);
\item[(ii)] $\mathcal{S}$ is balanced.
\end{enumerate}
Then $\M(\mathcal{R})=\M(\mathcal{S})\M(\mathcal{R}-\mathcal{S})$.
\end{lem}
\begin{proof}
Let $G$ and $H$ be the dual graphs of $\mathcal{R}$ and $\mathcal{S}$, respectively. Then $G$ and $H$ satisfy the conditions in Lemma 2.1 in \cite{Lai2}, and the lemma follows from the duality between tilings and matchings.
\end{proof}

 Eric H. Kuo \cite{Kuo} introduced a graphical counterpart of the well-known \emph{Dodgson condensation} in linear algebra (which is a special case of the \emph{Desnanot--Jacobi identity}, see e.g. \cite{Abeles, Dodgson, Mui}) to (re-)prove the Aztec diamond theorem by Ekies, Kupperberg, Larsen, and Propp \cite{EKLP1,EKLP2}. His method, usually mentioned as \emph{Kuo condensation}, has become a strong tool in the enumeration of tilings and matchings. We will employ the following versions of Kuo condensation in our proofs.

\begin{thm}[Theorem 5.1 in \cite{Kuo}]\label{kuothm1}
Let $G=(V_1,V_2,E)$ be a (weighted) planar bipartite graph with the two vertex classes $V_1$ and $V_2$ so that $|V_1|=|V_2|$.  Assume that $u,v,w,s$ are four vertices appearing on a face of $G$ in a cyclic order, such that $u,w\in V_1$ and $v,s\in V_2$. Then
\begin{equation}
\M(G)\M(G-\{u,v,w,s\})=\M(G-\{u,v\})\M(G-\{w,s\})+\M(G-\{u,s\})\M(G-\{v,w\}).
\end{equation}
\end{thm}

\begin{thm}[Theorem 5.3 in \cite{Kuo}]\label{kuothm2}
Let $G=(V_1,V_2,E)$ be a (weighted) planar bipartite graph with the two vertex classes $V_1$ and $V_2$ so that $|V_1|=|V_2|+1$.  Assume that $u,v,w,s$ are four vertices appearing on a face of $G$ in a cyclic order, such that $u,v,w\in V_1$ and $s\in V_2$. Then
\begin{equation}
\M(G-\{v\})\M(G-\{u,w,s\})=\M(G-\{u\})\M(G-\{v,w,s\})+\M(G-\{w\})\M(G-\{u,v,s\}).
\end{equation}
\end{thm}

\section{Quartered hexagons and proof of Lemma \ref{QAR}}\label{Quarter}

This section is devoted to the proof of Lemma \ref{QAR}, based on the previous result of the author in \cite{Lai} about a family of regions called \emph{quartered hexagons}  as follows.

We start with a trapezoidal region whose northern, northeastern, and southern have lengths $n,m,n+\left\lfloor\frac{m+1}{2}\right\rfloor$, respectively, and the western side follows the vertical  zigzag lattice path with $\frac{m}{2}$ steps (when $m$ is odd, the western side has $\frac{m-1}{2}$ and a half `bumps'). Next, we remove $k=\left\lfloor\frac{m+1}{2}\right\rfloor$ up-pointing unit triangles at the positions $a_1, a_2, \dotsc, a_k$ (ordered from left to right) from the base of the trapezoidal region and obtain the \emph{quartered hexagon} $L_{m,n}(a_1,a_2,\dotsc,a_k)$ (see Figure \ref{halfhex3b} (a)  for the case of even $m$, and Figure \ref{halfhex3b}(b) for the case of odd $m$). We also consider the weighted version $\overline{L}_{m,n}(a_1,a_2,\dotsc,a_k)$ of the quartered hexagon $L_{m,n}(a_1,a_2,\dotsc,a_k)$ by assigning to each vertical lozenge on the western  side a weight $\frac{1}{2}$ (see Figures \ref{halfhex3b}(c) and (d); the lozenges having shaded cores  are weighted by $\frac{1}{2}$).

\begin{figure}
  \centering
  \includegraphics[width=8cm]{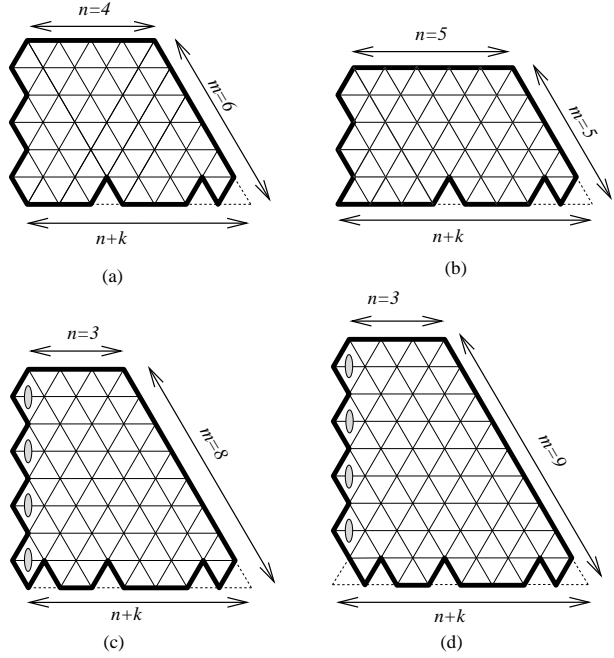}
  \caption{The quartered hexagons: (a) $L_{6,4}(3,6,7)$, (b)  $L_{5,5}(4,7,8)$, (c) $\overline{L}_{8,3}(1,3,6,7)$, and (d) $\overline{L}_{9,3}(1,2,4,7,8)$.}\label{halfhex3b}
\end{figure}

\begin{lem}\label{QAR2}
For any $1\leq k <n$ and $1\leq a_1<a_2<\dotsc<a_k\leq n$
\begin{equation}\label{(a)}
\M(L_{2k,n}(a_1,a_2,\dotsc,a_k))=\frac{a_1a_2\dotsc a_k}{\Hf_2(2k+1)}\prod_{1\leq i<j\leq k}(a_j-a_i)\prod_{1\leq i<j\leq k}(a_i+a_j),
\end{equation}
\begin{equation}\label{(b)}
\M(L_{2k-1,n}(a_1,a_2,\dotsc,a_k))=\frac{1}{\Hf_2(2k)}\prod_{1\leq i<j\leq k}(a_j-a_i)\prod_{1\leq i<j\leq k}(a_i+a_j-1),
\end{equation}
\begin{equation}\label{(c)}
\M(\overline{L}_{2k,n}(a_1,a_2,\dotsc,a_k))=\frac{2^{-k}}{\Hf_2(2k+1)}\prod_{1\leq i<j\leq k}(a_j-a_i)\prod_{1\leq i\leq j\leq k}(a_i+a_j-1),
\end{equation}
\begin{equation}\label{(d)}
\M(\overline{L}_{2k-1,n}(a_1,a_2,\dotsc,a_k))=\frac{1}{\Hf_2(2k)}\prod_{1\leq i<j\leq k}(a_j-a_i)\prod_{1\leq i<j\leq k}(a_i+a_j-2).
\end{equation}
\end{lem}

\begin{figure}
  \centering
  \setlength{\unitlength}{3947sp}%
\begingroup\makeatletter\ifx\SetFigFont\undefined%
\gdef\SetFigFont#1#2#3#4#5{%
  \reset@font\fontsize{#1}{#2pt}%
  \fontfamily{#3}\fontseries{#4}\fontshape{#5}%
  \selectfont}%
\fi\endgroup%
\resizebox{13cm}{!}{
\begin{picture}(0,0)%
\includegraphics{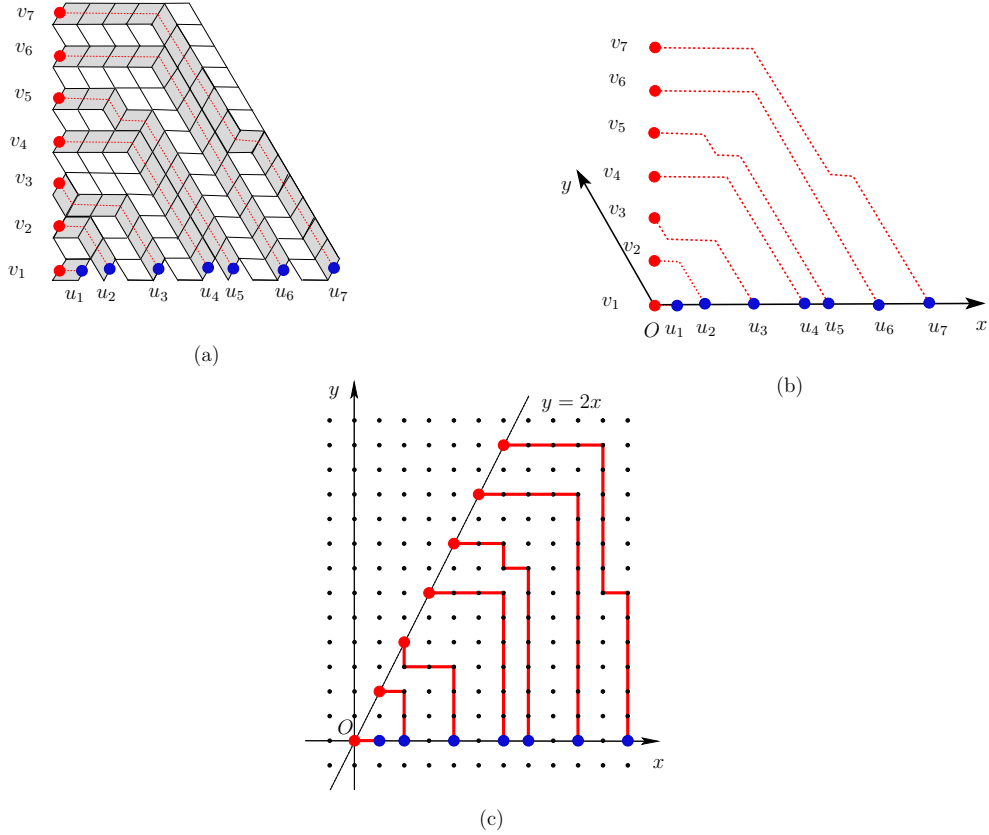}%
\end{picture}%
\begin{picture}(9339,8056)(2897,-7788)
\put(5961,-3583){\makebox(0,0)[lb]{\smash{{\SetFigFont{12}{14.4}{\rmdefault}{\mddefault}{\updefault}{$y$}%
}}}}
\put(9043,-7161){\makebox(0,0)[lb]{\smash{{\SetFigFont{12}{14.4}{\rmdefault}{\mddefault}{\updefault}{$x$}%
}}}}
\put(7391,-7711){\makebox(0,0)[lb]{\smash{{\SetFigFont{12}{14.4}{\rmdefault}{\mddefault}{\updefault}{(c)}%
}}}}
\put(10214,-3550){\makebox(0,0)[lb]{\smash{{\SetFigFont{12}{14.4}{\rmdefault}{\mddefault}{\updefault}{(b)}%
}}}}
\put(4673,-3257){\makebox(0,0)[lb]{\smash{{\SetFigFont{12}{14.4}{\rmdefault}{\mddefault}{\updefault}{(a)}%
}}}}
\put(6044,-6838){\makebox(0,0)[lb]{\smash{{\SetFigFont{12}{14.4}{\rmdefault}{\mddefault}{\updefault}{$O$}%
}}}}
\put(7983,-3708){\makebox(0,0)[lb]{\smash{{\SetFigFont{12}{14.4}{\rmdefault}{\mddefault}{\updefault}{$y=2x$}%
}}}}
\put(8169,-1621){\makebox(0,0)[lb]{\smash{{\SetFigFont{12}{14.4}{\rmdefault}{\mddefault}{\updefault}{$y$}%
}}}}
\put(12114,-2956){\makebox(0,0)[lb]{\smash{{\SetFigFont{12}{14.4}{\rmdefault}{\mddefault}{\updefault}{$x$}%
}}}}
\put(8956,-3024){\makebox(0,0)[lb]{\smash{{\SetFigFont{12}{14.4}{\rmdefault}{\mddefault}{\updefault}{$O$}%
}}}}
\put(11660,-2979){\makebox(0,0)[lb]{\smash{{\SetFigFont{12}{14.4}{\rmdefault}{\mddefault}{\itdefault}{$u_7$}%
}}}}
\put(8556,-2729){\makebox(0,0)[lb]{\smash{{\SetFigFont{12}{14.4}{\rmdefault}{\mddefault}{\itdefault}{$v_1$}%
}}}}
\put(8761,-2206){\makebox(0,0)[lb]{\smash{{\SetFigFont{12}{14.4}{\rmdefault}{\mddefault}{\itdefault}{$v_2$}%
}}}}
\put(8623,-1851){\makebox(0,0)[lb]{\smash{{\SetFigFont{12}{14.4}{\rmdefault}{\mddefault}{\itdefault}{$v_3$}%
}}}}
\put(8571,-1491){\makebox(0,0)[lb]{\smash{{\SetFigFont{12}{14.4}{\rmdefault}{\mddefault}{\itdefault}{$v_4$}%
}}}}
\put(8616,-1034){\makebox(0,0)[lb]{\smash{{\SetFigFont{12}{14.4}{\rmdefault}{\mddefault}{\itdefault}{$v_5$}%
}}}}
\put(8623,-591){\makebox(0,0)[lb]{\smash{{\SetFigFont{12}{14.4}{\rmdefault}{\mddefault}{\itdefault}{$v_6$}%
}}}}
\put(2912,-2432){\makebox(0,0)[lb]{\smash{{\SetFigFont{12}{14.4}{\rmdefault}{\mddefault}{\itdefault}{$v_1$}%
}}}}
\put(2972,-1997){\makebox(0,0)[lb]{\smash{{\SetFigFont{12}{14.4}{\rmdefault}{\mddefault}{\itdefault}{$v_2$}%
}}}}
\put(2979,-1554){\makebox(0,0)[lb]{\smash{{\SetFigFont{12}{14.4}{\rmdefault}{\mddefault}{\itdefault}{$v_3$}%
}}}}
\put(2927,-1194){\makebox(0,0)[lb]{\smash{{\SetFigFont{12}{14.4}{\rmdefault}{\mddefault}{\itdefault}{$v_4$}%
}}}}
\put(2972,-737){\makebox(0,0)[lb]{\smash{{\SetFigFont{12}{14.4}{\rmdefault}{\mddefault}{\itdefault}{$v_5$}%
}}}}
\put(2979,-294){\makebox(0,0)[lb]{\smash{{\SetFigFont{12}{14.4}{\rmdefault}{\mddefault}{\itdefault}{$v_6$}%
}}}}
\put(2994, 44){\makebox(0,0)[lb]{\smash{{\SetFigFont{12}{14.4}{\rmdefault}{\mddefault}{\itdefault}{$v_7$}%
}}}}
\put(3451,-2664){\makebox(0,0)[lb]{\smash{{\SetFigFont{12}{14.4}{\rmdefault}{\mddefault}{\itdefault}{$u_1$}%
}}}}
\put(3751,-2656){\makebox(0,0)[lb]{\smash{{\SetFigFont{12}{14.4}{\rmdefault}{\mddefault}{\itdefault}{$u_2$}%
}}}}
\put(4246,-2649){\makebox(0,0)[lb]{\smash{{\SetFigFont{12}{14.4}{\rmdefault}{\mddefault}{\itdefault}{$u_3$}%
}}}}
\put(4734,-2649){\makebox(0,0)[lb]{\smash{{\SetFigFont{12}{14.4}{\rmdefault}{\mddefault}{\itdefault}{$u_4$}%
}}}}
\put(4974,-2641){\makebox(0,0)[lb]{\smash{{\SetFigFont{12}{14.4}{\rmdefault}{\mddefault}{\itdefault}{$u_5$}%
}}}}
\put(5446,-2641){\makebox(0,0)[lb]{\smash{{\SetFigFont{12}{14.4}{\rmdefault}{\mddefault}{\itdefault}{$u_6$}%
}}}}
\put(5956,-2649){\makebox(0,0)[lb]{\smash{{\SetFigFont{12}{14.4}{\rmdefault}{\mddefault}{\itdefault}{$u_7$}%
}}}}
\put(9155,-2994){\makebox(0,0)[lb]{\smash{{\SetFigFont{12}{14.4}{\rmdefault}{\mddefault}{\itdefault}{$u_1$}%
}}}}
\put(9455,-2986){\makebox(0,0)[lb]{\smash{{\SetFigFont{12}{14.4}{\rmdefault}{\mddefault}{\itdefault}{$u_2$}%
}}}}
\put(9950,-2979){\makebox(0,0)[lb]{\smash{{\SetFigFont{12}{14.4}{\rmdefault}{\mddefault}{\itdefault}{$u_3$}%
}}}}
\put(10438,-2979){\makebox(0,0)[lb]{\smash{{\SetFigFont{12}{14.4}{\rmdefault}{\mddefault}{\itdefault}{$u_4$}%
}}}}
\put(10678,-2971){\makebox(0,0)[lb]{\smash{{\SetFigFont{12}{14.4}{\rmdefault}{\mddefault}{\itdefault}{$u_5$}%
}}}}
\put(11150,-2971){\makebox(0,0)[lb]{\smash{{\SetFigFont{12}{14.4}{\rmdefault}{\mddefault}{\itdefault}{$u_6$}%
}}}}
\put(8638,-253){\makebox(0,0)[lb]{\smash{{\SetFigFont{12}{14.4}{\rmdefault}{\mddefault}{\itdefault}{$v_7$}%
}}}}
\end{picture}}
  \caption{Encoding a tiling of the region $\overline{L}_{13,5}(2,3,5,7,8,10,12)$ as a family of non-intersecting lattice paths in $\mathbb{Z}^2$.}\label{halfhex16}
\end{figure}

\begin{proof}
The tiling formulas (\ref{(a)}),  (\ref{(b)}), and  (\ref{(c)}) were already proven in  \cite[Theorem 3.1]{Lai} by using a bijection between a lozenge tilings and families of nonintersecting lattice paths (see also \cite{KGV} for an equivalent enumeration). The regions $L_{2k,n}(a_1,a_2,\dotsc,a_k)$,  $L_{2k-1,n}(a_1,a_2,\dotsc,a_k)$, and $\overline{L}_{2k,n}(a_1,a_2,\dotsc,a_k)$ were denoted respectively by $QH_{2k,n}(a_1,a_2,\dotsc,a_k)$,  $QH_{2k-1,n}(a_1,a_2,\dotsc,a_k)$, and $\overline{QH}_{2k,n}(a_1,a_2,\dotsc,a_k)$ in \cite{Lai}.   However, the work in \cite{Lai} does \emph{not} cover the last equality (\ref{(d)}) (in particular, the region $\overline{QH}_{2k-1,n}(a_1,a_2,\dotsc,a_k)$ in \cite{Lai} is \emph{different} from our region $\overline{L}_{2k-1,n}(a_1,a_2,\dotsc,a_k)$). We will show here a proof for (\ref{(d)}), following the lines in the proof of Theorem 3.1 in \cite{Lai}.

\medskip

First, we encode each weighted tiling $T$ of the region $\overline{L}_{2k-1,n}(a_1,a_2,\dotsc,a_k)$ by a family of $n$ disjoint lozenge paths consisting right-tilting and vertical lozenges that start from the bottom and end at the western side of the quartered hexagon (see the shaded paths in Figure \ref{halfhex16}(a)). This family of disjoint lozenge paths yields a family of a nonintersecting lattice paths on an obtuse coordinate system (see Figure \ref{halfhex16}(b)). Normalizing the latter coordinate  and rotating it in standard position, we obtain a family of non-intersecting lattice paths $\wp_T=(\lambda_1,\lambda_2,\dotsc,\lambda_k)$ in the square grid $\mathbb{Z}^2$  with the starting points $u_1,u_2,\dotsc, u_k$ and the ending points $v_1,v_2,\dotsc, v_k$, where $u_i=(a_i,0)$ and $v_j=(j-1,2j-2)$ (see Figure \ref{halfhex16}(c)). It is easy to see that the lattice path $\lambda_i$ starting from $u_i$ must end at $v_i$, for any $i=1,2,\dotsc,k$. We assign to each vertical lattice segment connecting $(j,2j)$ and $(j,2j-1)$ (for $j=1,2,\dotsc,n$) a weight $\frac{1}{2}$; all other lattice segments are weighted by $1$. We define the \emph{weight} of  a lattice path in $\mathbb{Z}^2$ to be the weight product of its constituent lattice segments, and the weight of a family of non-intersecting paths is the product of weights of all the paths in the family. This gives a weight-preserving bijection between tilings $T$ of $\overline{L}_{2k-1,n}(a_1,a_2,\dotsc,a_k)$ and families of non-intersecting lattice paths $\wp_T$.

 By Lindstr\"{o}m-Gessel-Viennot Lemma (see \cite[Lemma 1]{Lind},  or  \cite[Theorem 1.2]{Stem}), the sum of weights of all lattice path families  $\wp_T$  is equal  to the determinant of the matrix $A=[a_{i,j}]_{n\times n}$, whose entry $a_{i,j}$ is the weighted sum of all lattice paths connecting $u_i$ and $v_j$. The set of lattice paths connecting $u_i$ and $v_j$ can be partitioned into two subsets: the set $\mathcal{A}$ consisting of all paths that pass the point $(j-2,2j-2)$ and the set $\mathcal{B}$   consisting of all paths that pass the point $(j-1,2j-3)$. All paths in $\mathcal{A}$ have weight $\frac{1}{2}$ (since each of them contains exactly one vertical lattice segment of weight $\frac{1}{2}$) and all paths in $\mathcal{B}$ have weight $1$ (since they consist of all lattice segments with weight $1$). For example, all lattice paths in Figure \ref{halfhex16}(c) have weight $1$, except for the third path $\lambda_3$ that  has weight $\frac{1}{2}$, since its last segment has weight $\frac{1}{2}$. By counting the paths in the sets $\mathcal{A}$ and $\mathcal{B}$, we have
\begin{align}
a_{i,j}&=\frac{1}{2}\binom{a_i+j-3}{2j-3}+\binom{a_i+j-3}{2j-2}\\
&=\frac{a_i-1}{(2j-2)!}(a_i+j-3)(a_i+j-4)\dotsc(a_i-j+1).
\end{align}

\medskip

From the above expression of $a_{i,j}$, we can factor out $\frac{1}{(2j-2)!}$ from the $j$-th column of the matrix $A$. Next, we interchange the $j$-th and the $(n-j+1)$-th columns of the resulting matrix. We get a new matrix $B$ and
\begin{equation}
\det A= \frac{(-1)^{n(n-1)/2}}{2!4!\dotsc(2n-2)!} \det B
\end{equation}
The $(i,j)$-th entry of $B$ is
\begin{equation}
b_{i,j}=(a_i-1) (a_i-n+j)(a_i-n+j+1)\dotsc(a_i+n-j-2).
\end{equation}
We can rewrite $b_{i,j}$ as
\begin{equation}
(X_i-A_n-C)(X_i-A_{n-1}-C)\dotsc(X_i-A_{j+1}-C) \times (X_i+A_n)(X_i+A_{n-1})\dotsc(X_i+A_{j+1}),
\end{equation}
where $X_i=a_i$, $A_{j}=n-j-1$, and $C=2$.

Apply Krattenthaler's determinant identity (Identity (2.10) in \cite[Lemma 4]{Krat2}), we get
\begin{align}
\det B&= \det [(X_i-A_n-C)\dotsc(X_i-A_{j+1}-C) \times (X_i+A_n)\dotsc(X_i+A_{j+1})]_{n\times n}\notag\\
&=\prod_{1\leq i<j\leq n}(X_j-X_i)(C-X_i-X_j)\notag\\
&=(-1)^{n(n-1)/2}\prod_{1\leq i<j\leq n}(a_j-a_i)(a_i+a_j-2).
\end{align}
Thus, we obtain
\begin{equation}
\det A=\frac{1}{2!4!\dotsc(2n-2)!}\prod_{1\leq i<j\leq n}(a_j-a_i)(a_i+a_j-2).
\end{equation}
This implies  (\ref{(d)}) and finishes our proof.
\end{proof}

It is worth noticing that the author gave another proof for the equalities  (\ref{(a)}),  (\ref{(b)}), and  (\ref{(c)}) in \cite{Lai3} by using Kuo condensation. We can also prove (\ref{(d)}) by using the same method.

\begin{figure}\centering
\setlength{\unitlength}{3947sp}%
\begingroup\makeatletter\ifx\SetFigFont\undefined%
\gdef\SetFigFont#1#2#3#4#5{%
  \reset@font\fontsize{#1}{#2pt}%
  \fontfamily{#3}\fontseries{#4}\fontshape{#5}%
  \selectfont}%
\fi\endgroup%
\resizebox{8cm}{!}{
\begin{picture}(0,0)%
\includegraphics{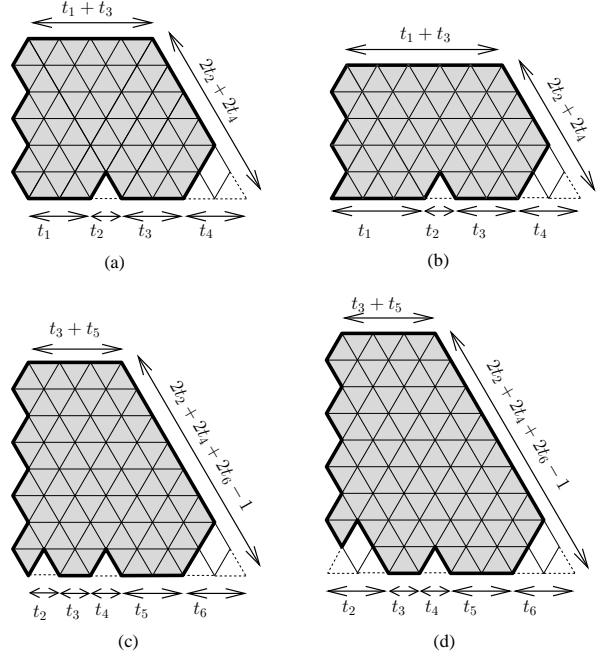}%
\end{picture}%
%
%

\begin{picture}(7577,8368)(1905,-8210)
\put(2566,-96){\makebox(0,0)[lb]{\smash{{\SetFigFont{14}{16.8}{\rmdefault}{\mddefault}{\updefault}{$t_1+t_3$}%
}}}}
\put(2259,-2919){\makebox(0,0)[lb]{\smash{{\SetFigFont{14}{16.8}{\rmdefault}{\mddefault}{\updefault}{$t_1$}%
}}}}
\put(2919,-2911){\makebox(0,0)[lb]{\smash{{\SetFigFont{14}{16.8}{\rmdefault}{\mddefault}{\updefault}{$t_2$}%
}}}}
\put(3526,-2911){\makebox(0,0)[lb]{\smash{{\SetFigFont{14}{16.8}{\rmdefault}{\mddefault}{\updefault}{$t_3$}%
}}}}
\put(2206,-7741){\makebox(0,0)[lb]{\smash{{\SetFigFont{14}{16.8}{\rmdefault}{\mddefault}{\updefault}{$t_2$}%
}}}}
\put(2619,-7726){\makebox(0,0)[lb]{\smash{{\SetFigFont{14}{16.8}{\rmdefault}{\mddefault}{\updefault}{$t_3$}%
}}}}
\put(2409,-4156){\makebox(0,0)[lb]{\smash{{\SetFigFont{14}{16.8}{\rmdefault}{\mddefault}{\updefault}{$t_3+t_5$}%
}}}}
\put(6286,-2911){\makebox(0,0)[lb]{\smash{{\SetFigFont{14}{16.8}{\rmdefault}{\mddefault}{\updefault}{$t_1$}%
}}}}
\put(7134,-2911){\makebox(0,0)[lb]{\smash{{\SetFigFont{14}{16.8}{\rmdefault}{\mddefault}{\updefault}{$t_2$}%
}}}}
\put(7719,-2904){\makebox(0,0)[lb]{\smash{{\SetFigFont{14}{16.8}{\rmdefault}{\mddefault}{\updefault}{$t_3$}%
}}}}
\put(6810,-420){\makebox(0,0)[lb]{\smash{{\SetFigFont{14}{16.8}{\rmdefault}{\mddefault}{\updefault}{$t_1+t_3$}%
}}}}
\put(6084,-7681){\makebox(0,0)[lb]{\smash{{\SetFigFont{14}{16.8}{\rmdefault}{\mddefault}{\updefault}{$t_2$}%
}}}}
\put(6706,-7681){\makebox(0,0)[lb]{\smash{{\SetFigFont{14}{16.8}{\rmdefault}{\mddefault}{\updefault}{$t_3$}%
}}}}
\put(6202,-3846){\makebox(0,0)[lb]{\smash{{\SetFigFont{14}{16.8}{\rmdefault}{\mddefault}{\updefault}{$t_3+t_5$}%
}}}}
\put(4291,-2911){\makebox(0,0)[lb]{\smash{{\SetFigFont{14}{16.8}{\rmdefault}{\mddefault}{\updefault}{$t_4$}%
}}}}
\put(8484,-2904){\makebox(0,0)[lb]{\smash{{\SetFigFont{14}{16.8}{\rmdefault}{\mddefault}{\updefault}{$t_4$}%
}}}}
\put(2986,-7719){\makebox(0,0)[lb]{\smash{{\SetFigFont{14}{16.8}{\rmdefault}{\mddefault}{\updefault}{$t_4$}%
}}}}
\put(7126,-7674){\makebox(0,0)[lb]{\smash{{\SetFigFont{14}{16.8}{\rmdefault}{\mddefault}{\updefault}{$t_4$}%
}}}}
\put(3481,-7711){\makebox(0,0)[lb]{\smash{{\SetFigFont{14}{16.8}{\rmdefault}{\mddefault}{\updefault}{$t_5$}%
}}}}
\put(4216,-7704){\makebox(0,0)[lb]{\smash{{\SetFigFont{14}{16.8}{\rmdefault}{\mddefault}{\updefault}{$t_6$}%
}}}}
\put(7606,-7678){\makebox(0,0)[lb]{\smash{{\SetFigFont{14}{16.8}{\rmdefault}{\mddefault}{\updefault}{$t_5$}%
}}}}
\put(8356,-7678){\makebox(0,0)[lb]{\smash{{\SetFigFont{14}{16.8}{\rmdefault}{\mddefault}{\updefault}{$t_6$}%
}}}}
\put(4282,-721){\rotatebox{300.0}{\makebox(0,0)[lb]{\smash{{\SetFigFont{14}{16.8}{\rmdefault}{\mddefault}{\updefault}{$2t_2+2t_4$}%
}}}}}
\put(8653,-967){\rotatebox{300.0}{\makebox(0,0)[lb]{\smash{{\SetFigFont{14}{16.8}{\rmdefault}{\mddefault}{\updefault}{$2t_2+2t_4$}%
}}}}}
\put(3914,-4801){\rotatebox{300.0}{\makebox(0,0)[lb]{\smash{{\SetFigFont{14}{16.8}{\rmdefault}{\mddefault}{\updefault}{$2t_2+2t_4+2t_6-1$}%
}}}}}
\put(7892,-4537){\rotatebox{300.0}{\makebox(0,0)[lb]{\smash{{\SetFigFont{14}{16.8}{\rmdefault}{\mddefault}{\updefault}{$2t_2+2t_4+2t_6-1$}%
}}}}}
\end{picture}%
}
  \caption{Obtain $\mathcal{Q}$- and $\mathcal{K}$-type regions from quartered hexagons by removing forced lozenges.}\label{halfhex3}
\end{figure}

\begin{proof}[Proof  of Lemma \ref{QAR}]
By removing vertical forced lozenges from the quartered hexagon $L_{2\E(\textbf{t}),\Od(\textbf{t})}(\textbf{I}),$ where
\[\textbf{I}:=\bigcup_{i=1}^{l}\left[\s_{2i-1}(\textbf{t})+1,\s_{2i}(\textbf{t})\right],\]
 we get the the region $\mathcal{Q}(t_1,t_2,\dotsc,t_{2l})$ (see Figures \ref{halfhex3}(a) and (c) for examples). Since all the forced lozenges have weight $1$, (\ref{QARa}) follows from (\ref{(a)}).

Similarly,  the region $\mathcal{Q}'_{a,b}(t_1,t_2,\dotsc,t_{2l})$ is also obtained by removing forced lozenges from the region
$\overline{L}_{2\E(\textbf{t}),\Od(\textbf{t})}(\textbf{I}).$
 Thus,  (\ref{QARb}) follows from (\ref{(c)}).

By the same arguments, we get  (\ref{QARc}) and (\ref{QARd}) from  (\ref{(b)}) and (\ref{(d)}), respectively.
\end{proof}

\section{Proof of the main theorems}\label{Proofmain}
We only prove Theorem \ref{main1}, as  Theorem \ref{main2} can be treated in the same way.

\begin{proof}[Proof of Theorem \ref{main1}]

We will prove both (\ref{ee}) and (\ref{oe})  by induction on $x+y+n$.

The base cases  for (\ref{ee}) are the situations: $x=0$, $y=1$, and $n=0$.

When $n=0$,  (\ref{ee}) follows directly from Corollary \ref{Proctiling}.

\begin{figure}\centering
\setlength{\unitlength}{3947sp}%
\begingroup\makeatletter\ifx\SetFigFont\undefined%
\gdef\SetFigFont#1#2#3#4#5{%
  \reset@font\fontsize{#1}{#2pt}%
  \fontfamily{#3}\fontseries{#4}\fontshape{#5}%
  \selectfont}%
\fi\endgroup%
\resizebox{12cm}{!}{
\begin{picture}(0,0)%
\includegraphics{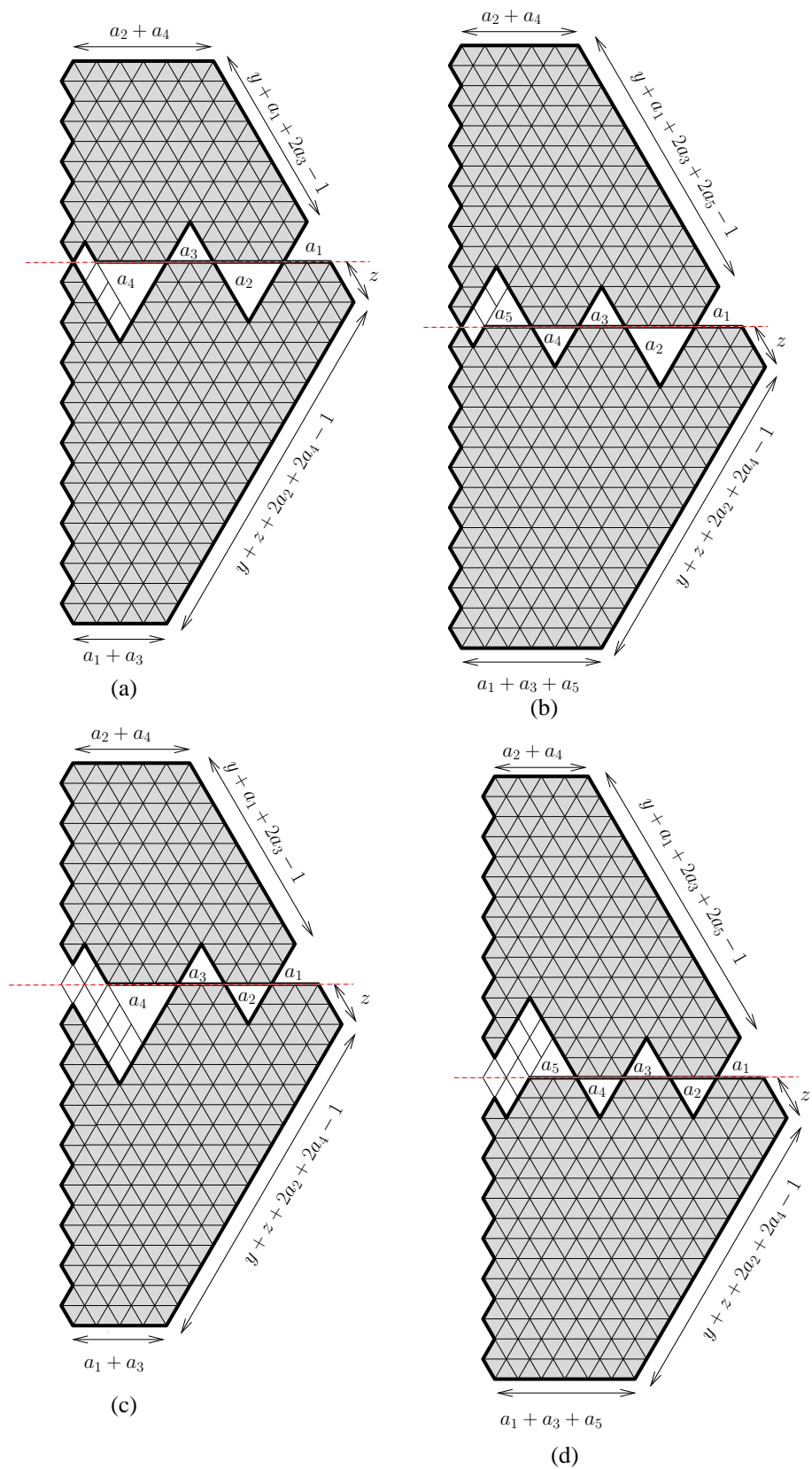}%
\end{picture}%
%
%

\begin{picture}(13408,24743)(1826,-23959)
\put(6781,-3471){\makebox(0,0)[lb]{\smash{{\SetFigFont{20}{24.0}{\rmdefault}{\mddefault}{\updefault}{$a_1$}%
}}}}
\put(5611,-4031){\makebox(0,0)[lb]{\smash{{\SetFigFont{20}{24.0}{\rmdefault}{\mddefault}{\updefault}{$a_2$}%
}}}}
\put(4691,-3541){\makebox(0,0)[lb]{\smash{{\SetFigFont{20}{24.0}{\rmdefault}{\mddefault}{\updefault}{$a_3$}%
}}}}
\put(3641,-4021){\makebox(0,0)[lb]{\smash{{\SetFigFont{20}{24.0}{\rmdefault}{\mddefault}{\updefault}{$a_4$}%
}}}}
\put(13577,-4576){\makebox(0,0)[lb]{\smash{{\SetFigFont{20}{24.0}{\rmdefault}{\mddefault}{\updefault}{$a_1$}%
}}}}
\put(12447,-5126){\makebox(0,0)[lb]{\smash{{\SetFigFont{20}{24.0}{\rmdefault}{\mddefault}{\updefault}{$a_2$}%
}}}}
\put(11537,-4636){\makebox(0,0)[lb]{\smash{{\SetFigFont{20}{24.0}{\rmdefault}{\mddefault}{\updefault}{$a_3$}%
}}}}
\put(10757,-4996){\makebox(0,0)[lb]{\smash{{\SetFigFont{20}{24.0}{\rmdefault}{\mddefault}{\updefault}{$a_4$}%
}}}}
\put(9942,-4601){\makebox(0,0)[lb]{\smash{{\SetFigFont{20}{24.0}{\rmdefault}{\mddefault}{\updefault}{$a_5$}%
}}}}
\put(13896,-17224){\makebox(0,0)[lb]{\smash{{\SetFigFont{20}{24.0}{\rmdefault}{\mddefault}{\updefault}{$a_1$}%
}}}}
\put(6440,-15660){\makebox(0,0)[lb]{\smash{{\SetFigFont{20}{24.0}{\rmdefault}{\mddefault}{\updefault}{$a_1$}%
}}}}
\put(5660,-16060){\makebox(0,0)[lb]{\smash{{\SetFigFont{20}{24.0}{\rmdefault}{\mddefault}{\updefault}{$a_2$}%
}}}}
\put(13086,-17634){\makebox(0,0)[lb]{\smash{{\SetFigFont{20}{24.0}{\rmdefault}{\mddefault}{\updefault}{$a_2$}%
}}}}
\put(4840,-15680){\makebox(0,0)[lb]{\smash{{\SetFigFont{20}{24.0}{\rmdefault}{\mddefault}{\updefault}{$a_3$}%
}}}}
\put(12256,-17254){\makebox(0,0)[lb]{\smash{{\SetFigFont{20}{24.0}{\rmdefault}{\mddefault}{\updefault}{$a_3$}%
}}}}
\put(11496,-17624){\makebox(0,0)[lb]{\smash{{\SetFigFont{20}{24.0}{\rmdefault}{\mddefault}{\updefault}{$a_4$}%
}}}}
\put(3840,-16130){\makebox(0,0)[lb]{\smash{{\SetFigFont{20}{24.0}{\rmdefault}{\mddefault}{\updefault}{$a_4$}%
}}}}
\put(10716,-17224){\makebox(0,0)[lb]{\smash{{\SetFigFont{20}{24.0}{\rmdefault}{\mddefault}{\updefault}{$a_5$}%
}}}}
\put(7841,-3961){\makebox(0,0)[lb]{\smash{{\SetFigFont{20}{24.0}{\rmdefault}{\mddefault}{\updefault}{$z$}%
}}}}
\put(14632,-5056){\makebox(0,0)[lb]{\smash{{\SetFigFont{20}{24.0}{\rmdefault}{\mddefault}{\updefault}{$z$}%
}}}}
\put(15035,-17735){\makebox(0,0)[lb]{\smash{{\SetFigFont{20}{24.0}{\rmdefault}{\mddefault}{\updefault}{$z$}%
}}}}
\put(7695,-16110){\makebox(0,0)[lb]{\smash{{\SetFigFont{20}{24.0}{\rmdefault}{\mddefault}{\updefault}{$z$}%
}}}}
\put(3521,149){\makebox(0,0)[lb]{\smash{{\SetFigFont{20}{24.0}{\rmdefault}{\mddefault}{\updefault}{$a_2+a_4$}%
}}}}
\put(9717,404){\makebox(0,0)[lb]{\smash{{\SetFigFont{20}{24.0}{\rmdefault}{\mddefault}{\updefault}{$a_2+a_4$}%
}}}}
\put(3210,-11650){\makebox(0,0)[lb]{\smash{{\SetFigFont{20}{24.0}{\rmdefault}{\mddefault}{\updefault}{$a_2+a_4$}%
}}}}
\put(10070,-11950){\makebox(0,0)[lb]{\smash{{\SetFigFont{20}{24.0}{\rmdefault}{\mddefault}{\updefault}{$a_2+a_4$}%
}}}}
\put(5851,-511){\rotatebox{300.0}{\makebox(0,0)[lb]{\smash{{\SetFigFont{20}{24.0}{\rmdefault}{\mddefault}{\updefault}{$y+a_1+2a_3-1$}%
}}}}}
\put(12241,-491){\rotatebox{300.0}{\makebox(0,0)[lb]{\smash{{\SetFigFont{20}{24.0}{\rmdefault}{\mddefault}{\updefault}{$y+a_1+2a_3+2a_5-1$}%
}}}}}
\put(5455,-12085){\rotatebox{300.0}{\makebox(0,0)[lb]{\smash{{\SetFigFont{20}{24.0}{\rmdefault}{\mddefault}{\updefault}{$y+a_1+2a_3-1$}%
}}}}}
\put(12385,-12715){\rotatebox{300.0}{\makebox(0,0)[lb]{\smash{{\SetFigFont{20}{24.0}{\rmdefault}{\mddefault}{\updefault}{$y+a_1+2a_3+2a_5-1$}%
}}}}}
\put(3076,-10351){\makebox(0,0)[lb]{\smash{{\SetFigFont{20}{24.0}{\rmdefault}{\mddefault}{\updefault}{$a_1+a_3$}%
}}}}
\put(3090,-22230){\makebox(0,0)[lb]{\smash{{\SetFigFont{20}{24.0}{\rmdefault}{\mddefault}{\updefault}{$a_1+a_3$}%
}}}}
\put(9646,-10816){\makebox(0,0)[lb]{\smash{{\SetFigFont{20}{24.0}{\rmdefault}{\mddefault}{\updefault}{$a_1+a_3+a_5$}%
}}}}
\put(10030,-23180){\makebox(0,0)[lb]{\smash{{\SetFigFont{20}{24.0}{\rmdefault}{\mddefault}{\updefault}{$a_1+a_3+a_5$}%
}}}}
\put(5747,-8886){\rotatebox{60.0}{\makebox(0,0)[lb]{\smash{{\SetFigFont{20}{24.0}{\rmdefault}{\mddefault}{\updefault}{$y+z+2a_2+2a_4-1$}%
}}}}}
\put(13080,-9116){\rotatebox{60.0}{\makebox(0,0)[lb]{\smash{{\SetFigFont{20}{24.0}{\rmdefault}{\mddefault}{\updefault}{$y+z+2a_2+2a_4-1$}%
}}}}}
\put(5895,-20495){\rotatebox{60.0}{\makebox(0,0)[lb]{\smash{{\SetFigFont{20}{24.0}{\rmdefault}{\mddefault}{\updefault}{$y+z+2a_2+2a_4-1$}%
}}}}}
\put(13531,-21763){\rotatebox{60.0}{\makebox(0,0)[lb]{\smash{{\SetFigFont{20}{24.0}{\rmdefault}{\mddefault}{\updefault}{$y+z+2a_2+2a_4-1$}%
}}}}}
\end{picture}}
  \caption{Splitting a $\mathcal{R}$-type region into two $\mathcal{Q}$-type regions and several forced vertical lozenges in the case when $x=0$.}\label{halfhex4}
\end{figure}

Next we consider the case $x=0$. By Region-splitting Lemma \ref{RS}, we can split the region into two subregions along the lattice line containing the bases of the triangular holes as follows. First we consider the case when $n$ are even, say $n=2l$. The shaded subregion above the dotted line in Figure \ref{halfhex4}(a) satisfies the conditions of Region-splitting Lemma \ref{RS}. Moreover, it is easy to see that this subregion is congruent to the region
$\mathcal{Q}\left( 0,\frac{y-1}{2},a_{2l},\dotsc,a_1\right).$
The complement of this subregion has several forced vertical lozenges. By removing these forced lozenges, we get the region $\mathcal{Q}\left(0,\frac{y-1}{2}+a_{2l},a_{2l-1},\dotsc,a_1,z\right)$ (shown as the lower shaded subregion in Figure \ref{halfhex4}(a) for the case $l=2$, $z=2$, $a_1=a_3=2$, $a_2=a_4=3$). Thus, we have for odd $y$
\begin{equation}
\M(\mathcal{R}_{0,y,z}(a_1,a_2,\dotsc,a_{2l}))=\M\left(\mathcal{Q}\left( 0,\frac{y}{2},a_{2l},\dotsc,a_1\right)\right)\M\left(\mathcal{Q}\left(0,\frac{y}{2}+a_{2l},a_{2l-1},\dotsc,a_1,z\right)\right).
\end{equation}

Similarly, as shown in Figure \ref{halfhex4}(b), we also have
\begin{equation}
\M(\mathcal{R}_{0,y,z}(a_1,a_2,\dotsc,a_{2l+1}))=\M\left(\mathcal{Q}\left( 0,\frac{y-1}{2}+a_{2l+1},\dotsc,a_1\right)\right)\M\left(\mathcal{Q}\left(0,\frac{y-1}{2},a_{2l+1},\dotsc,a_1,z\right)\right).
\end{equation}
Then (\ref{ee}) follows from Lemma \ref{QAR}.

\begin{figure}\centering
\setlength{\unitlength}{3947sp}%
\begingroup\makeatletter\ifx\SetFigFont\undefined%
\gdef\SetFigFont#1#2#3#4#5{%
  \reset@font\fontsize{#1}{#2pt}%
  \fontfamily{#3}\fontseries{#4}\fontshape{#5}%
  \selectfont}%
\fi\endgroup%
\resizebox{10cm}{!}{
 \begin{picture}(0,0)%
\includegraphics{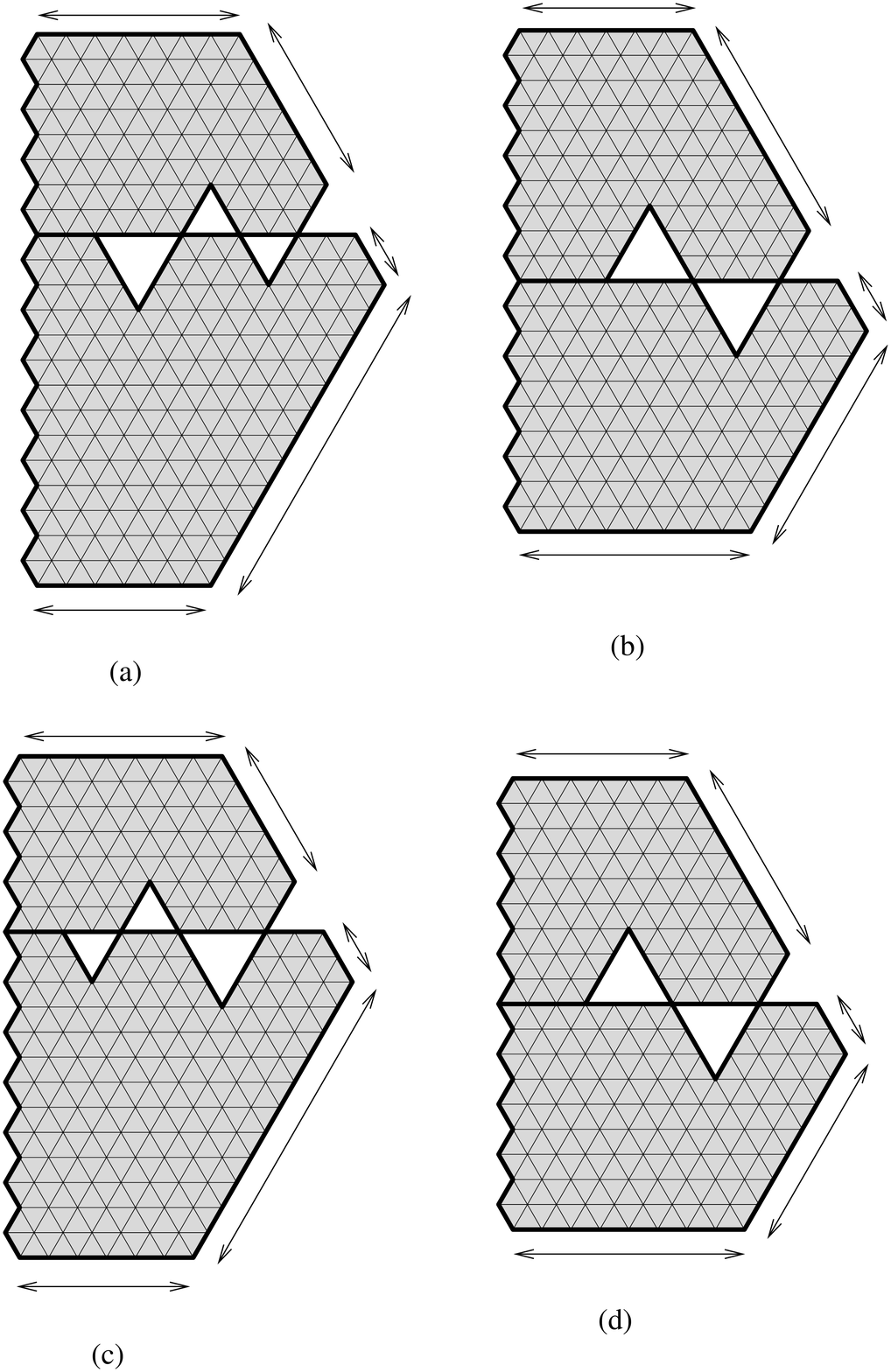}%
\end{picture}%
%
%

\begin{picture}(12068,18908)(1835,-18274)
\put(11866,-10546){\rotatebox{300.0}{\makebox(0,0)[lb]{\smash{{\SetFigFont{20}{24.0}{\rmdefault}{\mddefault}{\updefault}{$a_1+2a_2-1$}%
}}}}}
\put(5430,-9886){\rotatebox{300.0}{\makebox(0,0)[lb]{\smash{{\SetFigFont{20}{24.0}{\rmdefault}{\mddefault}{\updefault}{$a_1+2a_2-1$}%
}}}}}
\put(5491,-16366){\rotatebox{60.0}{\makebox(0,0)[lb]{\smash{{\SetFigFont{20}{24.0}{\rmdefault}{\mddefault}{\updefault}{$z+2a_2+2a_4-1$}%
}}}}}
\put(12815,-16115){\rotatebox{60.0}{\makebox(0,0)[lb]{\smash{{\SetFigFont{20}{24.0}{\rmdefault}{\mddefault}{\updefault}{$z+2a_2-1$}%
}}}}}
\put(12961,-6556){\rotatebox{60.0}{\makebox(0,0)[lb]{\smash{{\SetFigFont{20}{24.0}{\rmdefault}{\mddefault}{\updefault}{$z+2a_2$}%
}}}}}
\put(5986,-6961){\rotatebox{60.0}{\makebox(0,0)[lb]{\smash{{\SetFigFont{20}{24.0}{\rmdefault}{\mddefault}{\updefault}{$z+2a_2+2a_4$}%
}}}}}
\put(11941,-631){\rotatebox{300.0}{\makebox(0,0)[lb]{\smash{{\SetFigFont{20}{24.0}{\rmdefault}{\mddefault}{\updefault}{$a_1+2a_3$}%
}}}}}
\put(5760,-376){\rotatebox{300.0}{\makebox(0,0)[lb]{\smash{{\SetFigFont{20}{24.0}{\rmdefault}{\mddefault}{\updefault}{$a_1+2a_3$}%
}}}}}
\put(6796,-12496){\makebox(0,0)[lb]{\smash{{\SetFigFont{20}{24.0}{\rmdefault}{\mddefault}{\updefault}{$z$}%
}}}}
\put(13383,-13449){\makebox(0,0)[lb]{\smash{{\SetFigFont{20}{24.0}{\rmdefault}{\mddefault}{\updefault}{$z$}%
}}}}
\put(13704,-3721){\makebox(0,0)[lb]{\smash{{\SetFigFont{20}{24.0}{\rmdefault}{\mddefault}{\updefault}{$z$}%
}}}}
\put(7141,-3031){\makebox(0,0)[lb]{\smash{{\SetFigFont{20}{24.0}{\rmdefault}{\mddefault}{\updefault}{$z$}%
}}}}
\put(9340,-17049){\makebox(0,0)[lb]{\smash{{\SetFigFont{20}{24.0}{\rmdefault}{\mddefault}{\updefault}{$x+a_1+a_3$}%
}}}}
\put(2346,-17630){\makebox(0,0)[lb]{\smash{{\SetFigFont{20}{24.0}{\rmdefault}{\mddefault}{\updefault}{$x+a_1+a_3$}%
}}}}
\put(9511,-7711){\makebox(0,0)[lb]{\smash{{\SetFigFont{20}{24.0}{\rmdefault}{\mddefault}{\updefault}{$x+a_1+a_3$}%
}}}}
\put(2716,-8356){\makebox(0,0)[lb]{\smash{{\SetFigFont{20}{24.0}{\rmdefault}{\mddefault}{\updefault}{$x+a_1+a_3$}%
}}}}
\put(2511,-9550){\makebox(0,0)[lb]{\smash{{\SetFigFont{20}{24.0}{\rmdefault}{\mddefault}{\updefault}{$x+a_2+a_4$}%
}}}}
\put(9455,-9710){\makebox(0,0)[lb]{\smash{{\SetFigFont{20}{24.0}{\rmdefault}{\mddefault}{\updefault}{$x+a_2$}%
}}}}
\put(9436,254){\makebox(0,0)[lb]{\smash{{\SetFigFont{20}{24.0}{\rmdefault}{\mddefault}{\updefault}{$x+a_2$}%
}}}}
\put(2881,164){\makebox(0,0)[lb]{\smash{{\SetFigFont{20}{24.0}{\rmdefault}{\mddefault}{\updefault}{$x+a_2+a_4$}%
}}}}
\put(11208,-13704){\makebox(0,0)[lb]{\smash{{\SetFigFont{20}{24.0}{\rmdefault}{\mddefault}{\updefault}{$a_2$}%
}}}}
\put(4529,-12695){\makebox(0,0)[lb]{\smash{{\SetFigFont{20}{24.0}{\rmdefault}{\mddefault}{\updefault}{$a_2$}%
}}}}
\put(10053,-13149){\makebox(0,0)[lb]{\smash{{\SetFigFont{20}{24.0}{\rmdefault}{\mddefault}{\updefault}{$a_3$}%
}}}}
\put(3644,-12230){\makebox(0,0)[lb]{\smash{{\SetFigFont{20}{24.0}{\rmdefault}{\mddefault}{\updefault}{$a_3$}%
}}}}
\put(2879,-12605){\makebox(0,0)[lb]{\smash{{\SetFigFont{20}{24.0}{\rmdefault}{\mddefault}{\updefault}{$a_4$}%
}}}}
\put(5654,-12200){\makebox(0,0)[lb]{\smash{{\SetFigFont{20}{24.0}{\rmdefault}{\mddefault}{\updefault}{$a_1$}%
}}}}
\put(12348,-13164){\makebox(0,0)[lb]{\smash{{\SetFigFont{20}{24.0}{\rmdefault}{\mddefault}{\updefault}{$a_1$}%
}}}}
\put(12600,-3405){\makebox(0,0)[lb]{\smash{{\SetFigFont{20}{24.0}{\rmdefault}{\mddefault}{\updefault}{$a_1$}%
}}}}
\put(11535,-3900){\makebox(0,0)[lb]{\smash{{\SetFigFont{20}{24.0}{\rmdefault}{\mddefault}{\updefault}{$a_2$}%
}}}}
\put(10260,-3420){\makebox(0,0)[lb]{\smash{{\SetFigFont{20}{24.0}{\rmdefault}{\mddefault}{\updefault}{$a_3$}%
}}}}
\put(3437,-3248){\makebox(0,0)[lb]{\smash{{\SetFigFont{20}{24.0}{\rmdefault}{\mddefault}{\updefault}{$a_4$}%
}}}}
\put(4472,-2798){\makebox(0,0)[lb]{\smash{{\SetFigFont{20}{24.0}{\rmdefault}{\mddefault}{\updefault}{$a_3$}%
}}}}
\put(5311,-3181){\makebox(0,0)[lb]{\smash{{\SetFigFont{20}{24.0}{\rmdefault}{\mddefault}{\updefault}{$a_2$}%
}}}}
\put(6107,-2828){\makebox(0,0)[lb]{\smash{{\SetFigFont{20}{24.0}{\rmdefault}{\mddefault}{\updefault}{$a_1$}%
}}}}
\end{picture}%
}
  \caption{(a)--(b) Splitting a $\mathcal{R}$-type region into two $\mathcal{Q}$-type regions in the case when $y=1$. (c)--(d) Splitting a $\mathcal{R}$-type region into two $\mathcal{K}$-type regions in the case when $y=0$.}\label{halfhex5}
\end{figure}

The last base case of (\ref{ee}) is the case when $y=1$. By  Region-splitting Lemma \ref{RS} again, we can partition our region into two parts along the horizontal line containing the bases of the triangular holes and obtain
\begin{equation}
\M(\mathcal{R}_{x,1,z}(a_1,a_2,\dotsc,a_{2l}))=\M(\mathcal{Q}(x+a_{2l},a_{2l-1},\dotsc,a_1))\M(\mathcal{Q}(x,a_{2l},\dotsc,a_1,z))
\end{equation}
(see Figure \ref{halfhex5}(a) for the case $x=2$, $l=2$, $z=2$, $a_1=a_2=a_3=2$, $a_4=3$), and
\begin{equation}
\M(\mathcal{R}_{x,1,z}(a_1,a_2,\dotsc,a_{2l+1}))=\M(\mathcal{Q}(x,a_{2l+1},\dotsc,a_1))\M(\mathcal{Q}(x+a_{2l+1},a_{2l}\dotsc,a_1,z))
\end{equation}
(see Figure \ref{halfhex5}(b) for the case $x=2$, $l=1$, $z=2$, $a_1=2$ $a_2=a_3=3$). Then (\ref{ee}) also follows from Lemma \ref{QAR}.

\medskip

Next, we consider the base cases of (\ref{oe}), that are the situations when one of the parameters $x,y,n$ is equal to $0$.

The case $n=0$ still follows from Corollary \ref{Proctiling}.

If $x=0$, we apply the same arguments in the base case of (\ref{ee}) and have for even $y$
\begin{equation}
\M(\mathcal{R}_{0,y,z}(a_1,a_2,\dotsc,a_{2l}))=\M\left(\mathcal{K}\left( 0,\frac{y}{2},a_{2l},\dotsc,a_1\right)\right)\M\left(\mathcal{K}\left(0,\frac{y}{2}+a_{2l},a_{2l-1},\dotsc,a_1,z\right)\right)
\end{equation}
and
\begin{equation}
\M(\mathcal{R}_{0,y,z}(a_1,a_2,\dotsc,a_{2l+1}))=\M\left(\mathcal{K}\left( 0,\frac{y}{2}+a_{2l+1},\dotsc,a_1\right)\right)\M\left(\mathcal{K}\left(0,\frac{y}{2},a_{2l+1},\dotsc,a_1,z\right)\right)
\end{equation}
(see Figures \ref{halfhex4}(c) and (d), respectively).
Then (\ref{oe}) follows from Lemma \ref{QAR}.

The last base case of (\ref{oe}) is the case when $y=0$. Similar to the case $y=1$ in (\ref{ee}), we divide our region in two $\mathcal{K}$-type regions along the line containing the bases of the triangular holes. By Region-splitting Lemma \ref{RS}, we have
\begin{equation}
\M(\mathcal{R}_{x,0,z}(a_1,a_2,\dotsc,a_{2l}))=\M(\mathcal{K}(x+a_{2l},a_{2l-1},\dotsc,a_1))\M(\mathcal{K}(x,a_{2l},\dotsc,a_1,z))
\end{equation}
(see Figure \ref{halfhex5}(c) for the case $x=2$, $l=2$, $z=2$, $a_1=a_3=a_4=2$, $a_2=3$), and
\begin{equation}
\M(\mathcal{R}_{x,0,z}(a_1,a_2,\dotsc,a_{2l+1}))=\M(\mathcal{K}(x,a_{2l+1},\dotsc,a_1))\M(\mathcal{K}(x+a_{2l+1},a_{2l}\dotsc,a_1,z))
\end{equation}
(see Figure \ref{halfhex5}(d) for the case $x=2$, $l=1$, $z=2$, $a_1=2$ $a_2=a_3=3$). By Lemma \ref{QAR}, we obtain  (\ref{oe}).

\bigskip

For the induction step, we assume that $x>0$, $y>1,$ $n>0$ and that  (\ref{ee}) and (\ref{oe}) both hold for any $\mathcal{R}$-type regions with the sum of their $x$-, $y$-, and $n$-parameters strictly less than $x+y+n$.

\begin{figure}
  \centering
  \setlength{\unitlength}{3947sp}%
\begingroup\makeatletter\ifx\SetFigFont\undefined%
\gdef\SetFigFont#1#2#3#4#5{%
  \reset@font\fontsize{#1}{#2pt}%
  \fontfamily{#3}\fontseries{#4}\fontshape{#5}%
  \selectfont}%
\fi\endgroup%
\resizebox{12cm}{!}{
\begin{picture}(0,0)%
\includegraphics{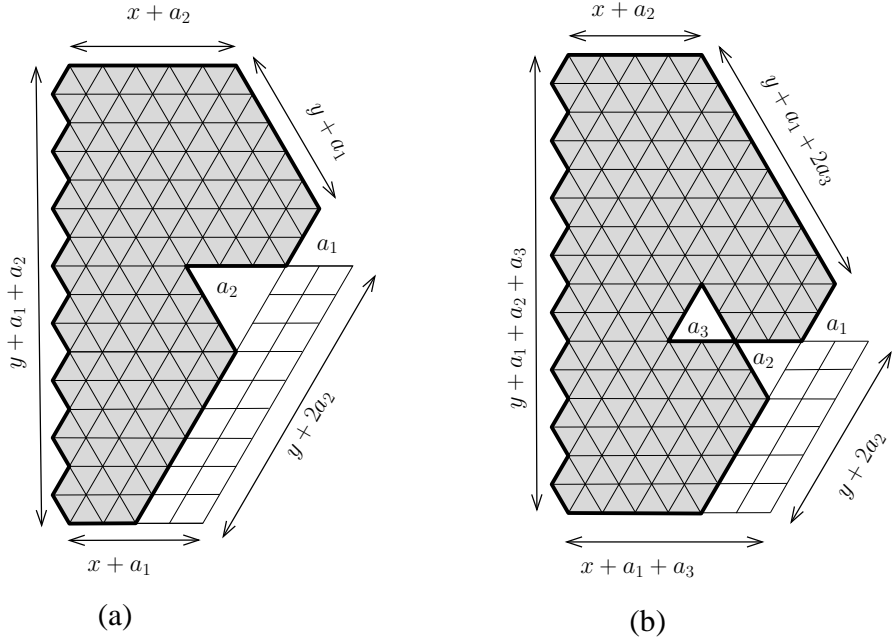}%
\end{picture}%
\begin{picture}(10606,7704)(2839,-8027)
\put(6552,-3404){\makebox(0,0)[lb]{\smash{{\SetFigFont{17}{20.4}{\rmdefault}{\mddefault}{\itdefault}{$a_1$}%
}}}}
\put(5381,-3921){\makebox(0,0)[lb]{\smash{{\SetFigFont{17}{20.4}{\rmdefault}{\mddefault}{\itdefault}{$a_2$}%
}}}}
\put(11651,-4751){\makebox(0,0)[lb]{\smash{{\SetFigFont{17}{20.4}{\rmdefault}{\mddefault}{\itdefault}{$a_2$}%
}}}}
\put(12510,-4351){\makebox(0,0)[lb]{\smash{{\SetFigFont{17}{20.4}{\rmdefault}{\mddefault}{\itdefault}{$a_1$}%
}}}}
\put(9636,-7221){\makebox(0,0)[lb]{\smash{{\SetFigFont{17}{20.4}{\rmdefault}{\mddefault}{\itdefault}{$x+a_1+a_3$}%
}}}}
\put(11778,-1394){\rotatebox{300.0}{\makebox(0,0)[lb]{\smash{{\SetFigFont{17}{20.4}{\rmdefault}{\mddefault}{\itdefault}{$y+a_1+2a_3$}%
}}}}}
\put(6462,-1724){\rotatebox{300.0}{\makebox(0,0)[lb]{\smash{{\SetFigFont{17}{20.4}{\rmdefault}{\mddefault}{\itdefault}{$y+a_1$}%
}}}}}
\put(4332,-689){\makebox(0,0)[lb]{\smash{{\SetFigFont{17}{20.4}{\rmdefault}{\mddefault}{\itdefault}{$x+a_2$}%
}}}}
\put(9765,-646){\makebox(0,0)[lb]{\smash{{\SetFigFont{17}{20.4}{\rmdefault}{\mddefault}{\itdefault}{$x+a_2$}%
}}}}
\put(6306,-5866){\rotatebox{60.0}{\makebox(0,0)[lb]{\smash{{\SetFigFont{17}{20.4}{\rmdefault}{\mddefault}{\itdefault}{$y+2a_2$}%
}}}}}
\put(3101,-4806){\rotatebox{90.0}{\makebox(0,0)[lb]{\smash{{\SetFigFont{17}{20.4}{\rmdefault}{\mddefault}{\itdefault}{$y+a_1+a_2$}%
}}}}}
\put(3871,-7196){\makebox(0,0)[lb]{\smash{{\SetFigFont{17}{20.4}{\rmdefault}{\mddefault}{\itdefault}{$x+a_1$}%
}}}}
\put(10886,-4386){\makebox(0,0)[lb]{\smash{{\SetFigFont{17}{20.4}{\rmdefault}{\mddefault}{\itdefault}{$a_3$}%
}}}}
\put(12751,-6346){\rotatebox{60.0}{\makebox(0,0)[lb]{\smash{{\SetFigFont{17}{20.4}{\rmdefault}{\mddefault}{\itdefault}{$y+2a_2$}%
}}}}}
\put(8951,-5286){\rotatebox{90.0}{\makebox(0,0)[lb]{\smash{{\SetFigFont{17}{20.4}{\rmdefault}{\mddefault}{\itdefault}{$y+a_1+a_2+a_3$}%
}}}}}
\end{picture}}
  \caption{Obtaining a $\mathcal{R}$-type region with less triangular holes by removing forced lozenges when $z=0$.}\label{halfhex12}
\end{figure}

If $z=0$, then, by removing forced lozenges, we obtain a new $\mathcal{R}$-type region (reflected over a horizontal line) with less triangular holes. In particular, we have
 \begin{align}
\M(\mathcal{R}_{x,y,z}(a_1,a_2,\dotsc,a_{2l}))=\M(\mathcal{R}_{x,y,a_{1}}(a_2,\dotsc,a_{2l}))
\end{align}
and
\begin{align}
\M(\mathcal{R}_{x,y,z}(a_1,a_2,\dotsc,a_{2l+1}))=\M(\mathcal{R}_{x,y,a_{1}}(a_2,\dotsc,a_{2l+1}))
\end{align}
(see Figure \ref{halfhex12}).
Then both (\ref{ee}) and (\ref{oe}) follows from the induction hypothesis. Thus, we can assume that $z>0$ in the rest of this proof.

\medskip

Next, we will use Kuo condensation to obtain two recurrences on the tiling numbers of $\mathcal{R}$-type regions.

\medskip

To obtain the first recurrence,  we apply Kuo's Theorem \ref{kuothm1} to the dual graph $G$ of the region $\mathcal{R}_{x,y,z}(a_1,a_2,\dotsc,a_{2l})$ with the four vertices $u,v,w,s$ corresponding to the four black unit triangles as in Figure \ref{halfhex7}(a).

\begin{figure}
  \centering
  \setlength{\unitlength}{3947sp}%
\begingroup\makeatletter\ifx\SetFigFont\undefined%
\gdef\SetFigFont#1#2#3#4#5{%
  \reset@font\fontsize{#1}{#2pt}%
  \fontfamily{#3}\fontseries{#4}\fontshape{#5}%
  \selectfont}%
\fi\endgroup%
\resizebox{10cm}{!}{
\begin{picture}(0,0)%
\includegraphics{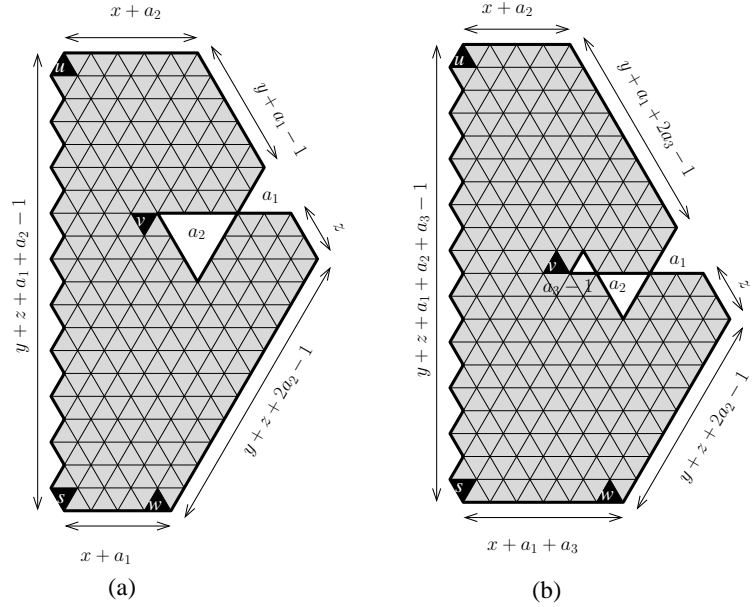}%
\end{picture}%
\begin{picture}(11057,9068)(1872,-9007)
\put(7999,-5628){\rotatebox{90.0}{\makebox(0,0)[lb]{\smash{{\SetFigFont{17}{20.4}{\rmdefault}{\mddefault}{\itdefault}{$y+z+a_1+a_2+a_3-1$}%
}}}}}
\put(2131,-5266){\rotatebox{90.0}{\makebox(0,0)[lb]{\smash{{\SetFigFont{17}{20.4}{\rmdefault}{\mddefault}{\itdefault}{$y+z+a_1+a_2-1$}%
}}}}}
\put(8809,-258){\makebox(0,0)[lb]{\smash{{\SetFigFont{17}{20.4}{\rmdefault}{\mddefault}{\itdefault}{$x+a_2$}%
}}}}
\put(3376,-301){\makebox(0,0)[lb]{\smash{{\SetFigFont{17}{20.4}{\rmdefault}{\mddefault}{\itdefault}{$x+a_2$}%
}}}}
\put(5431,-6796){\rotatebox{60.0}{\makebox(0,0)[lb]{\smash{{\SetFigFont{17}{20.4}{\rmdefault}{\mddefault}{\itdefault}{$y+z+2a_2-1$}%
}}}}}
\put(11809,-7188){\rotatebox{60.0}{\makebox(0,0)[lb]{\smash{{\SetFigFont{17}{20.4}{\rmdefault}{\mddefault}{\itdefault}{$y+z+2a_2-1$}%
}}}}}
\put(12559,-4218){\rotatebox{300.0}{\makebox(0,0)[lb]{\smash{{\SetFigFont{17}{20.4}{\rmdefault}{\mddefault}{\itdefault}{$z$}%
}}}}}
\put(2941,-8341){\makebox(0,0)[lb]{\smash{{\SetFigFont{17}{20.4}{\rmdefault}{\mddefault}{\itdefault}{$x+a_1$}%
}}}}
\put(5596,-3016){\makebox(0,0)[lb]{\smash{{\SetFigFont{17}{20.4}{\rmdefault}{\mddefault}{\itdefault}{$a_1$}%
}}}}
\put(4516,-3556){\makebox(0,0)[lb]{\smash{{\SetFigFont{17}{20.4}{\rmdefault}{\mddefault}{\itdefault}{$a_2$}%
}}}}
\put(10669,-4323){\makebox(0,0)[lb]{\smash{{\SetFigFont{17}{20.4}{\rmdefault}{\mddefault}{\itdefault}{$a_2$}%
}}}}
\put(9700,-4368){\makebox(0,0)[lb]{\smash{{\SetFigFont{17}{20.4}{\rmdefault}{\mddefault}{\itdefault}{$a_3-1$}%
}}}}
\put(11554,-3963){\makebox(0,0)[lb]{\smash{{\SetFigFont{17}{20.4}{\rmdefault}{\mddefault}{\itdefault}{$a_1$}%
}}}}
\put(8884,-8178){\makebox(0,0)[lb]{\smash{{\SetFigFont{17}{20.4}{\rmdefault}{\mddefault}{\itdefault}{$x+a_1+a_3$}%
}}}}
\put(10822,-1006){\rotatebox{300.0}{\makebox(0,0)[lb]{\smash{{\SetFigFont{17}{20.4}{\rmdefault}{\mddefault}{\itdefault}{$y+a_1+2a_3-1$}%
}}}}}
\put(5506,-1336){\rotatebox{300.0}{\makebox(0,0)[lb]{\smash{{\SetFigFont{17}{20.4}{\rmdefault}{\mddefault}{\itdefault}{$y+a_1-1$}%
}}}}}
\put(6631,-3391){\rotatebox{300.0}{\makebox(0,0)[lb]{\smash{{\SetFigFont{17}{20.4}{\rmdefault}{\mddefault}{\itdefault}{$z$}%
}}}}}
\end{picture}}
  \caption{How to apply Kuo condensation to a $\mathcal{R}$-type region.}\label{halfhex7}
\end{figure}

\begin{figure}
  \centering
  \setlength{\unitlength}{3947sp}%
\begingroup\makeatletter\ifx\SetFigFont\undefined%
\gdef\SetFigFont#1#2#3#4#5{%
  \reset@font\fontsize{#1}{#2pt}%
  \fontfamily{#3}\fontseries{#4}\fontshape{#5}%
  \selectfont}%
\fi\endgroup%
\resizebox{10cm}{!}{
\begin{picture}(0,0)%
\includegraphics{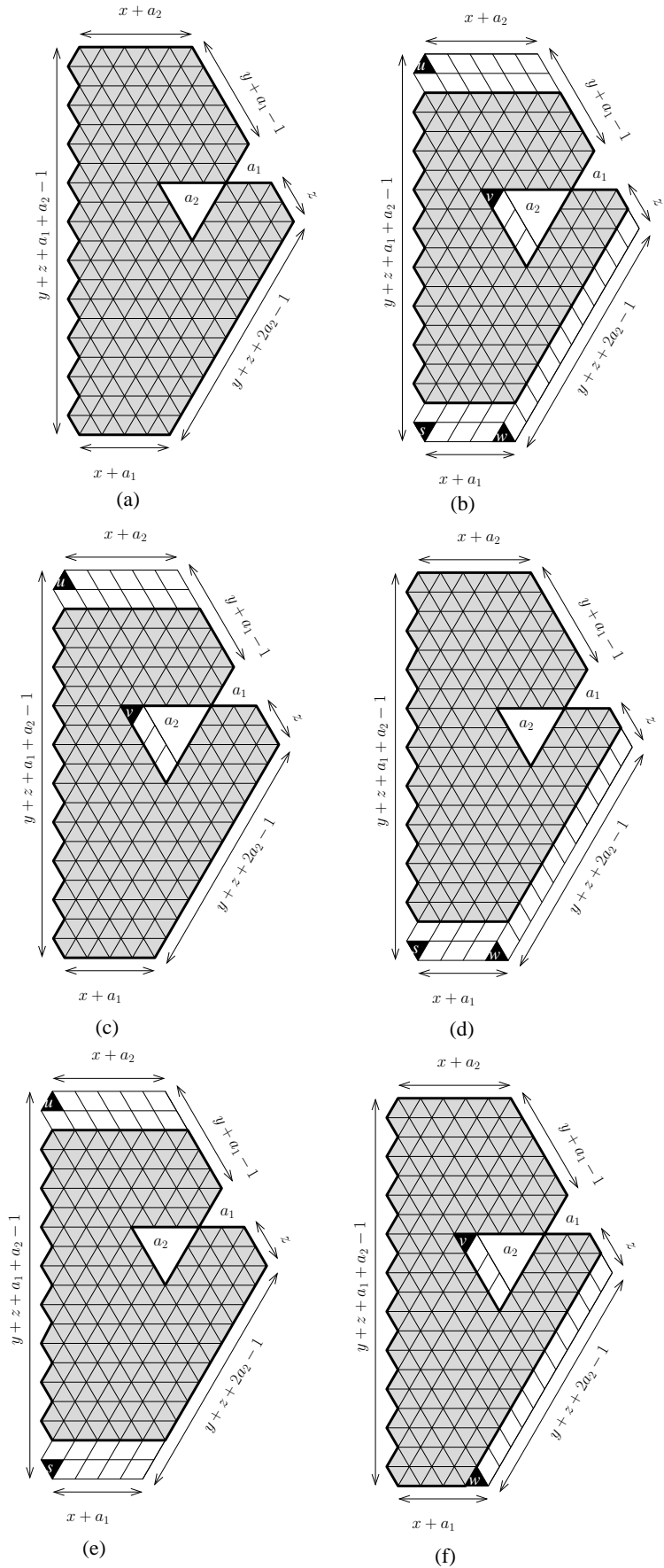}%
\end{picture}%

\begin{picture}(11569,27381)(1466,-27055)
\put(3006,-8041){\makebox(0,0)[lb]{\smash{{\SetFigFont{17}{20.4}{\rmdefault}{\mddefault}{\itdefault}{$x+a_1$}%
}}}}
\put(5661,-2716){\makebox(0,0)[lb]{\smash{{\SetFigFont{17}{20.4}{\rmdefault}{\mddefault}{\itdefault}{$a_1$}%
}}}}
\put(4581,-3256){\makebox(0,0)[lb]{\smash{{\SetFigFont{17}{20.4}{\rmdefault}{\mddefault}{\itdefault}{$a_2$}%
}}}}
\put(5571,-1036){\rotatebox{300.0}{\makebox(0,0)[lb]{\smash{{\SetFigFont{17}{20.4}{\rmdefault}{\mddefault}{\itdefault}{$y+a_1-1$}%
}}}}}
\put(6696,-3091){\rotatebox{300.0}{\makebox(0,0)[lb]{\smash{{\SetFigFont{17}{20.4}{\rmdefault}{\mddefault}{\itdefault}{$z$}%
}}}}}
\put(5496,-6496){\rotatebox{60.0}{\makebox(0,0)[lb]{\smash{{\SetFigFont{17}{20.4}{\rmdefault}{\mddefault}{\itdefault}{$y+z+2a_2-1$}%
}}}}}
\put(3441, -1){\makebox(0,0)[lb]{\smash{{\SetFigFont{17}{20.4}{\rmdefault}{\mddefault}{\itdefault}{$x+a_2$}%
}}}}
\put(2196,-4966){\rotatebox{90.0}{\makebox(0,0)[lb]{\smash{{\SetFigFont{17}{20.4}{\rmdefault}{\mddefault}{\itdefault}{$y+z+a_1+a_2-1$}%
}}}}}
\put(8976,-8161){\makebox(0,0)[lb]{\smash{{\SetFigFont{17}{20.4}{\rmdefault}{\mddefault}{\itdefault}{$x+a_1$}%
}}}}
\put(11631,-2836){\makebox(0,0)[lb]{\smash{{\SetFigFont{17}{20.4}{\rmdefault}{\mddefault}{\itdefault}{$a_1$}%
}}}}
\put(10476,-3391){\makebox(0,0)[lb]{\smash{{\SetFigFont{17}{20.4}{\rmdefault}{\mddefault}{\itdefault}{$a_2$}%
}}}}
\put(11541,-1156){\rotatebox{300.0}{\makebox(0,0)[lb]{\smash{{\SetFigFont{17}{20.4}{\rmdefault}{\mddefault}{\itdefault}{$y+a_1-1$}%
}}}}}
\put(12666,-3211){\rotatebox{300.0}{\makebox(0,0)[lb]{\smash{{\SetFigFont{17}{20.4}{\rmdefault}{\mddefault}{\itdefault}{$z$}%
}}}}}
\put(11466,-6616){\rotatebox{60.0}{\makebox(0,0)[lb]{\smash{{\SetFigFont{17}{20.4}{\rmdefault}{\mddefault}{\itdefault}{$y+z+2a_2-1$}%
}}}}}
\put(9411,-121){\makebox(0,0)[lb]{\smash{{\SetFigFont{17}{20.4}{\rmdefault}{\mddefault}{\itdefault}{$x+a_2$}%
}}}}
\put(8166,-5086){\rotatebox{90.0}{\makebox(0,0)[lb]{\smash{{\SetFigFont{17}{20.4}{\rmdefault}{\mddefault}{\itdefault}{$y+z+a_1+a_2-1$}%
}}}}}
\put(7701,-23266){\rotatebox{90.0}{\makebox(0,0)[lb]{\smash{{\SetFigFont{17}{20.4}{\rmdefault}{\mddefault}{\itdefault}{$y+z+a_1+a_2-1$}%
}}}}}
\put(8946,-18301){\makebox(0,0)[lb]{\smash{{\SetFigFont{17}{20.4}{\rmdefault}{\mddefault}{\itdefault}{$x+a_2$}%
}}}}
\put(2751,-17146){\makebox(0,0)[lb]{\smash{{\SetFigFont{17}{20.4}{\rmdefault}{\mddefault}{\itdefault}{$x+a_1$}%
}}}}
\put(5406,-11821){\makebox(0,0)[lb]{\smash{{\SetFigFont{17}{20.4}{\rmdefault}{\mddefault}{\itdefault}{$a_1$}%
}}}}
\put(4231,-12351){\makebox(0,0)[lb]{\smash{{\SetFigFont{17}{20.4}{\rmdefault}{\mddefault}{\itdefault}{$a_2$}%
}}}}
\put(5316,-10141){\rotatebox{300.0}{\makebox(0,0)[lb]{\smash{{\SetFigFont{17}{20.4}{\rmdefault}{\mddefault}{\itdefault}{$y+a_1-1$}%
}}}}}
\put(6441,-12196){\rotatebox{300.0}{\makebox(0,0)[lb]{\smash{{\SetFigFont{17}{20.4}{\rmdefault}{\mddefault}{\itdefault}{$z$}%
}}}}}
\put(5241,-15601){\rotatebox{60.0}{\makebox(0,0)[lb]{\smash{{\SetFigFont{17}{20.4}{\rmdefault}{\mddefault}{\itdefault}{$y+z+2a_2-1$}%
}}}}}
\put(3186,-9106){\makebox(0,0)[lb]{\smash{{\SetFigFont{17}{20.4}{\rmdefault}{\mddefault}{\itdefault}{$x+a_2$}%
}}}}
\put(1941,-14071){\rotatebox{90.0}{\makebox(0,0)[lb]{\smash{{\SetFigFont{17}{20.4}{\rmdefault}{\mddefault}{\itdefault}{$y+z+a_1+a_2-1$}%
}}}}}
\put(11001,-24796){\rotatebox{60.0}{\makebox(0,0)[lb]{\smash{{\SetFigFont{17}{20.4}{\rmdefault}{\mddefault}{\itdefault}{$y+z+2a_2-1$}%
}}}}}
\put(12201,-21391){\rotatebox{300.0}{\makebox(0,0)[lb]{\smash{{\SetFigFont{17}{20.4}{\rmdefault}{\mddefault}{\itdefault}{$z$}%
}}}}}
\put(8856,-17191){\makebox(0,0)[lb]{\smash{{\SetFigFont{17}{20.4}{\rmdefault}{\mddefault}{\itdefault}{$x+a_1$}%
}}}}
\put(11511,-11866){\makebox(0,0)[lb]{\smash{{\SetFigFont{17}{20.4}{\rmdefault}{\mddefault}{\itdefault}{$a_1$}%
}}}}
\put(10361,-12396){\makebox(0,0)[lb]{\smash{{\SetFigFont{17}{20.4}{\rmdefault}{\mddefault}{\itdefault}{$a_2$}%
}}}}
\put(11421,-10186){\rotatebox{300.0}{\makebox(0,0)[lb]{\smash{{\SetFigFont{17}{20.4}{\rmdefault}{\mddefault}{\itdefault}{$y+a_1-1$}%
}}}}}
\put(12546,-12241){\rotatebox{300.0}{\makebox(0,0)[lb]{\smash{{\SetFigFont{17}{20.4}{\rmdefault}{\mddefault}{\itdefault}{$z$}%
}}}}}
\put(11346,-15646){\rotatebox{60.0}{\makebox(0,0)[lb]{\smash{{\SetFigFont{17}{20.4}{\rmdefault}{\mddefault}{\itdefault}{$y+z+2a_2-1$}%
}}}}}
\put(9291,-9151){\makebox(0,0)[lb]{\smash{{\SetFigFont{17}{20.4}{\rmdefault}{\mddefault}{\itdefault}{$x+a_2$}%
}}}}
\put(8046,-14116){\rotatebox{90.0}{\makebox(0,0)[lb]{\smash{{\SetFigFont{17}{20.4}{\rmdefault}{\mddefault}{\itdefault}{$y+z+a_1+a_2-1$}%
}}}}}
\put(11076,-19336){\rotatebox{300.0}{\makebox(0,0)[lb]{\smash{{\SetFigFont{17}{20.4}{\rmdefault}{\mddefault}{\itdefault}{$y+a_1-1$}%
}}}}}
\put(10086,-21556){\makebox(0,0)[lb]{\smash{{\SetFigFont{17}{20.4}{\rmdefault}{\mddefault}{\itdefault}{$a_2$}%
}}}}
\put(2541,-26221){\makebox(0,0)[lb]{\smash{{\SetFigFont{17}{20.4}{\rmdefault}{\mddefault}{\itdefault}{$x+a_1$}%
}}}}
\put(5196,-20896){\makebox(0,0)[lb]{\smash{{\SetFigFont{17}{20.4}{\rmdefault}{\mddefault}{\itdefault}{$a_1$}%
}}}}
\put(4046,-21416){\makebox(0,0)[lb]{\smash{{\SetFigFont{17}{20.4}{\rmdefault}{\mddefault}{\itdefault}{$a_2$}%
}}}}
\put(5106,-19216){\rotatebox{300.0}{\makebox(0,0)[lb]{\smash{{\SetFigFont{17}{20.4}{\rmdefault}{\mddefault}{\itdefault}{$y+a_1-1$}%
}}}}}
\put(6231,-21271){\rotatebox{300.0}{\makebox(0,0)[lb]{\smash{{\SetFigFont{17}{20.4}{\rmdefault}{\mddefault}{\itdefault}{$z$}%
}}}}}
\put(5031,-24676){\rotatebox{60.0}{\makebox(0,0)[lb]{\smash{{\SetFigFont{17}{20.4}{\rmdefault}{\mddefault}{\itdefault}{$y+z+2a_2-1$}%
}}}}}
\put(2976,-18181){\makebox(0,0)[lb]{\smash{{\SetFigFont{17}{20.4}{\rmdefault}{\mddefault}{\itdefault}{$x+a_2$}%
}}}}
\put(1731,-23146){\rotatebox{90.0}{\makebox(0,0)[lb]{\smash{{\SetFigFont{17}{20.4}{\rmdefault}{\mddefault}{\itdefault}{$y+z+a_1+a_2-1$}%
}}}}}
\put(11166,-21016){\makebox(0,0)[lb]{\smash{{\SetFigFont{17}{20.4}{\rmdefault}{\mddefault}{\itdefault}{$a_1$}%
}}}}
\put(8511,-26341){\makebox(0,0)[lb]{\smash{{\SetFigFont{17}{20.4}{\rmdefault}{\mddefault}{\itdefault}{$x+a_1$}%
}}}}
\end{picture}}
  \caption{Obtaining the first recurrence on the tiling numbers of $\mathcal{R}$-type regions.}\label{halfhex9}
\end{figure}

\begin{figure}
  \centering
  \setlength{\unitlength}{3947sp}%
\begingroup\makeatletter\ifx\SetFigFont\undefined%
\gdef\SetFigFont#1#2#3#4#5{%
  \reset@font\fontsize{#1}{#2pt}%
  \fontfamily{#3}\fontseries{#4}\fontshape{#5}%
  \selectfont}%
\fi\endgroup%
\resizebox{9cm}{!}{
\begin{picture}(0,0)%
\includegraphics{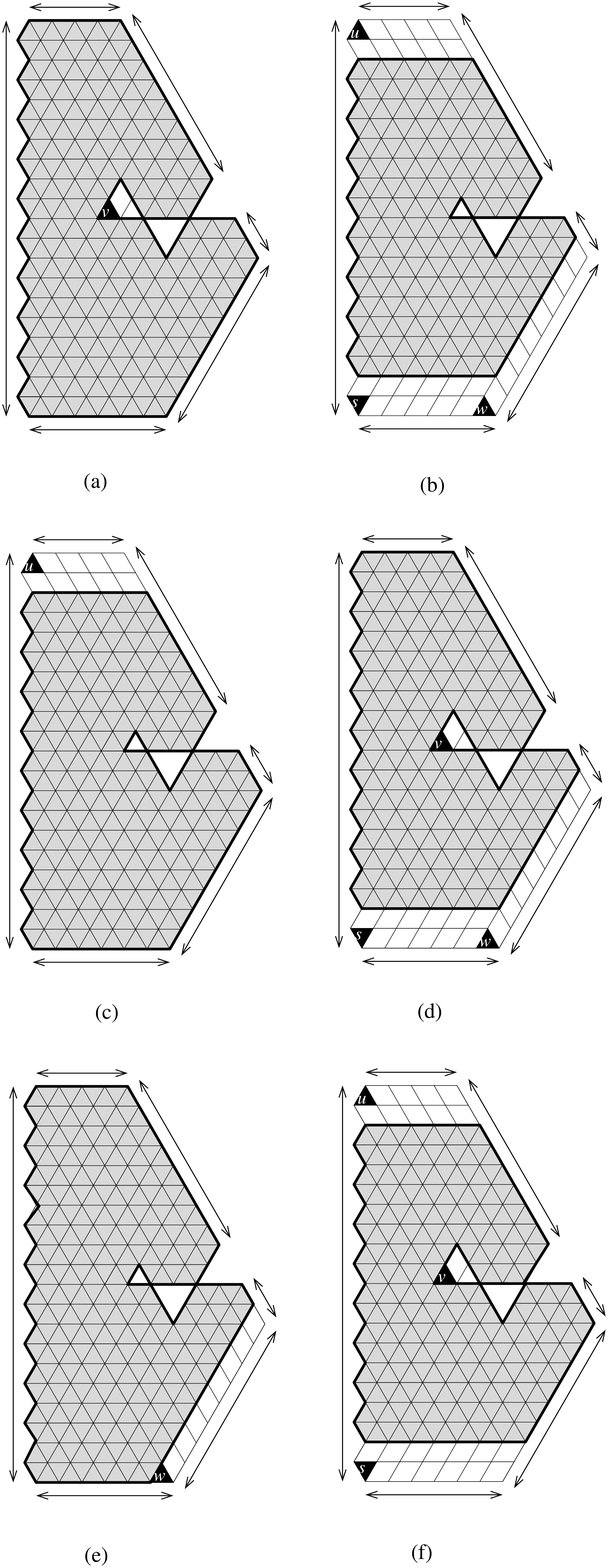}%
\end{picture}%
\begin{picture}(10918,27196)(1739,-26739)
\put(9489,-13066){\makebox(0,0)[lb]{\smash{{\SetFigFont{17}{20.4}{\rmdefault}{\mddefault}{\itdefault}{$a_3$}%
}}}}
\put(4681,-3976){\makebox(0,0)[lb]{\smash{{\SetFigFont{17}{20.4}{\rmdefault}{\mddefault}{\itdefault}{$a_2$}%
}}}}
\put(5556,-3586){\makebox(0,0)[lb]{\smash{{\SetFigFont{17}{20.4}{\rmdefault}{\mddefault}{\itdefault}{$a_1$}%
}}}}
\put(2886,-7801){\makebox(0,0)[lb]{\smash{{\SetFigFont{17}{20.4}{\rmdefault}{\mddefault}{\itdefault}{$x+a_1+a_3$}%
}}}}
\put(4824,-629){\rotatebox{300.0}{\makebox(0,0)[lb]{\smash{{\SetFigFont{17}{20.4}{\rmdefault}{\mddefault}{\itdefault}{$y+a_1+2a_3-1$}%
}}}}}
\put(6561,-3841){\rotatebox{300.0}{\makebox(0,0)[lb]{\smash{{\SetFigFont{17}{20.4}{\rmdefault}{\mddefault}{\itdefault}{$z$}%
}}}}}
\put(5811,-6811){\rotatebox{60.0}{\makebox(0,0)[lb]{\smash{{\SetFigFont{17}{20.4}{\rmdefault}{\mddefault}{\itdefault}{$y+z+2a_2-1$}%
}}}}}
\put(2811,119){\makebox(0,0)[lb]{\smash{{\SetFigFont{17}{20.4}{\rmdefault}{\mddefault}{\itdefault}{$x+a_2$}%
}}}}
\put(2001,-5251){\rotatebox{90.0}{\makebox(0,0)[lb]{\smash{{\SetFigFont{17}{20.4}{\rmdefault}{\mddefault}{\itdefault}{$y+z+a_1+a_2+a_3-1$}%
}}}}}
\put(3729,-4066){\makebox(0,0)[lb]{\smash{{\SetFigFont{17}{20.4}{\rmdefault}{\mddefault}{\itdefault}{$a_3$}%
}}}}
\put(10284,-3976){\makebox(0,0)[lb]{\smash{{\SetFigFont{17}{20.4}{\rmdefault}{\mddefault}{\itdefault}{$a_2$}%
}}}}
\put(9309,-3999){\makebox(0,0)[lb]{\smash{{\SetFigFont{17}{20.4}{\rmdefault}{\mddefault}{\itdefault}{$a_3-1$}%
}}}}
\put(11166,-3571){\makebox(0,0)[lb]{\smash{{\SetFigFont{17}{20.4}{\rmdefault}{\mddefault}{\itdefault}{$a_1$}%
}}}}
\put(8496,-7786){\makebox(0,0)[lb]{\smash{{\SetFigFont{17}{20.4}{\rmdefault}{\mddefault}{\itdefault}{$x+a_1+a_3$}%
}}}}
\put(10434,-614){\rotatebox{300.0}{\makebox(0,0)[lb]{\smash{{\SetFigFont{17}{20.4}{\rmdefault}{\mddefault}{\itdefault}{$y+a_1+2a_3-1$}%
}}}}}
\put(12171,-3826){\rotatebox{300.0}{\makebox(0,0)[lb]{\smash{{\SetFigFont{17}{20.4}{\rmdefault}{\mddefault}{\itdefault}{$z$}%
}}}}}
\put(11421,-6796){\rotatebox{60.0}{\makebox(0,0)[lb]{\smash{{\SetFigFont{17}{20.4}{\rmdefault}{\mddefault}{\itdefault}{$y+z+2a_2-1$}%
}}}}}
\put(8421,134){\makebox(0,0)[lb]{\smash{{\SetFigFont{17}{20.4}{\rmdefault}{\mddefault}{\itdefault}{$x+a_2$}%
}}}}
\put(7611,-5236){\rotatebox{90.0}{\makebox(0,0)[lb]{\smash{{\SetFigFont{17}{20.4}{\rmdefault}{\mddefault}{\itdefault}{$y+z+a_1+a_2+a_3-1$}%
}}}}}
\put(4731,-13021){\makebox(0,0)[lb]{\smash{{\SetFigFont{17}{20.4}{\rmdefault}{\mddefault}{\itdefault}{$a_2$}%
}}}}
\put(3804,-13111){\makebox(0,0)[lb]{\smash{{\SetFigFont{17}{20.4}{\rmdefault}{\mddefault}{\itdefault}{$a_3-1$}%
}}}}
\put(5616,-12661){\makebox(0,0)[lb]{\smash{{\SetFigFont{17}{20.4}{\rmdefault}{\mddefault}{\itdefault}{$a_1$}%
}}}}
\put(2946,-16876){\makebox(0,0)[lb]{\smash{{\SetFigFont{17}{20.4}{\rmdefault}{\mddefault}{\itdefault}{$x+a_1+a_3$}%
}}}}
\put(4884,-9704){\rotatebox{300.0}{\makebox(0,0)[lb]{\smash{{\SetFigFont{17}{20.4}{\rmdefault}{\mddefault}{\itdefault}{$y+a_1+2a_3-1$}%
}}}}}
\put(6621,-12916){\rotatebox{300.0}{\makebox(0,0)[lb]{\smash{{\SetFigFont{17}{20.4}{\rmdefault}{\mddefault}{\itdefault}{$z$}%
}}}}}
\put(5871,-15886){\rotatebox{60.0}{\makebox(0,0)[lb]{\smash{{\SetFigFont{17}{20.4}{\rmdefault}{\mddefault}{\itdefault}{$y+z+2a_2-1$}%
}}}}}
\put(2871,-8956){\makebox(0,0)[lb]{\smash{{\SetFigFont{17}{20.4}{\rmdefault}{\mddefault}{\itdefault}{$x+a_2$}%
}}}}
\put(2061,-14326){\rotatebox{90.0}{\makebox(0,0)[lb]{\smash{{\SetFigFont{17}{20.4}{\rmdefault}{\mddefault}{\itdefault}{$y+z+a_1+a_2+a_3-1$}%
}}}}}
\put(7731,-23401){\rotatebox{90.0}{\makebox(0,0)[lb]{\smash{{\SetFigFont{17}{20.4}{\rmdefault}{\mddefault}{\itdefault}{$y+z+a_1+a_2+a_3-1$}%
}}}}}
\put(8541,-18031){\makebox(0,0)[lb]{\smash{{\SetFigFont{17}{20.4}{\rmdefault}{\mddefault}{\itdefault}{$x+a_2$}%
}}}}
\put(10344,-13036){\makebox(0,0)[lb]{\smash{{\SetFigFont{17}{20.4}{\rmdefault}{\mddefault}{\itdefault}{$a_2$}%
}}}}
\put(11226,-12646){\makebox(0,0)[lb]{\smash{{\SetFigFont{17}{20.4}{\rmdefault}{\mddefault}{\itdefault}{$a_1$}%
}}}}
\put(8556,-16861){\makebox(0,0)[lb]{\smash{{\SetFigFont{17}{20.4}{\rmdefault}{\mddefault}{\itdefault}{$x+a_1+a_3$}%
}}}}
\put(10494,-9689){\rotatebox{300.0}{\makebox(0,0)[lb]{\smash{{\SetFigFont{17}{20.4}{\rmdefault}{\mddefault}{\itdefault}{$y+a_1+2a_3-1$}%
}}}}}
\put(12231,-12901){\rotatebox{300.0}{\makebox(0,0)[lb]{\smash{{\SetFigFont{17}{20.4}{\rmdefault}{\mddefault}{\itdefault}{$z$}%
}}}}}
\put(11481,-15871){\rotatebox{60.0}{\makebox(0,0)[lb]{\smash{{\SetFigFont{17}{20.4}{\rmdefault}{\mddefault}{\itdefault}{$y+z+2a_2-1$}%
}}}}}
\put(8481,-8941){\makebox(0,0)[lb]{\smash{{\SetFigFont{17}{20.4}{\rmdefault}{\mddefault}{\itdefault}{$x+a_2$}%
}}}}
\put(7671,-14311){\rotatebox{90.0}{\makebox(0,0)[lb]{\smash{{\SetFigFont{17}{20.4}{\rmdefault}{\mddefault}{\itdefault}{$y+z+a_1+a_2+a_3-1$}%
}}}}}
\put(11541,-24961){\rotatebox{60.0}{\makebox(0,0)[lb]{\smash{{\SetFigFont{17}{20.4}{\rmdefault}{\mddefault}{\itdefault}{$y+z+2a_2-1$}%
}}}}}
\put(4791,-22111){\makebox(0,0)[lb]{\smash{{\SetFigFont{17}{20.4}{\rmdefault}{\mddefault}{\itdefault}{$a_2$}%
}}}}
\put(3811,-22186){\makebox(0,0)[lb]{\smash{{\SetFigFont{17}{20.4}{\rmdefault}{\mddefault}{\itdefault}{$a_3-1$}%
}}}}
\put(5676,-21751){\makebox(0,0)[lb]{\smash{{\SetFigFont{17}{20.4}{\rmdefault}{\mddefault}{\itdefault}{$a_1$}%
}}}}
\put(3006,-25966){\makebox(0,0)[lb]{\smash{{\SetFigFont{17}{20.4}{\rmdefault}{\mddefault}{\itdefault}{$x+a_1+a_3$}%
}}}}
\put(4944,-18794){\rotatebox{300.0}{\makebox(0,0)[lb]{\smash{{\SetFigFont{17}{20.4}{\rmdefault}{\mddefault}{\itdefault}{$y+a_1+2a_3-1$}%
}}}}}
\put(6681,-22006){\rotatebox{300.0}{\makebox(0,0)[lb]{\smash{{\SetFigFont{17}{20.4}{\rmdefault}{\mddefault}{\itdefault}{$z$}%
}}}}}
\put(5931,-24976){\rotatebox{60.0}{\makebox(0,0)[lb]{\smash{{\SetFigFont{17}{20.4}{\rmdefault}{\mddefault}{\itdefault}{$y+z+2a_2-1$}%
}}}}}
\put(2931,-18046){\makebox(0,0)[lb]{\smash{{\SetFigFont{17}{20.4}{\rmdefault}{\mddefault}{\itdefault}{$x+a_2$}%
}}}}
\put(2121,-23416){\rotatebox{90.0}{\makebox(0,0)[lb]{\smash{{\SetFigFont{17}{20.4}{\rmdefault}{\mddefault}{\itdefault}{$y+z+a_1+a_2+a_3-1$}%
}}}}}
\put(12291,-21991){\rotatebox{300.0}{\makebox(0,0)[lb]{\smash{{\SetFigFont{17}{20.4}{\rmdefault}{\mddefault}{\itdefault}{$z$}%
}}}}}
\put(10554,-18779){\rotatebox{300.0}{\makebox(0,0)[lb]{\smash{{\SetFigFont{17}{20.4}{\rmdefault}{\mddefault}{\itdefault}{$y+a_1+2a_3-1$}%
}}}}}
\put(8616,-25951){\makebox(0,0)[lb]{\smash{{\SetFigFont{17}{20.4}{\rmdefault}{\mddefault}{\itdefault}{$x+a_1+a_3$}%
}}}}
\put(10419,-22126){\makebox(0,0)[lb]{\smash{{\SetFigFont{17}{20.4}{\rmdefault}{\mddefault}{\itdefault}{$a_2$}%
}}}}
\put(11286,-21736){\makebox(0,0)[lb]{\smash{{\SetFigFont{17}{20.4}{\rmdefault}{\mddefault}{\itdefault}{$a_1$}%
}}}}
\put(9519,-22171){\makebox(0,0)[lb]{\smash{{\SetFigFont{17}{20.4}{\rmdefault}{\mddefault}{\itdefault}{$a_3$}%
}}}}
\end{picture}}
  \caption{Obtaining the second recurrence for number tilings of $\mathcal{R}$-type regions.}\label{halfhex10}
\end{figure}

Consider the region corresponding to $G-\{u,v,w,s\}$ (see Figure \ref{halfhex9}(b)). The removal of the four black triangles, that correspond to the four vertices $u,v,w,s$, yields several  forced lozenges. By removing these forced lozenges, we get a the region $\mathcal{R}_{x,y-2,z-1}(a_1,a_2,\dotsc,a_{2l}+1)$ whose tiling number is the same as that of the original one. Thus, we have
\begin{equation}
\M(G-\{u,v,w,s\})=\M(\mathcal{R}_{x,y-2,z-1}(a_1,a_2,\dotsc,a_{2l}+1)).
\end{equation}
Consider forced lozenges shown in Figures \ref{halfhex9}(c)--(f), we obtain respectively
\begin{equation}
\M(G-\{u,v\})=\M(\mathcal{R}_{x,y,z-1}(a_1,a_2,\dotsc,a_{2l})),
\end{equation}
\begin{equation}
\M(G-\{w,s\})=\M(\mathcal{R}_{x,y-2,z}(a_1,a_2,\dotsc,a_{2l}+1)),
\end{equation}
\begin{equation}
\M(G-\{u,s\})=\M(\mathcal{R}_{x+1,y-2,z}(a_1,a_2,\dotsc,a_{2l})),
\end{equation}
and
\begin{equation}
\M(G-\{v,w\})=\M(\mathcal{R}_{x-1,y,z-1}(a_1,a_2,\dotsc,a_{2l}+1)).
\end{equation}

Plugging all the above equalities into the recurrence in Kuo's Theorem \ref{kuothm1}, we get the first recurrence
\begin{align}\label{Kuorecur1}
\M(\mathcal{R}_{x,y,z}(a_1,a_2,\dotsc,a_{2l}))&\M(\mathcal{R}_{x,y-2,z-1}(a_1,a_2,\dotsc,a_{2l}+1))=\notag\\&\M(\mathcal{R}_{x,y,z-1}(a_1,a_2,\dotsc,a_{2l}))\M(\mathcal{R}_{x,y-2,z}(a_1,a_2,\dotsc,a_{2l}+1))\notag\\
&+\M(\mathcal{R}_{x+1,y-2,z}(a_1,a_2,\dotsc,a_{2l}))\M(\mathcal{R}_{x-1,y,z-1}(a_1,a_2,\dotsc,a_{2l}+1)).
\end{align}

\medskip

To obtain the second recurrence, we apply Kuo's Theorem \ref{kuothm2} to the region obtained from the region  $\mathcal{R}_{x,y,z,t}(a_1,a_2,\dotsc,a_{2l+1})$ by adding a band of unit triangles along the left side of the $a_{2l+1}$-triangle in the array. The four vertices $u,v,w,s$ are chosen as in Figure  \ref{halfhex7}(b). Figure \ref{halfhex10} tells us that the product of the tiling numbers of the two regions on the top row equals  the product of the tiling numbers of the two regions on the middle row, plus the product of the tiling numbers of the two regions on the bottom row. Precisely, we have
\begin{align}\label{Kuorecur2}
\M(\mathcal{R}_{x,y,z}(a_1,a_2,\dotsc,a_{2l+1}))&\M(\mathcal{R}_{x+1,y,z-1}(a_1,a_2,\dotsc,a_{2l+1}-1))=\notag\\&\M(\mathcal{R}_{x+1,y,z}(a_1,a_2,\dotsc,a_{2l+1}-1))\M(\mathcal{R}_{x,y,z-1}(a_1,a_2,\dotsc,a_{2l+1}))\notag\\
&+\M(\mathcal{R}_{x,y+2,z-1}(a_1,a_2,\dotsc,a_{2l+1}-1))\M(\mathcal{R}_{x+1,y-2,z}(a_1,a_2,\dotsc,a_{2l+1})).
\end{align}

To finish the proof we only need to verify that the both formulas (\ref{ee}) and (\ref{oe}) also satisfy the same recurrences (\ref{Kuorecur1}) and (\ref{Kuorecur2}). However, this verification will be carried out in the next section.
\end{proof}

\section{Verifying that formulas  (\ref{ee}) and (\ref{oe}) satisfy recurrences (\ref{Kuorecur1}) and (\ref{Kuorecur2})}\label{Recurrence}
This section is devoted to the verification that the both formulas (\ref{ee}) and (\ref{oe}) satisfy the recurrences (\ref{Kuorecur1}) and (\ref{Kuorecur2}).

Assume that $x,y,z,n$ are nonnegative integers and  that $\textbf{a}=\{a_i\}_{i=1}^{n}$ is a sequence of positive integers. Define the function $\Phi_{x,y,z}(\textbf{a})$  to be the expression on the right-hand side of (\ref{ee}) if $y$ is odd,  and to be the expression on the right-hand side of (\ref{oe}) if $y$ is even.

\begin{lem}\label{numberrecur} Assume that $x,y,z,n$ are positive integers and that $\textbf{a}=\{a_i\}_{i=1}^{n}$ is a sequence of positive integers. For $n$ is even, we have
\begin{align}\label{numberrecur1}
\Phi_{x,y,z}&(a_1,a_2,\dotsc,a_{n})\Phi_{x,y-2,z-1}(a_1,a_2,\dotsc,a_{n}+1)=\notag\\&\Phi_{x,y,z-1}(a_1,a_2,\dotsc,a_{n})\Phi_{x,y-2,z}(a_1,a_2,\dotsc,a_{n}+1)\notag\\
&+\Phi_{x+1,y-2,z}(a_1,a_2,\dotsc,a_{n})\Phi_{x-1,y,z-1}(a_1,a_2,\dotsc,a_{n}+1).
\end{align}
For odd $n$
\begin{align}\label{numberrecur2}
\Phi_{x,y,z}&(a_1,a_2,\dotsc,a_{n})\Phi_{x+1,y,z-1}(a_1,a_2,\dotsc,a_{n}-1))=\notag\\&\Phi_{x+1,y,z}(a_1,a_2,\dotsc,a_{n}-1))\Phi_{x,y,z-1}(a_1,a_2,\dotsc,a_{n}))\notag\\
&+\Phi_{x,y+2,z-1}(a_1,a_2,\dotsc,a_{n}-1)\Phi_{x+1,y-2,z}(a_1,a_2,\dotsc,a_{n}),
\end{align}
where \[\Phi_{x,y,z}(a_1,a_2,\dotsc,a_{n-1},0):=\Phi_{x,y,z}(a_1,a_2,\dotsc,a_{n-1})\] by convention.
\end{lem}

\begin{proof}
We define three component functions $f,g,h$ of the $\Phi$-function as follows.
\begin{equation}
f_{x,y,z}(\textbf{a}):=\prod_{i=1}^{\left\lfloor\frac{y}{2}\right\rfloor}(2x+2i)_{2\s_n(\textbf{a})+2y+2z-4i+1}.
\end{equation}
 The $g$-function is defined as
\begin{equation}
g_{y,z}(\textbf{a}):=\frac{1}{2^{y-1}}\frac{\Hf(\s_n(\textbf{a})+y+z-1)\Hf_2(y)\Hf_2(2\E(\textbf{a})+2z+1)\Hf_2(2\Od(\textbf{a})+1)\Hf_2(2\s_n(\textbf{a})+y+2z)}
{\Hf(\s_n(\textbf{a})+z)\Hf_2(2\E(\textbf{a})+y+2z)\Hf_2(2\Od(\textbf{a})+y)\Hf_2(2\s_n(\textbf{a})+2y+2z-1)}
\end{equation}
when $y$ is odd, and
\begin{equation}
g_{y,z}(\textbf{a}):=\frac{\Hf(\s_n(\textbf{a})+y+z)\Hf_2(y)\Hf_2(2\E(\textbf{a})+2z)\Hf_2(2\Od(\textbf{a}))\Hf_2(2\s_n(\textbf{a})+y+2z)}
{\Hf(\s_n(\textbf{a})+z)\Hf_2(2\E(\textbf{a})+y+2z)\Hf_2(2\Od(\textbf{a})+y)\Hf_2(2\s_n(\textbf{a})+2y+2z)}
\end{equation}
when $y$ is even. The function $h$ is also defined based on the parity of $y$ as:
\begin{equation}
h_{x,y,z}(\textbf{a}):=\Q\left(x+\frac{y-1}{2}+a_{2\lfloor\frac{n+1}{2}\rfloor},a_{2\lfloor\frac{n+1}{2}\rfloor-1},\dotsc,a_1\right)\Q\left(x+\frac{y-1}{2}+a_{2\lfloor\frac{n}{2}\rfloor+1},\dotsc,a_1,z\right)
\end{equation}
if $y$ is odd, and
\begin{equation}
h_{x,y,z}(\textbf{a}):=\K\left(x+\frac{y}{2}+a_{2\lfloor\frac{n+1}{2}\rfloor},a_{2\lfloor\frac{n+1}{2}\rfloor-1},\dotsc,a_1\right)\K\left(x+\frac{y}{2}+a_{2\lfloor\frac{n}{2}\rfloor+1},a_{2\lfloor\frac{n}{2}\rfloor},\dotsc,a_1,z\right)
\end{equation}
if $y$ is even. Therefore we always have \[\Phi_{x,y,z}(\textbf{a})=f_{x,y,z}(\textbf{a})g_{y,z}(\textbf{a})h_{x,y,z}(\textbf{a}).\]

\medskip

To prove  (\ref{numberrecur1}), we need to show that
\begin{align}\label{numberrecur1b}
\frac{\Phi_{x,y,z-1}(a_1,a_2,\dotsc,a_{2l})}{\Phi_{x,y-2,z-1}(a_1,a_2,\dotsc,a_{2l}+1)}&\frac{\Phi_{x,y-2,z}(a_1,a_2,\dotsc,a_{2l}+1)}{\Phi_{x,y,z}(a_1,a_2,\dotsc,a_{2l})}\notag\\
&+\frac{\Phi_{x+1,y-2,z}(a_1,a_2,\dotsc,a_{2l})}{\Phi_{x,y,z}(a_1,a_2,\dotsc,a_{2l})}\frac{\Phi_{x-1,y,z-1}(a_1,a_2,\dotsc,a_{2l}+1)}{\Phi_{x,y-2,z-1}(a_1,a_2,\dotsc,a_{2l}+1)}=1,
\end{align}
for any positive integer $l$ (i.e. $n=2l$).

The first term on the left-hand side of (\ref{numberrecur1b}) can be written as the product of three similar terms for the component functions $f,g,h$ (called the $f$-, $g$-, and $h$- \emph{factors}, respectively).
\begin{align}\label{factoreq1}
\frac{\Phi_{x,y-2,z-1}(a_1,a_2,\dotsc,a_{2l}+1)}{\Phi_{x,y,z-1}(a_1,a_2,\dotsc,a_{2l})}&\frac{\Phi_{x,y,z}(a_1,a_2,\dotsc,a_{2l})}{\Phi_{x,y-2,z}(a_1,a_2,\dotsc,a_{2l}+1)}=\notag\\
&\frac{f_{x,y-2,z-1}(a_1,a_2,\dotsc,a_{2l}+1)}{f_{x,y,z-1}(a_1,a_2,\dotsc,a_{2l})}\frac{f_{x,y,z}(a_1,a_2,\dotsc,a_{2l})}{f_{x,y-2,z}(a_1,a_2,\dotsc,a_{2l}+1)}\notag\\
&\times\frac{g_{y-2,z-1}(a_1,a_2,\dotsc,a_{2l}+1)}{g_{y,z-1}(a_1,a_2,\dotsc,a_{2l})}\frac{g_{y,z}(a_1,a_2,\dotsc,a_{2l})}{g_{y-2,z}(a_1,a_2,\dotsc,a_{2l}+1)}\notag\\
&\times\frac{h_{x,y-2,z-1}(a_1,a_2,\dotsc,a_{2l}+1)}{h_{x,y,z-1}(a_1,a_2,\dotsc,a_{2l})}\frac{h_{x,y,z}(a_1,a_2,\dotsc,a_{2l})}{h_{x,y-2,z}(a_1,a_2,\dotsc,a_{2l}+1)}.
\end{align}

\medskip

We consider first the case when $y$ is odd.

\medskip

In the rest of this proof, we use the shorthand notations $\s:=\sum_{i=1}^{n}a_i$, $\od:=\sum_{\text{$i$ odd}}a_i$ and $\e:=\sum_{\text{$i$ even}}a_i$.

When $y$ is odd, the $f$-factor on the right-hand side of (\ref{factoreq1}) can be rewritten as
\begin{align}\label{component1}
&\frac{f_{x,y-2,z-1}(a_1,a_2,\dotsc,a_{2l}+1)}{f_{x,y,z-1}(a_1,a_2,\dotsc,a_{2l})}\frac{f_{x,y,z}(a_1,a_2,\dotsc,a_{2l})}{f_{x,y-2,z}(a_1,a_2,\dotsc,a_{2l}+1)}\notag\\
&\quad\quad\quad\quad=\frac{\prod_{i=1}^{\frac{y-1}{2}-1}(2x+2i)_{2\s+2y+2z-4i-1}\prod_{i=1}^{\frac{y-1}{2}}(2x+2i)_{2\s+2y+2z-4i-1}}
{\prod_{i=1}^{\frac{y-1}{2}}(2x+2i)_{2\s+2y+2z-4i+1}\prod_{i=1}^{\frac{y-1}{2}-1}(2x+2i)_{2\s+2y+2z-4i-3}}\notag\\
&\quad\quad\quad\quad=\prod_{i=1}^{\frac{y-1}{2}}\frac{(2x+2i)_{2\s+2y+2z-4i-1}}{(2x+2i)_{2\s+2y+2z-4i+1}}\prod_{i=1}^{\frac{y-1}{2}-1}\frac{(2x+2i)_{2\s+2y+2z-4i-1}}{(2x+2i)_{2\s+2y+2z-4i-3}}\notag\\
&\quad\quad\quad\quad=\frac{1}{(2x+2\s+2y+2z-3)_2}.
\end{align}
We can also simplify the $g$-factor in (\ref{factoreq1}) as
\begin{align}\label{component2}
&\frac{g_{y-2,z-1}(a_1,a_2,\dotsc,a_{2l}+1)}{g_{y,z-1}(a_1,a_2,\dotsc,a_{2l})}\frac{g_{y,z}(a_1,a_2,\dotsc,a_{2l})}{g_{y-2,z}(a_1,a_2,\dotsc,a_{2l}+1)}=\notag\\
&\quad\quad\quad\quad=\frac{(2\e+2z)(2\e+2z+1)(2\s+2y+2z-3)(2\s+2y+2z-4)}{(\s+y+z-2)(\s+z)}.
\end{align}
By removing common factors in the numerator and denominator of the $h$-factor in (\ref{factoreq1}), we have
\begin{align}
&\frac{h_{x,y-2,z-1}(a_1,a_2,\dotsc,a_{2l}+1)}{h_{x,y,z-1}(a_1,a_2,\dotsc,a_{2l})}\frac{h_{x,y,z}(a_1,a_2,\dotsc,a_{2l})}{h_{x,y-2,z}(a_1,a_2,\dotsc,a_{2l}+1)}=\notag\\
&\quad\quad\quad\quad=\frac{\Q\left(x+\frac{y-1}{2},a_{2l},\dotsc,a_1,z-1\right)}{\Q\left(x+\frac{y-1}{2},a_{2l},\dotsc,a_1,z\right)}\frac{\Q\left(x+\frac{y-1}{2}-1,a_{2l}+1,\dotsc,a_1,z\right)}{\Q\left(x+\frac{y-1}{2}-1,a_{2l}+1,\dotsc,a_1,z-1\right)}.
\end{align}

We have the following claim as a direct consequence of Lemma \ref{QAR}:

\begin{claim} For any sequence of nonnegative integers $\textbf{t}=(t_1,t_2,\dotsc,t_{2l})$.
\begin{align}
\frac{\Q(t_1,\dots,t_{2l}+1)}{\Q(t_1,\dots,t_{2l})}=&\frac{(\s_{2l}(\textbf{t})+1)}{(2\E(\textbf{t})+1)!(2\s_{2l}(\textbf{t})+1)!}\prod_{i=1}^{l}\frac{(\s_{2l}(\textbf{t})-\s_{2i-1}(\textbf{t}))!(\s_{2l}(\textbf{t})+\s_{2i}(\textbf{t})+1)!}{(\s_{2l}(\textbf{t})-\s_{2i}(\textbf{t}))!(\s_{2l}(\textbf{t})+\s_{2i-1}(\textbf{t})+1)!}.
\end{align}
\end{claim}

Applying the above claim, our $h$-factor can be simplify further as
\begin{align}\label{component3}
&\frac{h_{x,y-2,z-1}(a_1,a_2,\dotsc,a_{2l}+1)}{h_{x,y,z-1}(a_1,a_2,\dotsc,a_{2l})}\frac{h_{x,y,z}(a_1,a_2,\dotsc,a_{2l})}{h_{x,y-2,z}(a_1,a_2,\dotsc,a_{2l}+1)}=\notag\\
&\quad\quad\quad\quad=\frac{(2x+\s+y+z-1)(\s+z)}{(2\e+2z+1)(2\e+2z)}.
\end{align}

From (\ref{component1}), (\ref{component2}), and (\ref{component3}), we simplify the first term on the left-hand side of (\ref{numberrecur1b}) as:
\begin{align}\label{recur1eq1}
\frac{\Phi_{x,y-2,z-1}(a_1,a_2,\dotsc,a_{2l}+1)}{\Phi_{x,y,z-1}(a_1,a_2,\dotsc,a_{2l})}&\frac{\Phi_{x,y,z}(a_1,a_2,\dotsc,a_{2l})}{\Phi_{x,y-2,z}(a_1,a_2,\dotsc,a_{2l}+1)}=
\frac{2(2x+\s+y+z-1)(2\s+2y+2z-3)}{(2x+2\s+2y+2z-3)_2}.
\end{align}

To simplify the second term on the left-hand side of (\ref{numberrecur1b}), we also write it as the product of three factors as
\begin{align}\label{factoreq2}
&\frac{\Phi_{x+1,y-2,z}(a_1,a_2,\dotsc,a_{2l})}{\Phi_{x,y,z}(a_1,a_2,\dotsc,a_{2l})}\frac{\Phi_{x-1,y,z-1}(a_1,a_2,\dotsc,a_{2l}+1)}{\Phi_{x,y-2,z-1}(a_1,a_2,\dotsc,a_{2l}+1)}=\notag\\
&\quad\quad\quad\quad\frac{f_{x+1,y-2,z}(a_1,a_2,\dotsc,a_{2l})}{f_{x,y,z}(a_1,a_2,\dotsc,a_{2l})}\frac{f_{x-1,y,z-1}(a_1,a_2,\dotsc,a_{2l}+1)}{f_{x,y-2,z-1}(a_1,a_2,\dotsc,a_{2l}+1)}\notag\\
&\quad\quad\quad\quad \times\frac{g_{y-2,z}(a_1,a_2,\dotsc,a_{2l})}{g_{y,z}(a_1,a_2,\dotsc,a_{2l})}\frac{g_{y,z-1}(a_1,a_2,\dotsc,a_{2l}+1)}{g_{y-2,z-1}(a_1,a_2,\dotsc,a_{2l}+1)}\notag\\
&\quad\quad\quad\quad \times\frac{h_{x+1,y-2,z}(a_1,a_2,\dotsc,a_{2l})}{h_{x,y,z}(a_1,a_2,\dotsc,a_{2l})}\frac{h_{x-1,y,z-1}(a_1,a_2,\dotsc,a_{2l}+1)}{h_{x,y-2,z-1}(a_1,a_2,\dotsc,a_{2l}+1)}.
\end{align}
Next, we will simplify each of the three factors on the right-hand side of (\ref{factoreq2}). The first $f$-factor is simplified as
\begin{align}
&\frac{f_{x+1,y-2,z}(a_1,a_2,\dotsc,a_{2l})}{f_{x,y,z}(a_1,a_2,\dotsc,a_{2l})}\frac{f_{x-1,y,z-1}(a_1,a_2,\dotsc,a_{2l}+1)}{f_{x,y-2,z-1}(a_1,a_2,\dotsc,a_{2l}+1)}\notag\\
&\quad=\frac{\prod_{i=1}^{\frac{y-1}{2}-1}(2x+2(i+1))_{2\s+2y+2z-4(i+1)+1}\prod_{i=1}^{\frac{y-1}{2}}(2x+2i-2)_{2\s+2y+2z-4i+1}}
{\prod_{i=1}^{\frac{y-1}{2}}(2x+2i)_{2\s+2y+2z-4i+1}\prod_{i=1}^{\frac{y-1}{2}-1}(2x+2(i+1)-2)_{2\s+2y+2z-4(i+1)+1}}\notag\\
&\quad=\frac{\prod_{i=2}^{\frac{y-1}{2}}(2x+2i)_{2\s+2y+2z-4i+1}\prod_{i=1}^{\frac{y-1}{2}}(2x+2i-2)_{2\s+2y+2z-4i+1}}
{\prod_{i=1}^{\frac{y-1}{2}}(2x+2i)_{2\s+2y+2z-4i+1}\prod_{i=2}^{\frac{y-1}{2}}(2x+2i-2)_{2\s+2y+2z-4i+1}}\notag\\
&\quad=\frac{(2x)_{2\s+2y+2z-3}}{(2x+2)_{2\s+2y+2z-3}}=\frac{(2x)(2x+1)}{(2x+2\s+2y+2z-1)_2}.
\end{align}
It is easy to verify, by definition, that the $g$-factor on the rihgt-hand side of (\ref{factoreq2}) can be reducible to $1$ in this case. In the $h$-factor, both its numerator and denominator are equal to
\begin{align}
&\Q\left(x+\frac{y-1}{2}+a_{2l},\dotsc,a_1\right)\Q\left(x+\frac{y-1}{2},a_{2l},\dotsc,a_1,z\right)\notag\\
&\qquad\times\Q\left(x+\frac{y-1}{2}+a_{2l},\dotsc,a_1\right)\Q\left(x+\frac{y-1}{2}-1,a_{2l}+1,\dotsc,a_1,z-1\right).
\end{align}
This means that the $h$-factor equals $1$. Therefore, the second term on the left-hand side of (\ref{numberrecur1b}) can be simplified as
\begin{align}\label{recur1eq2}
\frac{\Phi_{x+1,y-2,z}(a_1,a_2,\dotsc,a_{2l})}{\Phi_{x,y,z}(a_1,a_2,\dotsc,a_{2l})}\frac{\Phi_{x-1,y,z-1}(a_1,a_2,\dotsc,a_{2l}+1)}{\Phi_{x,y-2,z-1}(a_1,a_2,\dotsc,a_{2l}+1)}
=\frac{(2x)(2x+1)}{(2x+2\s+2y+2z-3)_2}.
\end{align}
By (\ref{recur1eq1}) and (\ref{recur1eq2}), the identity  (\ref{numberrecur1b})  is equivalent to
\begin{equation}
\frac{2(2x+\s+y+z-1)(2\s+2y+2z-3)}{(2x+2\s+2y+2z-3)_2}+\frac{(2x)(2x+1)}{(2x+2\s+2y+2z-3)_2}=1,
\end{equation}
which is easy to be verified.

\bigskip

Next, we verify the recurrence  (\ref{numberrecur1b}) for even $y$. We also write the first term on the left-hand side the recurrence as the product of three component factors. Similar to the case of odd $y$, the $f$-factor can be simplified as
\begin{align}\label{component4}
\frac{f_{x,y-2,z-1}(a_1,a_2,\dotsc,a_{2l}+1)}{f_{x,y,z-1}(a_1,a_2,\dotsc,a_{2l})}\frac{f_{x,y,z}(a_1,a_2,\dotsc,a_{2l})}{f_{x,y-2,z}(a_1,a_2,\dotsc,a_{2l}+1)}=\frac{1}{(2x+2\s+2y+2z-3)_2}.
\end{align}
The $g$-factor becomes
\begin{align}\label{component5}
&\frac{g_{y-2,z-1}(a_1,a_2,\dotsc,a_{2l}+1)}{g_{y,z-1}(a_1,a_2,\dotsc,a_{2l})}\frac{g_{y,z}(a_1,a_2,\dotsc,a_{2l})}{g_{y-2,z}(a_1,a_2,\dotsc,a_{2l}+1)}=\notag\\
&\quad\quad\quad\quad=\frac{(2\e+2z-1)(2\e+2z)(2\s+2y+2z-3)(2\s+2y+2z-2)}{(\s+y+z-1)(\s+z)}.
\end{align}
For simplifying the $h$-factor, we need the following claim from Lemma \ref{QAR}

\begin{claim} For any sequence of nonnegative integers $\textbf{t}=(t_1,t_2,\dotsc,t_{2l})$.
\begin{align}
\frac{\K(t_1,\dots,t_{2l}+1)}{\K(t_1,\dots,t_{2l})}=&\frac{1}{(2\E(\textbf{t}))!(2\s_{2l}(\textbf{t}))!}\prod_{i=1}^{l}\frac{(\s_{2l}(\textbf{t})-\s_{2i-1}(\textbf{t}))!(\s_{2l}(\textbf{t})+\s_{2i}(\textbf{t}))!}{(\s_{2l}(\textbf{t})-\s_{2i}(\textbf{t}))!(\s_{2l}(\textbf{t})+\s_{2i-1}(\textbf{t}))!}.
\end{align}
\end{claim}

This claim implies the following simplification of the $h$-factor
\begin{align}\label{component6}
&\frac{h_{x,y-2,z-1}(a_1,a_2,\dotsc,a_{2l}+1)}{h_{x,y,z-1}(a_1,a_2,\dotsc,a_{2l})}\frac{h_{x,y,z}(a_1,a_2,\dotsc,a_{2l})}{h_{x,y-2,z}(a_1,a_2,\dotsc,a_{2l}+1)}=\notag\\
&\quad\quad\quad\quad=\frac{\K\left(x+\frac{y}{2},a_{2l},\dotsc,a_1,z-1\right)}{\K\left(x+\frac{y}{2},a_{2l},\dotsc,a_1,z\right)}\frac{\K\left(x+\frac{y}{2}-1,a_{2l}+1,\dotsc,a_1,z\right)}{\K\left(x+\frac{y}{2}-1,a_{2l}+1,\dotsc,a_1,z-1\right)}\notag\\
&\quad\quad\quad\quad=\frac{(2x+\s+y+z-1)(\s+z)}{(2\e+2z-1)(2\e+2z)}.
\end{align}
By (\ref{component4}), (\ref{component5}), and (\ref{component6}), the first term on the left-hand side of (\ref{numberrecur1b}) is also reduced to
\begin{align}\label{recur1eq3}
\frac{\Phi_{x,y-2,z-1}(a_1,a_2,\dotsc,a_{2l}+1)}{\Phi_{x,y,z-1}(a_1,a_2,\dotsc,a_{2l})}&\frac{\Phi_{x,y,z}(a_1,a_2,\dotsc,a_{2l})}{\Phi_{x,y-2,z}(a_1,a_2,\dotsc,a_{2l}+1)}
=\frac{2(2x+\s+y+z-1)(2\s+2y+2z-3)}{(2x+2\s+2y+2z-3)_2}.
\end{align}

We repeat the arguments in the case of odd $y$ to the second term on the left-hand side of (\ref{numberrecur1b}). In particular, we write its as the products of the $f$-, $g$-, and $h$-factors. While the $f$-factor can be simplified to
\[\frac{(2x)_{2\s+2y+2z-3}}{(2x+2)_{2\s+2y+2z-3}}=\frac{(2x)(2x+1)}{(2x+2\s+2y+2z-3)_2},\]
the $g$- and $h$-factors are all reducible to $1$. Thus our recurrence (\ref{numberrecur1b}) is also equivalent to
\begin{equation}
\frac{2(2x+\s+y+z-1)(2\s+2y+2z-3)}{(2x+2\s+2y+2z-3)_2}+\frac{(2x)(2x+1)}{(2x+2\s+2y+2z-3)_2}=1,
\end{equation}
which is a true statement.

\bigskip

The rest of the proof is the verification of (\ref{numberrecur3}). We need to show that for $n=2l+1$
\begin{align}\label{numberrecur3b}
&\frac{\Phi_{x+1,y,z}(a_1,a_2,\dotsc,a_{2l+1}-1)\Phi_{x,y,z-1}(a_1,a_2,\dotsc,a_{2l+1})}{\Phi_{x,y,z}(a_1,a_2,\dotsc,a_{2l+1})\Phi_{x+1,y,z-1}(a_1,a_2,\dotsc,a_{2l+1}-1)}\notag\\
&\quad\quad+\frac{\Phi_{x,y+2,z-1}(a_1,a_2,\dotsc,a_{2l+1}-1)\Phi_{x+1,y-2,z}(a_1,a_2,\dotsc,a_{2l+1})}{\Phi_{x,y,z}(a_1,a_2,\dotsc,a_{2l+1})\Phi_{x+1,y,z-1}(a_1,a_2,\dotsc,a_{2l+1}-1)}=1.
\end{align}

\medskip

We consider first the case when $y$ is odd.

To simplify the first term on the left-hand side of (\ref{numberrecur3b}), we also break it down as the product of the $f$-, $h$-, and $g$-factors as
\begin{align}\label{recur3eq1}
&\frac{\Phi_{x+1,y,z}(a_1,a_2,\dotsc,a_{n}-1)\Phi_{x,y,z-1}(a_1,a_2,\dotsc,a_{n})}{\Phi_{x,y,z}(a_1,a_2,\dotsc,a_{n})\Phi_{x+1,y,z-1}(a_1,a_2,\dotsc,a_{n}-1)}=\notag\\
&\quad\quad\quad\frac{f_{x+1,y,z}(a_1,a_2,\dotsc,a_{n}-1)f_{x,y,z-1}(a_1,a_2,\dotsc,a_{n})}{f_{x,y,z}(a_1,a_2,\dotsc,a_{n})f_{x+1,y,z-1}(a_1,a_2,\dotsc,a_{n}-1)}\notag\\
&\quad\quad\quad\times\frac{g_{y,z}(a_1,a_2,\dotsc,a_{n}-1)g_{y,z-1}(a_1,a_2,\dotsc,a_{n})}{g_{y,z}(a_1,a_2,\dotsc,a_{n})g_{y,z-1}(a_1,a_2,\dotsc,a_{n}-1)}\notag\\
&\quad\quad\quad\times\frac{h_{x+1,y,z}(a_1,a_2,\dotsc,a_{n}-1)h_{x,y,z-1}(a_1,a_2,\dotsc,a_{n})}{h_{x,y,z}(a_1,a_2,\dotsc,a_{n})h_{x+1,y,z-1}(a_1,a_2,\dotsc,a_{n}-1)}.
\end{align}
The $f$-factor can be rewritten as
\begin{align}
&\frac{f_{x+1,y,z}(a_1,a_2,\dotsc,a_{n}-1)f_{x,y,z-1}(a_1,a_2,\dotsc,a_{n})}{f_{x,y,z}(a_1,a_2,\dotsc,a_{n})f_{x+1,y,z-1}(a_1,a_2,\dotsc,a_{n}-1)}\notag\\
&=\frac{\prod_{i=1}^{\frac{y-1}{2}}(2x+2i+2)_{2\s+2y+2z-4i-1}\prod_{i=1}^{\frac{y-1}{2}}(2x+2i)_{2\s+2y+2z-4i-1}}
{\prod_{i=1}^{\frac{y-1}{2}}(2x+2i)_{2\s+2y+2z-4i+1}\prod_{i=1}^{\frac{y-1}{2}}(2x+2i+2)_{2\s+2y+2z-4i-3}}\notag\\
&=\prod_{i=1}^{\frac{y-1}{2}}\frac{(2x+2i+2)_{2\s+2y+2z-4i-1}}{(2x+2i+2)_{2\s+2y+2z-4i-3}}\frac{(2x+2i)_{2\s+2y+2z-4i-1}}{(2x+2i)_{2\s+2y+2z-4i+1}}\notag\\
&=\prod_{i=1}^{\frac{y-1}{2}}(2x+2\s+2y+2z-2i-1)_2 \frac{1}{(2x+2\s+2y+2z-2i-1)_2}=1.
\end{align}
The $g$-factor is simplified as
\begin{align}
\frac{g_{y,z}(a_1,a_2,\dotsc,a_{n}-1)g_{y,z-1}(a_1,a_2,\dotsc,a_{n})}{g_{y,z}(a_1,a_2,\dotsc,a_{n})g_{y,z-1}(a_1,a_2,\dotsc,a_{n}-1)}=\frac{(\s+z-1)(2\s+2y+2z-3)_2}{(\s+y+z-2)(2\s+y+2z-2)_2}.
\end{align}
It is also easy to see that the $h$-factor is equal to $1$ in this case. Therefore the first term on the left-hand side of (\ref{numberrecur3b}) is now
\begin{align}
\frac{\Phi_{x+1,y,z}(a_1,a_2,\dotsc,a_{n}-1)\Phi_{x,y,z-1}(a_1,a_2,\dotsc,a_{n})}{\Phi_{x,y,z}(a_1,a_2,\dotsc,a_{n})\Phi_{x+1,y,z-1}(a_1,a_2,\dotsc,a_{n}-1)}
=\frac{2(\s+z-1)(2\s+2y+2z-3)}{(2\s+y+2z-2)_2}.
\end{align}

Arguing similarly for the second term on the left-hand side of (\ref{numberrecur3b}), its $f$- and $h$-factors are both equal to $1$, and its $g$-factor can be reduced to $\frac{y(y-1)}{(
2s+y+2z-2)_2}$. Therefore, (\ref{numberrecur3b}) is equivalent to
\begin{equation}\label{recur3eq2}
\frac{2(\s+z-1)(2\s+2y+2z-3)}{(2\s+y+2z-2)_2}+\frac{y(y-1)}{(
2s+y+2z-2)_2} =1,
\end{equation}
which is obviously true.

Next, we verify (\ref{numberrecur3b}) for even $y$. Similar to the case of odd $y$, the first term on the left-hand side of (\ref{numberrecur3b}) can be broken down as the product of three factors as in (\ref{recur3eq1}). The $f$- and the $h$-factors here are also $1$, while the $g$-factor is reducible to $\frac{2(\s+z-1)(2\s+2y+2z-3)}{(2\s+y+2z-2)_2}$. Thus, the first term is now $\frac{2(\s+z-1)(2\s+2y+2z-3)}{(2\s+y+2z-2)_2}$.  Simplifying similarly the second term on the left-hand side of (\ref{numberrecur3b}), we get $\frac{y(y-1)}{(2\s+y+2z-2)_2}$. It follows that  (\ref{numberrecur3b}) is still equivalent to the true statement (\ref{recur3eq2}).
\end{proof}

Similar to the case of the $\Phi$-functions,  we also define the function $\Phi'_{x,y,z}(\textbf{a})$ to be  the expression on the right-hand side of (\ref{eeb}) or (\ref{oob}) depending on whether $y$ is odd or even.  We also have the following lemma.
\begin{lem}\label{numberrecurb} Assume that $x,y,z,n$ are positive integers and that $\textbf{a}=\{a_i\}_{i=1}^{n}$ is a sequence of positive integers. For $n$ is even, we have
\begin{align}\label{numberrecur3}
\Phi'_{x,y,z}&(a_1,a_2,\dotsc,a_{n})\Phi'_{x,y-2,z-1}(a_1,a_2,\dotsc,a_{n}+1)=\notag\\&\Phi'_{x,y,z-1}(a_1,a_2,\dotsc,a_{n})\Phi'_{x,y-2,z}(a_1,a_2,\dotsc,a_{n}+1)\notag\\
&+\Phi'_{x+1,y-2,z}(a_1,a_2,\dotsc,a_{n})\Phi'_{x-1,y,z-1}(a_1,a_2,\dotsc,a_{n}+1).
\end{align}
For odd $n$
\begin{align}\label{numberrecur4}
\Phi'_{x,y,z}&(a_1,a_2,\dotsc,a_{n})\Phi'_{x+1,y,z-1}(a_1,a_2,\dotsc,a_{n}-1))=\notag\\&\Phi'_{x+1,y,z}(a_1,a_2,\dotsc,a_{n}-1))\Phi'_{x,y,z-1}(a_1,a_2,\dotsc,a_{n}))\notag\\
&+\Phi'_{x,y+2,z-1}(a_1,a_2,\dotsc,a_{n}-1)\Phi'_{x+1,y-2,z}(a_1,a_2,\dotsc,a_{n}),
\end{align}
where   \[\Phi'_{x,y,z}(a_1,a_2,\dotsc,a_{n-1},0):=\Phi'_{x,y,z}(a_1,a_2,\dotsc,a_{n-1})\] by convention.
\end{lem}
The proof of Lemma \ref{numberrecurb} is essentially the same as that of Lemma \ref{numberrecurb}, and will be omitted.

\section{Concluding remarks}\label{Conclusion}
It appears that the number of tilings of a hexagon with two (\emph{not} necessary symmetric) arrays of triangles removed is also given by a simple product formula. This tiling formula and its $q$-analog will be investigated in a separated paper \cite{LM}.

It would be interesting to find a `direct' combinatorial proof for the identities in Corollaries \ref{coro2} and \ref{coro3}, in the sense that the proof does not require tiling enumeration  of each component region.

Based on our data, a halved hexagon with an array of triangles removed from the \emph{western} side seems to have a nice tiling number (recall that in  this paper we are considering the case when the array of triangles has been removed from the \emph{northeastern} side of the halved hexagon). This will be investigated in a next part of this paper.

\end{document}